\documentclass{article}
\usepackage[english]{babel}
\usepackage[utf8]{inputenc}

\usepackage{microtype}
\usepackage{graphicx}
\usepackage{subfigure}
\usepackage{booktabs} 

\usepackage{amsmath}
\usepackage{amssymb}
\usepackage{mathtools}
\usepackage{amsthm}
\usepackage{algorithm}
\usepackage{algorithmic}
\usepackage{xcolor}
\allowdisplaybreaks[4]

\usepackage{johd}
\usepackage[capitalize,noabbrev]{cleveref}
\usepackage{dsfont}
\theoremstyle{plain}
\newtheorem{theorem}{Theorem}[section]
\newtheorem{problem}{Problem}[section]
\newtheorem{proposition}{Proposition}[section]
\newtheorem{lemma}{Lemma}[section]
\newtheorem{corollary}{Corollary}[section]
\theoremstyle{definition}
\newtheorem{definition}{Definition}[section]
\newtheorem{assumption}{Assumption}[section]
\theoremstyle{remark}
\newtheorem{remark}{Remark}[section]

\newcommand{\upE}{\mathbf{E}}

\newcommand{\upI}{\mathbf{I}}

\newcommand{\upT}{\mathbf{T}}

\newcommand{\calA}{\mathcal{A}}

\newcommand{\calE}{\mathcal{E}}
\newcommand{\calF}{\mathcal{F}}

\newcommand{\calI}{\mathcal{I}}

\newcommand{\calN}{\mathcal{N}}
\newcommand{\calO}{\mathcal{O}}
\newcommand{\calR}{\mathcal{R}}

\newcommand{\bbE}{\mathbb{E}}
\newcommand{\bbN}{\mathbb{N}}
\newcommand{\bbP}{\mathbb{P}}
\newcommand{\bbR}{\mathbb{R}}

\newcommand{\llangle}{\left\langle}
\newcommand{\rrangle}{\right\rangle}

\newcommand{\tF}{\mathrm{F}}

\newcommand{\ticalO}{\widetilde{\mathcal{O}}}
\newcommand{\ticalE}{\widetilde{\mathcal{E}}}

\newcommand{\tisg}{\widetilde{g}}

\newcommand{\Wt}{{W^{(t)}}}
\newcommand{\Gt}{{G^{(t)}}}
\newcommand{\etat}{{\eta^{(t)}}}
\newcommand{\alphat}{{\alpha^{(t)}}}

\newcommand{\Wtn}{{W^{(t+1)}}}

\newcommand{\halfk}{{\frac{k}{2}}}
\newcommand{\indicator}{{\mathds{1}}}
\newcommand{\cst}{{\mathsf{c}}}

\newcommand{\hatalpha}{\widehat{\alpha}}

\newcommand{\hatsW}{\widehat{W}}

\newcommand{\olinesg}{\overline{g}}
\newcommand{\olinesW}{\overline{W}}

\DeclareMathOperator*{\expan}{flat}

\DeclareMathOperator*{\tr}{tr}

\newcommand{\calII}{\mathcal{II}}
\newcommand{\calIII}{\mathcal{III}}

\newcommand{\err}{\text{err}}

\title{Near-Optimal Tensor PCA via Normalized Stochastic Gradient Ascent with Overparameterization}

\author{Shihong Ding$^{1}$
	\quad
	Yihong Gu$^{2}$
	\quad 
	Yuanshi Liu$^{1}$
	\quad Cong Fang$^{1\textsuperscript{†}}$\\
\\
        \small $^{1}$Peking University\quad $^{2}$Harvard Medical School
        \\\\
}
\date{}

\begin{document}
\maketitle

\renewcommand{\thefootnote}{\dag}
\footnotetext{Corresponding author.}

\begin{abstract}

We study the Order-$k$ ($k \geq 4$) spiked tensor model for the tensor principal component analysis (PCA) problem: given $N$ i.i.d. observations of a $k$-th order tensor generated from the model $\mathbf{T} = \lambda \cdot v_*^{\otimes k} + \mathbf{E}$, where $\lambda > 0$ is the signal-to-noise ratio (SNR), $v_*$ is a unit vector, and $\mathbf{E}$ is a random noise tensor, the goal is to recover the planted vector $v_*$. 

We propose a normalized stochastic gradient ascent (NSGA) method with overparameterization for solving the tensor PCA problem. Without any global (or spectral) initialization step, the proposed algorithm successfully recovers the signal $v_*$ when $N\lambda^2 \geq \widetilde{\Omega}(d^{\lceil k/2 \rceil})$, thereby breaking the previous conjecture that (stochastic) gradient methods require at least $\Omega(d^{k-1})$ samples for recovery. For even $k$, the $\widetilde{\Omega}(d^{k/2})$ threshold coincides with the optimal threshold under computational constraints, attained by sum-of-squares relaxations and related algorithms. Theoretical analysis demonstrates that the overparameterized stochastic gradient method not only establishes a significant initial optimization advantage during the early learning phase but also achieves strong generalization guarantees. This work provides the first evidence that overparameterization improves statistical performance relative to exact parameterization that is solved via standard continuous optimization.
\end{abstract}

\section{Introduction}
Tensor PCA aims to discover principled signal from high-dimensional data corrupted by strong random noise, making it a canonical “needle-in-a-haystack” problem and a particular form of high-dimensional denoising. The feasibility of signal recovery in strong-noise regimes is deeply connected to computational hardness, a relationship that has been extensively studied in multiple contexts. A classical example is the spiked tensor model, introduced by \citet{montanari2014statistical}, in which the observed data consists of an unknown rank-one tensor superimposed with a random noise tensor. This model provides a fundamental framework for understanding the trade-off between computational efficiency and statistical power in tensor PCA. Our work builds directly upon this spiked tensor model, which is formally defined as follows:
\begin{problem}\label{PCA-def}[Signal Recovery for Tensor PCA]
	Given $N$ i.i.d. observations $\{\upT^{(t)}=\lambda\cdot v_*^{\otimes k}+\upE^{(t)}\}_{t=1}^N$ where $v_*\in\bbR^d$ is an arbitrary unit vector, $\lambda\gtrsim 1$ is the signal-to-noise ratio (SNR), and $\upE$ $\left(\{\upE^{(t)}\}_{t=1}^N\sim\upE\right)$ is a random noise tensor with zero mean. The goal is to find a  unit vector $v$ such that $\left\|v-v_*\right\|^2\leq o(1)$.
\end{problem} 
For $k \geq 3$, this model exhibits the so-called statistical-to-computational gap. Consider a dataset containing $N$ tensor observations. 
In the regime where $N\lambda^2 \lesssim d$, recovery of the signal vector $v_*$ is information-theoretically impossible: no estimator can achieve $\ell_2$ error that satisfies $\|v-v_*\|^2\leq o(1)$. Within the regime where $d \lesssim N\lambda^2 \ll d^{k/2}$, it is information-theoretically possible to recover the signal vector $v_*$. Beyond such information-theoretical guarantee, significant efforts have been devoted to understanding whether algorithms with computational constraints can achieve the recovery. However, all existing polynomial-time algorithms fail to achieve non-trivial recovery within this regime. Based on current studies \citep{kunisky2019notes, brennan2021statistical, dudejia2020statistical, hopkins2017thepower, brennan2020statistical, zhang2017Tensor}, it is widely conjectured that no polynomial-time algorithm can achieve non-trivial recovery in this regime. In contrast, for $N\lambda^2 \gtrsim d^{k/2}$, efficient polynomial-time algorithms that achieve the desired recovery do exist. 

In the challenging regime where $d \lesssim N\lambda^2 \ll d^{k/2}$, the maximum-likelihood estimator under isotropic Gaussian noise
is able to recover the signal $v_*$, which is the global maximizer of the objective function $\widehat{f}(v) := \langle v^{\otimes k}, \upT \rangle$ over the unit sphere. Due to the non-convexity of $\widehat{f}(v)$, there currently exists no polynomial-time algorithm that can efficiently compute this estimator. 

The best known threshold for signal recovery $v_*$ in terms of the SNR and sample size $N\lambda^2$ is at most $d^{k/2}$. 
These methods generally follow one of two strategies: either strictly controlling the search process or region, or designing a well-constructed initialization that exploits structural properties of the observation tensor's principal components.
Specifically, the sum-of-squares (SoS) methods \citep{hopkins2015tensor} and its related spectral methods \citep{hopkins2016fast} achieve a tight threshold of $\widetilde{\Omega}(d^{k/2})$. 
Additionally, a Gaussian homotopy-based algorithm proposed in \citet{anandkumar2016homotopy} for $k=3$ also achieves a threshold of $\widetilde{\Omega}(d^{k/2})$. For a more detailed discussion of this threshold, we refer the reader to Section \ref{related_works}.

The above work investigates effective signal recovery for Tensor PCA under any polynomial-time algorithm. In recent years, gradient-based (or first-order) methods have emerged as the dominant solvers for modern large-scale problems, owing to their low per-iteration cost, scalability, and black-box nature—which facilitates straightforward implementation using automatic differentiation tools. Consequently, understanding efficient recovery under  the 
computational constrains that are only implemented by 
gradient-based methods  has attracted considerable research interest in recent years \citep{biroli2019how, arous2020algorithmic, arous2020online}. 

The studies on (stochastic) gradient methods for tensor PCA predominantly adopt update rules derived from the gradient of the maximum likelihood estimation (MLE) objective $\widehat{f}(v)$. In the upper bound part,
\citet{arous2020algorithmic} proved that gradient descent and Langevin dynamics can achieve efficient recovery of the signal when $N\lambda^2 \gtrsim d^{k-1}$. Under the same condition, \citet{arous2020online} demonstrated that online SGD also attains strong recovery guarantees. Complementarily, the algorithmic lower bound perspective provides two types of evidence to support the failure of (stochastic) gradient algorithms to recover $v_*$ in the regime $d^{k/2} \lesssim N\lambda^2 \ll d^{k-1}$. The one studies the difficulty of topological complexity of the objective landscape $\widehat{f}(v)$, manifested through the proliferation of spurious critical points near the maximum likelihood potential \citep{arous2017theland, ros2019complex}. The other suggests that the weakness of the signal in the region of maximal entropy for the uninformative prior constitutes the primary cause \citep{arous2020algorithmic}. In particular, \citet{arous2020algorithmic} identified a free energy well around the equator when the square of SNR falls below $\mathcal{O}(d^{k-1})$, from which computational hardness under Gibbs initialization can be derived. It is still open whether gradient-based methods are indeed less efficient than SoS and spectral algorithms, the latter of which achieve the $\widetilde{\Omega}(d^{k/2})$ recovery threshold. 
This motivates the following research question: 
\begin{center}
\it Can gradient-based methods recover $v_*$ in the regime $d^{k/2}\lesssim N\lambda^2\ll d^{k-1}$.
\end{center}

We propose two finite-horizon online normalized stochastic gradient algorithms to recover the signal vector $v_*$, breaking the prior conjecture that (stochastic) gradient methods require at least $\Omega(d^{k-1})$ threshold for recovery. For even-order tensors (i.e., when $k$ is even),  Algorithm \ref{SGA} first performs $T$ iterations of normalized stochastic gradient ascent with a shift term, producing an estimator $W^{(T)}$ of the rank-one matrix $v_*v_*^{\top}$. It then returns the leading eigenvector of the matrix $W^{(T)}+\left(W^{(T)}\right)^{\top}$. For odd-order tensors, Algorithm \ref{k-odd} runs two parallel instances of Algorithm \ref{SGA}. At each iteration $t$, each instance preprocesses the sampled tensor $\upT^{(t)}$ into an even-order tensor without prior information, and the final output is selected probabilistically from the two estimated vectors.

To the best of our knowledge, our algorithm is the first stochastic gradient method that——without any global (or spectral) initialization step——successfully recovers $v_*$ with constant probability when the SNR and sample size satisfy $N\lambda^2 \geq \widetilde{\Omega}(d^{\lceil k/2 \rceil})$. For even $k$, the threshold matches that of state-of-the-art methods \citep{montanari2014statistical,hopkins2015tensor,hopkins2016fast,hopkins2017thepower,anandkumar2016homotopy}. For odd $k$, this threshold requirement can be improved to $\widetilde{\Omega}(d^{k/2})$ by incorporating the partial trace vector as a preprocessed vector. This implies that, with a preprocessing procedure that incorporates global information, our algorithm achieves a near-optimal threshold $\widetilde{\Omega}(d^{k/2})$ of $N\lambda^2$. Furthermore, our algorithm achieves an estimation error of $\ticalO(d^{k/8 + 3/2}/(N\lambda^2))$ for even $k$ and $\ticalO(d^{k/8 + 19/8} / (N\lambda^2))$ for odd $k$. Compared to the $\ticalO(d^{k/4} / \sqrt{N\lambda^2})$ error of existing algorithms \citep{hopkins2015tensor,hopkins2016fast}, our algorithm establishes a \emph{state-of-the-art} convergence rate in tensor PCA.

\begin{figure}[t]
	\centering 
	\includegraphics[width=0.8\textwidth]{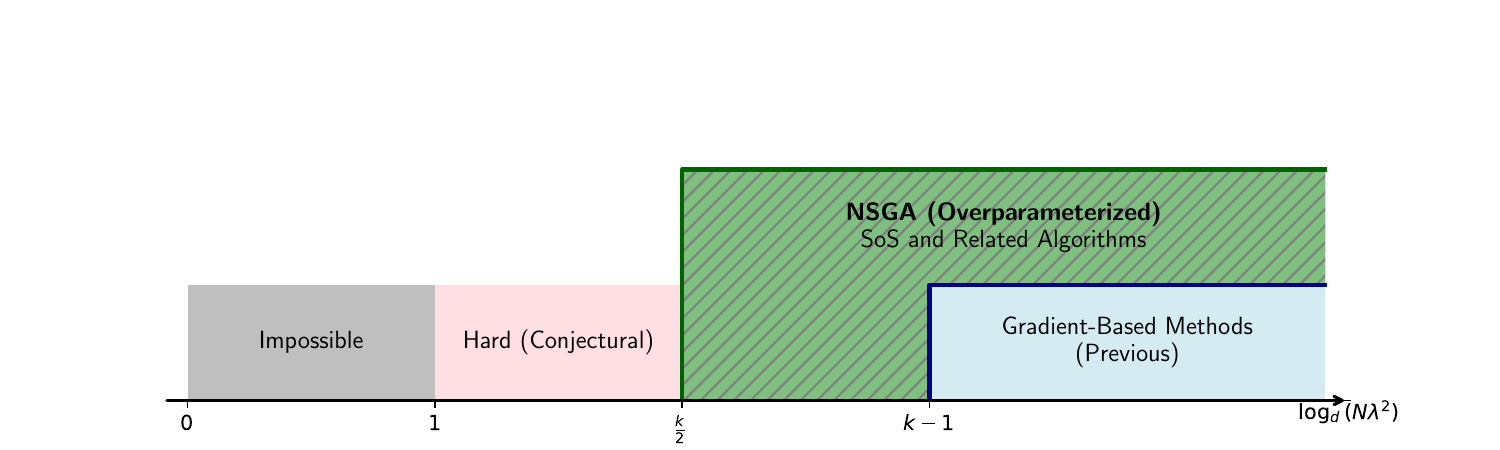} 
	\caption{ The performance of the algorithms for tensor PCA ($k$ even). Combining overparameterization, our algorithm elevates the recovery threshold of gradient-based methods from the light blue region to the green striped region. } 
	\label{fig1} 
\end{figure}

The performance improvement stems from two key insights: \emph{(1) Effective use of randomness}: By introducing a normalized factor, we can identify a suitable matrix-valued reference variable whose dynamics——in the population sense——converge to the rank-$1$ matrix $v_*v_*^{\top}$. Leveraging the sub-Gaussian property of the stochastic noise tensor, we can further control the discrepancy between the algorithmic-iteration dynamics and the population-level dynamics by excluding low-probability events that lead to harmful updates. 
\emph{(2) Overparameterization}: We use a matrix-valued parameterization combined with identity initialization. This overparameterized representation helps avoid trapping in free energy wells near initialization under the maximum likelihood energy landscape, thereby mitigating optimization difficulties. 

Our theoretical results demonstrate that overparameterized stochastic gradient methods not only establish a significant initial advantage during the early optimization phase but also achieve strong generalization guarantees——a finding that may inspire the design of overparameterized solvers in broader machine learning contexts. To the best of our knowledge, this work provides the \emph{first} evidence that the overparameterization  can enhance statistical performance  beyond  what is achievable by exact parameterization that is solved by commonly-used continuous optimization algorithms (see more discussion in section \ref{discussion}).
Moreover, the algorithmic framework may extend to other models with homogeneous structures, such as neural networks with homogeneous activation functions and general tensor decomposition models. With appropriately designed step sizes, the algorithm can also be adapted to infinite-horizon online learning settings.

\noindent\textbf{Our Contributions. } The contributions of this paper are as follows:
\begin{enumerate}
	\item[(1)]  We propose a new normalized stochastic gradient ascent algorithm with overparameterization for solving the tensor PCA problem, which successfully recovers the true signal vector $v_*$ without any global (or spectral) initialization step.
	\item[(2)] We provide a theoretical analysis demonstrating that the proposed algorithm achieves strong recovery guarantees with constant probability when 
	$N\lambda^2\geq\widetilde{\Omega}(d^{\lceil k/2\rceil})$, significantly improving over the previously conjectured threshold of $\widetilde{\Omega}(d^{k-1})$ for (stochastic) gradient methods \citep{arous2020algorithmic, arous2020online}. To the best of our knowledge, this is the first gradient-based method that attains non-trivial recovery at the critical threshold.
\end{enumerate}

\noindent\textbf{Notations. }We denote real vectors by lowercase letters (e.g., $u,v$) and real matrices by uppercase letters (e.g., $Q,W,X$). A vector in $\bbR^d$ is written as $x=\left(x_1,\cdots,x_d\right)$, and a matrix in $\bbR^{d\times d}$ as $W=\left(W_{ij}\right)_{d\times d}$. For $n\in\bbN_+$, $[n]$ represents the set $\{1,\cdots, n\}$. For functions $f,g:\bbR\rightarrow\bbR$, we write $f\lesssim g$ for $f=\calO(g)$, meaning there exists a constant $C$ such that $f\leq C\cdot g$; $f\gtrsim g$ for $f=\Omega(g)$, meaning that there exists a constant $C$ such that $f\geq C\cdot g$; $f\asymp g$ if $g\lesssim f\lesssim g$. We write $f=\ticalO(g)$ if
$f(n)\leq \mathrm{poly}\log(n)\cdot g(n)$, and $f=\widetilde{\Omega}(g)$ if $f\geq \mathrm{poly}\log(n)\cdot g(n)$.
The standard entrywise inner product is denoted $\langle\cdot,\cdot\rangle$. For vectors $u,v\in\bbR^d$, $\langle u,v\rangle=\sum_{i=1}^{d}u_iv_i$. For matrices $Q,W\in\bbR^{d\times d}$, $\llangle Q,W\rrangle=\tr\left(Q^{\top}W\right)$. The $\l_2$-norm of $v\in\bbR^d$ is given by $\|v\|$. The matrix norm used throughout the paper will be the Frobenius norm, denoted by $\left\|W\right\|_{\tF}$ for matrix $W\in\bbR^{d\times d}$.

Tensors of order-$k\geq 4$ are denoted by boldface uppercase letters (e.g., $\upT,\upE\in\otimes^k\bbR^d$). We denote by $\upT(\cdot):\bbR^d\rightarrow\otimes^{k-1}\bbR^d$ the multilinear function such that $\upT(x)=\left(\sum_{i_1=1}^dx_{i_1}\upT_{i_1,i_2,\cdots,i_k}\right)$, applying $x$ to the first modes of the tensor $\upT$. Tensors can be flattened into vectors via the operator $\expan:\otimes^k\bbR^d\rightarrow\bbR^{d^k}$, which reinterprets tensor indices into a lexicographic ordering and has the following form $\expan(\upT)_{(i_1-1)d^{k-1}+(i_2-1)d^{k-2}+\cdots+(i_{k-1}-1)d+i_k}:=\upT_{i_1,\cdots,i_k}$ for any $i_1,\cdots,i_k\in[d]$ under given tensor $\upT\in\otimes^k\bbR^d$. The inner product for tensors is defined as $\langle\upT_1,\upT_2\rangle:=\llangle\expan(\upT_1),\expan(\upT_2)\rrangle$. The $k$-fold outer product of vector $v\in\bbR^d$ is $v^{\otimes k}$.

\section{Related Works}\label{related_works}
\noindent\textbf{Tensor PCA Estimators' Performance: } A lot of work has established that a broad class of algorithms fails to solve the tensor PCA problem within the computationally hard regime ($d \lesssim N\lambda^2 \ll d^{k/2}$). These include SoS relaxations \citep{hopkins2015tensor,hopkins2017thepower,hopkins2018statisticalonfer}, low-degree polynomial estimators \citep{brennan2021statis,kunisky2019notes}, statistical query (SQ) algorithms \citep{dudejia2020statistical,brennan2021statis}, and run-time of memory bounded algorithms \citep{dudeja2024stat}. Furthermore, it has been shown that even Langevin dynamics applied to the maximum likelihood objective cannot efficiently solve tensor PCA within the conjecturally hard regime \citep{arous2020algorithmic}. 
Finally, via average-case reduction, the computational hardness of the hypergraph planted clique problem implies hardness of tensor PCA \citep{brennam2020reduci,zhang2018tensorSVD}.

When the sample size $N$ and the SNR $\lambda$ satisfy $N\lambda^2\geq \widetilde{\Omega}(d^{k/2})$, a variety of methods have been developed for solving the tensor PCA problem. These include SoS relaxations \citep{hopkins2015tensor,hopkins2017thepower}, spectral methods \citep{montanari2014statistical,hopkins2016fast,biroli2019how,zheng2015interpolating}, tensor power methods with global initialization \citep{anandkumar2016homotopy,biroli2019how}, and higher-order generalizations of belief propagation \citep{wein2019thekikuchi}. The best achievable estimation error for these algorithms is $\ticalO(d^{k/4}/\sqrt{N\lambda^2})$.

With theoretical guarantees, gradient-based methods are also effective in recovering the true signal $v_*$ under the regime $N\lambda^2\gtrsim d^{k-1}$ \citep{arous2020algorithmic,arous2020online,montanari2014statistical,huang2022power,wu2024sharp}. Several existing studies have established non-asymptotic, discrete-time convergence guarantees for projected gradient ascent from the perspective of Power Iteration \citep{montanari2014statistical,huang2022power,wu2024sharp}. At the same time, convergence guarantees for (stochastic) gradient ascent are primarily established either in continuous time (via Langevin dynamics) or in an asymptotic sense \citep{arous2020algorithmic,arous2020online}. Moreover, under the condition $N\lambda^2\gtrsim d^{k/2}$, several heuristic gradient algorithms have been empirically shown to perform well under suitable settings \citep{biroli2019how}. Additionally, \citet{arous2024stogra} investigated the high-dimensional dynamics of online stochastic gradient descent with natural random initialization in the multi-spiked tensor model.


\noindent\textbf{Smoothing Methods: }Smoothing methods play a significant role in both theoretical analysis and algorithm design. On the theoretical analysis, \citet{spielman2004smoothed} pioneered the use of smoothed analysis, demonstrating that the shadow-vertex simplex algorithm has polynomial smoothed complexity, thereby providing a theoretical explanation for its efficiency in practice. Building upon the notion of smoothed complexity and the analytical framework introduced by \citet{spielman2004smoothed}, \citet{Arthur2011smoothedkmeans} established that the k-means algorithm admits polynomial smoothed complexity. \citet{chandrasek2024smoothedlearning} proposed a smoothed agnostic learning model for concepts with low intrinsic dimension. In a complementary line of work,  \citet{Bhojanapalli2018smoothed} showed that, under mild conditions, an approximate second-order stationary point is sufficient to guarantee approximate global optimality. 

In algorithm design, smoothing methods also play an essential role. \citet{jin2017howto,jin2018accelerated} introduced small random perturbations during the iterative process to effectively avoid saddle points, emulating the effect of computing gradients after applying localized smoothing to the objective function. \citet{damian2023smoothing} applied stochastic gradient descent (SGD) to a smoothed loss function, 
improving the sample complexity of SGD for single-index models and closing the theoretical gap between gradient-based methods and the correlational SQ lower bound. 

The theoretical analysis in this paper employs a technique of excluding low-probability events to prevent undesirable updates during the iteration process. This approach ensures that the principal component of the observed tensor remains dominant throughout optimization, thereby establishing the convergence results with high probability. Although this method diverges fundamentally from smoothed analysis, both share the common aim of mitigating the impact of worst-case, low-probability outliers on algorithmic convergence. Compared to the smoothing method that requires computing gradients of a smoothed loss function \citep{damian2023smoothing}, the proposed algorithm avoids the complex gradient computations: each iteration relies solely on stochastic gradient information. 

\noindent\textbf{Overparameterized Methods: }In modern machine learning research, overparameterization--the practice of using more parameters than traditionally statistically necessary--is widely employed to improve model training. Although classical statistical theory suggests that overparameterization leads to overfitting, such models often exhibit remarkable generalization performance in practice \citep{hardt2016train,zhang2017understanding}. Overparameterization has been studied across various model classes, including linear models \citep{woodworth2020kernel,haochen2021shape,vaskev2019implicit,ding2025scaling}, matrix factorization models \citep{li2018algorithmic,xiong2023how,jin2023understanding}, and neural networks \citep{zhang2017understanding,Belkin2018reconciling,Li2018learning,Jacot2018Neural,Kaplan2022scaling}.

Existing methods typically achieve overparameterization by increasing the number of parameters without altering their structure. Examples include decomposing each dimension of a linear model into positive and negative parts to reformulate regression via quadratic parameterization \citep{woodworth2020kernel,haochen2021shape,vaskev2019implicit,ding2025scaling}, or using high-rank factorization to reformulate low-rank matrix factorization problems \citep{li2018algorithmic,xiong2023how,jin2023understanding}. To the best of our knowledge, no theoretical evidence currently demonstrates that overparameterization provides a significant statistical advantage. For example, in the case of normal data in the regression problem (or data satisfying the restricted isometry property in the matrix sensing problem), sparse (or low-rank) signal recovery can be achieved using quadratically parameterized models \citep{haochen2021shape, li2018algorithmic} without any explicit regularizer. However, the same recovery guarantees can also be attained through an exact parameterization combined with an $\ell_1$ or nuclear-norm regularizer, which can also be efficiently solved via proximal gradient methods.

In contrast to these approaches, this work investigates overparameterization by modifying the parameterization structure. Specifically, we replace vector parameters with matrix parameters and analyze the convergence behavior of stochastic gradient method under this reformulation. Moreover, through theoretical analysis, we demonstrate that this matrix overparameterization approach effectively prevents the 
MLE objective from being trapped in free energy well during early training stages. 
It provides the first theoretical guarantee of a statistical advantage in overparameterized models.

\section{Problem Formulation}\label{prelimin}
\subsection{Setup and Assumptions}

Suppose we observe i.i.d. observations $\{\mathbf{T}^{(t)}\}_{t=1}^N \sim \bf{T}$, where $\bf{T}$ satisfies
\begin{align}
    \mathbf{T} = \lambda \cdot v_*^{\otimes k} + \mathbf{E}
\end{align} where $\lambda \in \mathbb{R}^+$ represents the known SNR.
$v_* \in \mathbb{R}^d$ with $\|v_*\|_2 = 1$ is the unknown, $\mathbf{E}$ is the random noise tensor. The goal is to estimate $v_*$ based on $N$ i.i.d. observations \citep{montanari2014statistical}.

We first summarize the above data-generating process as a condition and impose some regularity conditions that are widely adopted in the literature.

\begin{assumption}\label{ass-2}
\item[\textbf{[A$_\text{1}$]}] At each iteration step $t$, our algorithm samples a new observation tensor $\upT^{(t)}$ from the data stream. Tensors sampled across different iterations are mutually independent.

\item[\textbf{[A$_\text{2}$]}] There exists $\sigma > 0$ such that the following sub-Gaussian tail bound holds,
\begin{align}
	\bbE\left[\exp\left\{\llangle u, \expan(\upE)\rrangle\right\}\right]&\leq\exp\left\{\sigma^2 \|u\|_2^2 \right\},\quad \forall u\in\bbR^{d^k}.\notag
\end{align} 

\item[\textbf{[A$_\text{3}$]}] The planted vector dimension $d$ and tensor order $k$ satisfy $d\geq k$. The SNR $\lambda$ scales as $\Omega(1)\leq\lambda\leq \mathcal{O}(d^{k/4})$, and the sub-Gaussian parameter $\sigma$ satisfies $\sigma\geq\Omega(1)$.
\end{assumption}
\textbf{[A$_\text{2}$]} implies that the vectorization $\expan(\upE)$ of the zero-mean random noise tensor $\upE$ possesses a \textit{rotationally invariant sub-Gaussian property}. Specifically, for any orthogonal matrix $Q \in \bbR^{d^k\times d^k}$, each coordinate of $Q \expan(\upE)$ is sub-Gaussian with parameter $\sigma$ (Definition \ref{sub-Gaussian}). While standard tensor PCA literature \citep{hopkins2015tensor,hopkins2016fast,arous2020algorithmic,dudeja2024stat} assumes $\upE$ has i.i.d. $\calN(0,\sigma^2)$ entries (which constitutes a special case satisfying Assumption \ref{ass-2}), our framework accommodates broader noise models. The assumption holds in particular when $\expan(\upE)$ satisfies either of the following sufficient conditions:
(1) Its coordinates are mutually independent sub-Gaussian random variables with parameter $\sigma$;
(2) The uniform bound $\left\|\expan(\upE)\right\|\leq\sigma$ holds. Furthermore, \textbf{[A$_\text{3}$]} represents a mild condition. The SNR $\lambda$ obtained from a single observation tensor typically resides in the constant regime ($\lambda\asymp1$). Significantly exceeding the constant SNR level (i.e., attaining $\lambda\gg\Omega(1)$) generally necessitates averaging multiple ($\lambda^2$) independently sampled observation tensors \citep{dudeja2024stat}. For SNR regimes exceeding $\calO(d^{k/4})$, we note that efficient recovery of $v_*$ is already addressed by existing offline algorithms \citep{montanari2014statistical,hopkins2015tensor,hopkins2016fast}. Consequently, these high-SNR cases fall outside the primary scope of our analysis.

We can derive a sharper convergence rate under the isotropic noise.
\begin{assumption}\label{assumption-estimation-second-moment}
Recall the sub-Gaussian parameter $\sigma$ in Assumption \ref{ass-2}. There exists a constant $\mathsf{c}_0$ such that
\begin{align*}
    \mathbb{E}\left[\llangle u, \mathrm{flat}(\mathbf{E})\rrangle^2\right] = \mathsf{c}_0 \sigma^2,
\end{align*}
for any unit vector $u\in\bbR^{d^k}$
\end{assumption}
Assumption \ref{assumption-estimation-second-moment} also holds in the i.i.d. gaussian entries setting \citep{hopkins2015tensor,hopkins2016fast,arous2020algorithmic,dudeja2024stat}. However, this assumption does not compromise the convergence guarantees of Algorithm \ref{SGA} and Algorithm $\ref{k-odd}$. This is because, for any unit vector 
$u\in\bbR^{d^k}$, Assumption \ref{ass-2} guarantees a uniform bound on the second-order moment for $\llangle u,\expan(\upE)\rrangle$, as established in Corollary \ref{corollary-estimation-second-moment}. In Section \ref{tensor}, we further establish the convergence rates for both algorithms even in the absence of Assumption \ref{assumption-estimation-second-moment}.
\begin{corollary}\label{corollary-estimation-second-moment}
	Suppose Assumption \ref{ass-2} holds. For any unit vector $u\in\bbR^{d^k}$, the following second-order moment condition holds for $\llangle u,\expan(\upE)\rrangle$
	\begin{align}
		\bbE\left[\llangle u,\expan(\upE)\rrangle^2\right]\leq4\sigma^2.\notag
	\end{align}
\end{corollary}

\subsection{Our Method}\label{alg}

\begin{algorithm}[t]
	\caption{Normalized Stochastic Gradient Ascent (NSGA) }
	\label{SGA}
	\hspace*{0.02in}{\bf Input: }Initial weight $W^{(0)}=I_d\in\bbR^{d\times d}$, initial step-size $\eta_0$, total sample size $N$, decaying phase length $T_1=\left\lfloor N/\log(N)\right\rfloor$.
	\\
	\hspace*{0.02in}{\bf Output: }$\hat{v}\in\bbR^d$. 
	\begin{algorithmic}[1] 
		\WHILE{$t\leq N$}
		\IF{$t>0$ and $t$ mod $T_1=0$}
		\STATE $\eta_0\leftarrow\eta_0/2$.
		\ENDIF
		\STATE Sample a fresh data $\upT^{(t+1)}$.
		\STATE 
		$$
		W^{(t+1)}\leftarrow \left(1\underbrace{-\frac{\eta_0(k-4)}{2\left\|W^{(t)}\right\|_{\tF}^{k/2}}\widehat{\mathsf{R}}^{(t+1)}\left(W^{(t)}\right)}_{\mathcal{A}}\right)W^{(t)}+\underbrace{\frac{\eta_0}{\left\|W^{(t)}\right\|_{\tF}^{k/2-2}}}_{\mathcal{B}}\nabla_{W}\widehat{\mathsf{R}}^{(t+1)}\left(W^{(t)}\right).
		$$
		\ENDWHILE
		\STATE Let $\hat{v}$ be the top eigenvector of the matrix $W^{(N)} + (W^{(N)})^\top$.  
	\end{algorithmic}
\end{algorithm}
To recover the signal vector $v_*$, we leverage the structure of Problem \ref{PCA-def} to design a specialized reward function, and propose an online algorithm with NSGA updates on an over-parameterized solution matrix $W\in \mathbb{R}^{d\times d}$. At each time-step $t$, the algorithm samples a new observation tensor $\upT^{(t)}$ from the data stream and updates $W$ using its normalized stochastic gradient.

We first present the algorithm when $k$ is even. At each time-step $t \in [N]$, a new observation tensor $\upT^{(t)}$
is sampled from the data stream. The parameter matrix $W$
is then updated via a 
gradient-based algorithm for the following reward function
\begin{align}
\label{eq:loss-even-empirical}
    \widehat{\mathsf{R}}^{(t)}=\widehat{\mathsf{R}}^{(t)}_{\mathrm{even}}(W):=\left\langle W^{\otimes \frac{k}{2}}, \mathbf{T}^{(t)} \right\rangle.
\end{align}
The reward function $\widehat{\mathsf{R}}_{\mathrm{even}}^{(t)}(W)$ is a natural generalization of the MLE objective $\widehat{f}^{(t)}(v):=\llangle v^{\otimes k},\upT^{(t)}\rrangle$ from the vector parameterization to the matrix parameterization. Existing studies demonstrate that overparameterized stochastic gradient methods efficiently approximate optimal solutions, as seen in quadratically parameterized models \citep{woodworth2020kernel,haochen2021shape,ding2025scaling} and matrix factorization models \citep{li2018algorithmic,xiong2023how,jin2023understanding}. Inspired by these approaches, we use $W\in\bbR^{d\times d}$ to approximate $v_*v_*^{\top}$. Since $v_*v_*^{\top}$ has a unit Frobenius norm, we introduce a normalization factor $1/\left\|W\right\|_{\tF}^{k/2-2}$ (term $\mathcal{B}$ in Algorithm \ref{SGA}) to the gradient update term of the reward function during optimization. Additionally, we incorporate a small perturbation with respect to $W^{(t)}$ opposing the direction of the reward function (term $\mathcal{A}$ in Algorithm \ref{SGA}) when obtaining $W^{(t+1)}$.
In fact, the update of our algorithm can also be regarded as  one step of stochastic gradient ascent applied to the normalized reward function  $\widehat{\mathsf{R}}^{(t)}_{\mathrm{even}}(W)/\|W\|_{\tF}^{k/2-2}$.


The algorithm adopts the geometric decay strategy \citep{wu2022last} in the schedule of step-size: it remains constant for the first $T_1=\lfloor N/\log(N)\rfloor$ iterations, then halves every $T_1$ steps thereafter. Combining warm-up initialization with learning rate decay is prevalent in deep learning optimization \citep{goyal2017acc}. Geometric decay strategies are empirically superior to polynomial decay within the decay stage, as they effectively balance aggressive early learning with stable late-stage refinement \citep{ge2019thestep}. Motivated by these advantages, our step-size strategy integrates an initial constant phase with subsequent geometric decay. The full algorithm for even-order ($k$ even) tensors is presented in Algorithm \ref{SGA}. 

For the odd $k$, a preprocessing step is applied to the tensor $\upT$ obtained at each sampling iteration. Given a preprocessed unit vector $u\in\bbR^d$, 
we construct a new preprocessed tensor $\upT(u)=\left(\sum_{i_1=1}^du_{i_1}\upT_{i_1,i_2,\cdots,i_k}\right)\in\otimes^{k-1}\bbR^d$. Depending on the method used to generate $u$, both the critical threshold of $N\lambda^2$ required for signal recovery in Algorithm \ref{k-odd} and its resulting algorithmic classification may vary. In Corollary \ref{prob-odd}, $u$ is obtained by uniform sampling on the unit sphere, resulting in an $\widetilde{\Omega}(d^{\lceil k/2 \rceil})$ threshold. Since $u$ in this case contains no global structural information, Algorithm \ref{k-odd} remains a local optimization method. In Remark \ref{remark-odd}, $u$ is constructed via partial trace computation, leading to an improved $\widetilde{\Omega}(d^{k/2})$ threshold. Here, according to the global information captured by $u$, Algorithm \ref{k-odd} qualifies as a global optimization method. At each time-step $t\in[N]$, a new observation $\upT^{(t)}$ is sampled from the data stream. We consider the associated reward (loss) function:
\begin{align}
	\widehat{\mathsf{R}}^{(t)}_{\mathrm{odd}}(W):=\llangle W^{\otimes\frac{k-1}{2}},\upT^{(t)}(u)\rrangle.
\end{align}
One can notice that the tensor $\upT^{(t)}(u)$ resulting from this preprocessing is an even-order tensor exhibiting a structure analogous to the observed tensor in Problem \ref{PCA-def}--specifically, it comprises a principal component and a random noise matrix. However, the update at time-step $t$ requires concurrent execution of both normalized gradient ascent of $\widehat{\mathsf{R}}^{(t)}_{\mathrm{odd}}$ and $-\widehat{\mathsf{R}}^{(t)}_{\mathrm{odd}}$ due to sign ambiguity in the SNR $\lambda\langle v_*,u\rangle$ of the principal component for $\upT^{(t)}(u)$. Pseudocode for optimizing the weight matrix $W$ when $k$ is odd is provided in Algorithm \ref{k-odd}.
\begin{algorithm}[t]
	\caption{Bi-Threaded  NSGA}
	\label{k-odd}
	\hspace*{0.02in}{\bf Input: }Initial weight $W^{(0)}=I_d\in\bbR^{d\times d}$, initial step-size $\eta_0$, total sample size $N$, decaying phase length $T_1=\left\lfloor N/\log(N)\right\rfloor$,  preprocessed unit vector $u\in\bbR^d$.
	\\
	\hspace*{0.02in}{\bf Output: }$\hat{v}\in\bbR^d$.
	\begin{algorithmic}[1]
		\STATE Execute two parallel instances of the NSGA (Algorithm \ref{SGA}): the first satisfies the update rule as
		$$
		W^{(t+1)}\leftarrow \left(1-\frac{\eta_0(k-5)}{2\left\|W^{(t)}\right\|_{\tF}^{(k-1)/2}}\widehat{\mathsf{R}}^{(t+1)}\left(W^{(t)}\right)\right)W^{(t)}+\frac{\eta_0}{\left\|W^{(t)}\right\|_{\tF}^{(k-1)/2-2}}\nabla_{W}\widehat{\mathsf{R}}^{(t+1)}\left(W^{(t)}\right)
		$$
		with $\widehat{\mathsf{R}}^{(t+1)}=\widehat{\mathsf{R}}^{(t+1)}_{\mathrm{odd}}$ at each time-step $t$, while the second has the same update rule with  $\widehat{\mathsf{R}}^{(t+1)}=-\widehat{\mathsf{R}}^{(t+1)}_{\mathrm{odd}}$ at each time-step $t$.
		\STATE  The first instance yields an output value denoted as $\hat{v}^{(1)}$, while the second instance yields an output value denoted as $\hat{v}^{(2)}$. Randomly pick up $\hat{v}^{(l)}$ from $l\in\{1,2\}$ following the probability $\bbP[l]=0.5$ as $\hat{v}$.
	\end{algorithmic}
\end{algorithm}

\section{Main Result}\label{tensor}

\begin{theorem}\label{main-theorem}
    Consider the tensor PCA problem \ref{PCA-def} with even order $k\geq 4$, and suppose Assumptions \ref{ass-2} and \ref{assumption-estimation-second-moment} hold. For any failing probability $\delta \in (0,1)$, if the sample size $N$ satisfies
    \begin{align}
    \label{eq:cond-n-even-k}
    N\gtrsim\left(\max\left\{\frac{\log(N)}{d^{\frac{k}{4}-1}},\log(d)\right\}\right)^2\cdot\frac{\sigma k\log^{\frac{7}{2}}\left(\frac{\sigma kNd}{\delta}\right)}{\lambda^2}\cdot d^{\frac{k}{2}},
    \end{align} 
    then by picking the initial step-size $\eta_0$ as 
    \begin{align}\label{eq:cond-eta-even-k}
        \eta_0\asymp\max\left\{\frac{\log(N)}{k},\frac{kd^{\frac{k}{4}-1}}{\max\left\{k(k-4),\log^{-1}(d)\right\}}\right\}\cdot\frac{\lceil\log(N)\rceil}{\lambda N},
    \end{align} 
    Algorithm \ref{SGA} can return $\hat{v}$ satisfying
    \begin{align}
		&\min\left\{\left\|\hat{v}-v_*\right\|^2,\left\|\hat{v}+v_*\right\|^2\right\}\lesssim\left(1-\frac{\eta_0\lambda k}{2e}\right)^{\frac{\lfloor N/\log(N)\rfloor}{2}}\frac{1}{k\delta^{\frac{1}{2}}}\notag
        \\
        &\qquad+\frac{\left(\max\left\{\frac{\log(N)}{d^{\frac{k}{4}-1}},\log(d)\right\}\right)^{\frac{1}{2}}\lceil\log(N)\rceil\left(\cst_0\sigma^2+\lambda^2k+\sigma \log\left(\frac{\sigma kNd}{\delta}\right)d^2\right)}{\lambda^2k\delta^{\frac{1}{2}}}\cdot\frac{d^{\frac{k}{8}-\frac{1}{2}}}{N},\notag
	\end{align}
	with probability at least $1-\delta$.
\end{theorem}
For even-order tensor PCA ($k\geq 4$), Theorem \ref{main-theorem} establishes that Algorithm 1 achieves improvements in both sample efficiency and recovery precision (since our algorithm operates in an online manner, sample size is equivalent with total iteration number). By treating $\sigma$ and $k$ as constants, it requires at most $\ticalO\left(d^{k/2}\right)$ critical threshold of $N\lambda^2$ to recover the planted vector $v_*$--matching the best known threshold of computationally efficient offline methods like degree-4 SoS \citep{hopkins2015tensor} and its related spectral methods \citep{hopkins2016fast}. This represents a significant reduction compared to state-of-the-art first-order methods, which require 
$\calO(d^{k-1})$ threshold \citep{arous2020algorithmic,arous2020online,wu2024sharp}. Furthermore, given 
$N$ samples, Algorithm \ref{SGA} achieves a recovery accuracy of $\ticalO\left(d^{k/8+3/2}/(N\lambda^2)\right)$, yielding a strictly faster convergence rate than the $\ticalO\left(d^{k/4}/\sqrt{N\lambda^2}\right)$ accuracy of SoS and its related spectral methods.
\begin{corollary}\label{corollary-no-Ass2-evenk}
    Consider the tensor PCA problem \ref{PCA-def} with even order $k\geq 4$, and suppose Assumption \ref{ass-2} holds. For any $\delta\in (0,1)$, under the same condition for $N$ in Eq.~\eqref{eq:cond-n-even-k} and the choice of $\eta_0$ in Eq.~\eqref{eq:cond-eta-even-k}, Algorithm \ref{SGA} returns an estimator $\hat{v}$ which satisfies
	\begin{align}
		\min\left\{\left\|\hat{v}-v_*\right\|^2,\left\|\hat{v}+v_*\right\|^2\right\}\leq\ticalO\left(\left(1-\frac{\eta_0\lambda k}{2e}\right)^{\frac{\lfloor N/\log(N)\rfloor}{2}}\frac{1}{k\delta^{\frac{1}{2}}}+\frac{\left(\lambda^2k+\sigma d^2\right) d^{\frac{k}{8}-\frac{1}{2}}}{\lambda^2k\delta^{\frac{1}{2}}N}+\frac{\sigma}{\lambda\delta^{\frac{1}{2}}\sqrt{N}}\right),\notag
	\end{align}
	with probability at least $1-\delta$.
\end{corollary}

In complement to Theorem \ref{main-theorem}, we further establish the convergence analysis of Algorithm \ref{k-odd} for odd-order tensor PCA problems with $k\geq5$.
\begin{theorem}\label{sub-theorem}
	Consider the tensor PCA problem \ref{PCA-def} with odd order $k\geq 5$, and suppose Assumptions \ref{ass-2} and \ref{assumption-estimation-second-moment} hold. For any failing probability $\delta \in (0,1)$ and initial preprocessed unit $u\in \mathbb{R}^{d}$, if the sample size $N$ satisfies
	\begin{align}\label{eq:cond-N-odd-k}
        N\gtrsim\left(\max\left\{\frac{\log(N)}{d^{\frac{k-1}{4}-1}},\log(d)\right\}\right)^2\cdot\frac{\sigma k\log^{\frac{7}{2}}\left(\frac{\sigma kNd}{\delta}\right)}{\lambda^2\left|\langle v_*,u\rangle\right|^2}\cdot d^{\frac{k-1}{2}},
	\end{align}
    then by picking the initial step-size $\eta_0$ as
    \begin{align}\label{eq:cond-eta-odd-k}
        \eta_0\asymp\max\left\{\frac{\log(N)}{k-1},\frac{(k-1)d^{\frac{k-1}{4}-1}}{\max\left\{(k-1)(k-5),\log^{-1}(d)\right\}}\right\}\cdot\frac{\lceil\log(N)\rceil}{\lambda\left|\langle v_*,u\rangle\right| N},
    \end{align}
    $\hat{v}^{(1)}$ and $\hat{v}^{(2)}$ generated by Algorithm \ref{k-odd} satisfy
	\begin{align}
		&\min_{l\in\{1,2\}}\left\{\min\left\{\left\|\hat{v}^{(l)}-v_*\right\|^2,\left\|\hat{v}^{(l)}+v_*\right\|^2\right\}\right\}\lesssim\left(1-\frac{\eta_0\lambda\left|\langle v_*,u\rangle\right| (k-1)}{2e}\right)^{\frac{\lfloor N/\log(N)\rfloor}{2}}\frac{1}{k\delta^{\frac{1}{2}}}\notag
        \\
        &\qquad+\frac{\left(\max\left\{\frac{\log(N)}{d^{\frac{k-1}{4}-1}},\log(d)\right\}\right)^{\frac{1}{2}}\lceil\log(N)\rceil\left(\cst_0\sigma^2+\lambda^2\left|\langle v_*,u\rangle\right|^2k+\sigma \log\left(\frac{\sigma kNd}{\delta}\right)d^2\right)}{\lambda^2\left|\langle v_*,u\rangle\right|^2k\delta^{\frac{1}{2}}}\cdot\frac{d^{\frac{k}{8}-\frac{5}{8}}}{N},\notag
	\end{align}
	with probability at least $1-\delta$.
\end{theorem} 
\begin{proof}
	For odd $k$, notice that $\upT$ can be decomposed into $d$ slices $\left\{\upT_i\right\}_{i=1}^d$, where each slice  $\upT_i\in\otimes^{k-1}\bbR^d$ is a sub-tensor satisfying
	\begin{align}
		\upT_i=\lambda (v_*)_i\cdot v_*^{\otimes(k-1)}+\upE_i,\notag
	\end{align}
	The collection of random tensor slices $\left\{\upE_i\right\}_{i=1}^d$ can be concatenated to form the noise tensor $\upE$. Consequently, the preprocessed tensor $\upT(u)$ obtained in Algorithm \ref{k-odd} (where $u$ is the unit preprocessing vector) admits the equivalent expression:
	\begin{align}
		\upT(u) = \sum_{i=1}^d u_i \cdot \upT_i=\lambda\langle v_*,u\rangle v_*^{\otimes (k-1)}+\upE(u). \notag
	\end{align}
	Observing that $\upT(u)$ is an even-order tensor with a SNR of $\lambda\langle v_*, u\rangle$, and that the random noise tensor $\upE(u)$ satisfies Assumption \ref{ass-2} (as established by $u$ being a unit vector and $\upE$ satisfying Assumption \ref{ass-2}), it follows that $\upT({u})$ satisfies the conditions of Theorem \ref{main-theorem}. Therefore, direct application of Theorem \ref{main-theorem} completes the proof of Theorem \ref{sub-theorem}.
\end{proof}
\begin{corollary}\label{coro-odd}
	Consider the tensor PCA problem \ref{PCA-def} with odd order $k\geq 5$, and suppose Assumptions \ref{ass-2} hold. For any $\delta\in(0,1)$, under the same condition for $N$ in Eq.~\eqref{eq:cond-N-odd-k} and the choice of $\eta_0$ in Eq.~\eqref{eq:cond-eta-odd-k}, Algorithm \ref{k-odd} return the estimators $\left\{\hat{v}^{(1)},\hat{v}^{(2)}\right\}$ which satisfy
	\begin{equation}
		\begin{aligned}
			&\min_{l\in\{1,2\}}\left\{\min\left\{\left\|\hat{v}^{(l)}-v_*\right\|^2,\left\|\hat{v}^{(l)}+v_*\right\|^2\right\}\right\}
			\\
			\leq&\ticalO\left(\left(1-\frac{\eta_0\lambda\left|\langle v_*,u\rangle\right| (k-1)}{2e}\right)^{\frac{\lfloor N/\log(N)\rfloor}{2}}\frac{1}{k\delta^{\frac{1}{2}}}+\frac{\left(\lambda^2\left|\llangle v_*,u\rrangle\right|^2k+\sigma d^2\right)d^{\frac{k}{8}-\frac{5}{8}}}{\lambda^2\left|\llangle v_*,u\rrangle\right|^2k\delta^{\frac{1}{2}}N}+\frac{\sigma}{\lambda\left|\llangle v_*,u\rrangle\right|\delta^{\frac{1}{2}}\sqrt{N}}\right),\notag
		\end{aligned}
	\end{equation}
	with probability at least $1-\delta$.
\end{corollary}
\begin{corollary}\label{one-output-odd}
	Consider the tensor PCA problem \ref{PCA-def} with odd order $k\geq 5$. The output of Algorithm \ref{k-odd} may be refined through an estimator that selects the optimal candidate from $\left\{\pm \hat{v}^{(1)},\pm \hat{v}^{(2)}\right\}$. Given a total sample size $N$, we allocate the first $N/2$ observed tensors to execute the algorithm. The estimator $\upT_{\mathrm{esti}}$ is then constructed from the remaining $N/2$ tensors: $\upT_{\mathrm{esti}}=\frac{2}{N}\sum_{i=N/2+1}^{N}\upT^{(i)}$. The optimal vector is selected as the element in $\left\{\pm \hat{v}^{(1)}, \pm \hat{v}^{(2)}\right\}$ maximizing the inner product $\llangle v^{\otimes k},\upT_{\mathrm{esti}}\rrangle$. This vector, designated $\hat{v}$, serves as the final output of Algorithm \ref{k-odd}. Suppose Assumptions \ref{ass-2} and \ref{assumption-estimation-second-moment} hold. Under the same condition for $N$ in Eq.~\eqref{eq:cond-N-odd-k} and the choice of $\eta_0$ in Eq.~\eqref{eq:cond-eta-odd-k}, treating $k$ and $\mathsf{c}_0$ as constants, the resulting error between $\hat{v}$ and the planted vector $v_*$ satisfies
	\begin{align}
		\left\|\hat{v}-v_*\right\|^2\leq\ticalO\left(\frac{\sigma d^{\frac{k}{8}+\frac{11}{8}}}{\lambda^2\left|\langle v_*,u\rangle\right|^2\delta^{\frac{1}{2}}N}+\frac{\sigma}{\lambda\sqrt{N}}\right),\notag
	\end{align}
	with probability at least $1-\delta$ for any $\delta\in(0,1)$.
\end{corollary}
\begin{corollary}\label{prob-odd}
	Consider the tensor PCA problem \ref{PCA-def} with odd order $k\geq 5$. Suppose Assumptions \ref{ass-2} and \ref{assumption-estimation-second-moment} hold, and the preprocessed vector $u$ is generated by uniform sampling on the unit sphere in $\bbR^d$. For any $\delta\in(0,1/2)$, we define a hyper-parameter $\tau\in\left(0, \sqrt{d} t_{d-1,(1+\delta)/2}/\left(\sqrt{d}+t_{d-1,(1+\delta)/2}\right)\right]$ where $t_{d-1,(1+\delta)/2}$ denotes the $(1+\delta)/2$-quantile of a t-distribution with $d-1$ degrees of freedom. 
    Under the same choice of $\eta_0$ in Eq.~\eqref{eq:cond-eta-odd-k} but replaces $|\langle v_*,u\rangle|$ with $\tau d^{-1/2}$, then with a sample size $N\lambda^2\geq\widetilde{\Omega}\left(\sigma d^{\frac{k+1}{2}}/\tau^2\right)$, the parameter set $\left\{\pm \hat{v}^{(1)},\pm \hat{v}^{(2)}\right\}$ obtained by Algorithm \ref{k-odd} satisfies the following inequality 
	\begin{align}
		\min_{l\in\{1,2\}}\left\{\min\left\{\left\|\hat{v}^{(l)}-v_*\right\|^2,\left\|\hat{v}^{(l)}+v_*\right\|^2\right\}\right\}\leq\ticalO\left(\frac{\sigma d^{\frac{k}{8}+\frac{19}{8}}}{\lambda^2\tau^2\delta^{\frac{1}{2}}N}\right),\notag
	\end{align}
	with probability at least $1-2\delta$. This holds when treating $k$ as constant. 
    Corollary \ref{prob-odd} provides a quantitative probability for the initialization of Theorem \ref{sub-theorem}. The initialization probability quantification method--substituting $\tau d^{-1/2}$ for $|\langle v_*,u\rangle|$--can be directly applied to Corollary \ref{coro-odd} and Corollary \ref{one-output-odd}.
\end{corollary}
\begin{remark}\label{remark-odd}
	Consider the tensor PCA problem \ref{PCA-def} with odd order $k\geq 5$. If each entry of the random noise tensor $\upE$ is drawn i.i.d from $\mathcal{N}(0, \sigma^2)$, and the preprocessed vector $u$ in Algorithm \ref{k-odd} is obtained via a method analogous to the partial trace algorithm for tensor PCA \citep{hopkins2016fast}, then, by Theorem \ref{sub-theorem}, Algorithm \ref{k-odd} can recover the vector $v_*$ with threshold $\widetilde{\Omega}(d^{k/2})$.
	
	Specifically, during the initialization phase, $N_1\lambda^2 \gtrsim \widetilde{\Omega}(d^{k/2})$ observed tensors are sampled, and their average is computed as:
	$$
	\upT_{N_1}=\lambda v_*^{\otimes k}+\frac{1}{\sqrt{N_1}}\upE,
	$$
	From this, the following preprocessed partial trace vector $u$ is constructed:
	$$
	u=\frac{\upT_{N_1}\left(\upI_d^{\otimes\frac{k-1}{2}}\right)}{\left\|\upT_{N_1}\left(\upI_d^{\otimes\frac{k-1}{2}}\right)\right\|},
	$$
	This preprocessed vector satisfies $|\langle u, v_*\rangle| \gtrsim d^{-1/4}$ with high probability. Therefore, following the proof technique of Corollary \ref{prob-odd}, one may directly substitute $|\langle u, v_*\rangle|$ with $d^{-1/4}$ in both the parameter settings and the conclusion of Theorem \ref{sub-theorem}. Consequently, treating $k$ as constant, and provided that the number of samples used in the iterative phase satisfies $N_2\lambda^2 \geq \ticalO(d^{k/2})$, the algorithm recovers $v_*$ with an accuracy of $\ticalO(\sigma d^{k/8 + 15/8} / N_2)$.
\end{remark}

\section{Proof Sketch and Discussions}\label{discussion}
To establish Theorem \ref{main-theorem}, we focus on a sequence of high-probability events (the events $\left\{\calA^{(t+1)}(\delta)\right\}_{t=0}^{T-1}$ detailed in Lemma \ref{lemma:high-prob}) by discarding a series of low-probability events which consist of a set of failure scenarios. On the high-probability events, the stochastic noise tensor remains bounded throughout the iterative process. The proof proceeds in two distinct phases: 

\noindent\textbf{Phase 1 (Alignment)}: We demonstrate that SGD drives the principal component of $W^{(t)}/\left\|W^{(t)}\right\|_{\tF}$ towards the target matrix $v_*v_*^{\top}$, i.e., achieves the alignment of the principal component (see Theorem \ref{thm-phase-I-tensor-PCA}). This phase centers on analyzing the trajectory of the reference variable $\alpha^{(t)}:=\llangle v_*, W^{(t)}v_*\rrangle/\left\|W^{(t)}\right\|_{\tF}$. The analysis unfolds in two parts: (a) We establish a uniform high-probability lower bound for $\alpha^{(t)}$ over the time interval $[T_1]$ (see Lemma \ref{PCA-phase-I-upper-bound}). (b) We prove that $\max_{t\leq T_1}\alpha^{(t)}$ converges to a neighborhood of 1 with high probability (see Lemma \ref{PCA-phase-I-lower-bound}). Consequently, at the end of this phase ($t=T_1$), the lower bound guarantees $\alpha^{(T_1)}\geq 1-\calO(1/k)$ with high probability (see Lemma \ref{PCA-phase-I-lower-bound-last-iterate}). 

\noindent\textbf{Phase 2 (Estimation)}: In this phase, we establish the global convergence of Algorithm \ref{SGA} for reference variable $\alpha^{(T)}$ (see Theorem \ref{phase-II-convergence-thm}). The analysis of Algorithm \ref{SGA}'s iterates can be effectively reduced to studying SGD with geometrically decaying step sizes on a one-dimensional linear regression problem. This phase also consists of two key components: (a) We assert that $\alpha^{(t)}$ remains uniformly lower bounded by $1-3\epsilon/2$ over the subsequent time interval $[T_1:T]$ with high probability (see Lemma \ref{phase-II-high-probability}). (b) We construct an auxiliary sequence $\left\{\beta^{(t)}\right\}_{t=1}^{T-T_1}$ that closely tracks $\left\{\alpha^{(T_1+t)}\right\}_{t=1}^{T-T_1}$ with high probability. The update dynamics of $\beta^{(t)}$ over $[T-T_1]$ are approximated by SGD in a standard linear regression setting. We provide separate bounds for the inherent variance term (see Lemma \ref{esti-V}) and the bias term (see Lemma \ref{esti-B}), enabling a precise characterization of the convergence behavior. 

Finally, leveraging the PCA
of the algorithm's output matrix parameter 
$\frac{W^{(T)}+(W^{(T)})^{\top}}{2\left\|W^{(T)}\right\|_{\tF}}$, we prove that the $\ell_2$--norm error between $\frac{W^{(T)}+(W^{(T)})^{\top}}{2\left\|W^{(T)}\right\|_{\tF}}$'s dominant singular vector and the planted vector $v_*$ is bounded by the error between the reference variable $\alpha^{(T)}$ and $1$.

The two key insights mentioned in the introduction, which are crucial for enhancing algorithmic performance, play the following roles in our theoretical analysis:

\noindent\emph{(1) Effective use of randomness: } Due to the higher-order structure of the signal vector in the tensor PCA problem, the dynamics of the reference variable $\alpha$ in discrete-time SGD exhibit a leading growth term at step $t+1$ of the polynomial form: $\eta_t \left[\alpha^{(t)}\right]^{k/2 - 1}$. Therefore, in the case where $k>4$, it is natural to analyze the dynamics of $\alpha^{-(k/2 - 2)}$ during the SGD iterations, whereas for $k=4$, we focus on the dynamics of $\alpha$. When $k>4$, the dominant decrease term in $\alpha^{-(k/2-2)}$ at step $t+1$ is $\eta_t C_1(\lambda, k)$, where $C_1(\lambda, k)$ is a constant depending on $\lambda$ and $k$. For $k=4$, the dominant increase term in $\log(\alpha)$ at step $t+1$ is $\eta_t C_2(\lambda, k)$, where $ C_2(\lambda, k)$ is another constant depending on $\lambda$ and $k$. The higher-order effects generated by the stochastic noise tensor are scaled by $\eta_t^2$, and are thus dominated by these leading-order decrease (or increase) terms. Moreover, the zero-mean and sub-Gaussian properties of the noise tensor ensure that the first-order stochastic effects can be controlled via martingale concentration inequalities. Consequently, the increase in the reference variable $\alpha^{(t)}$ during emph{Phase I}——or equivalently, the decrease in $\left[\alpha^{(t)}\right]^{-(k/2-2)}$ when $k>4$——is intuitively justified. This further implies that $W_t/\|W_t\|$ converges to the rank-1 matrix $v_*v_*^{\top}$ during \emph{Phase I}.

\noindent\emph{(2) Overparameterization: }Initializing with the identity matrix yields an initial value $\alpha^{(0)} = d^{-1/2}$. In contrast, under the vector parameterization with initialization obtained by uniform sampling on the unit sphere, the corresponding reference variable satisfies $\langle v_*, v_0 \rangle^2 \gtrsim d^{-1}$ with high probability. Elevating the initial scale to $d^{-1/2}$ is crucial for achieving the recovery of $v_*$ with the desired sample complexity of $\ticalO(d^{k/2})$.

\begin{figure}[t]
	\centering 
	\includegraphics[width=1.0\textwidth]{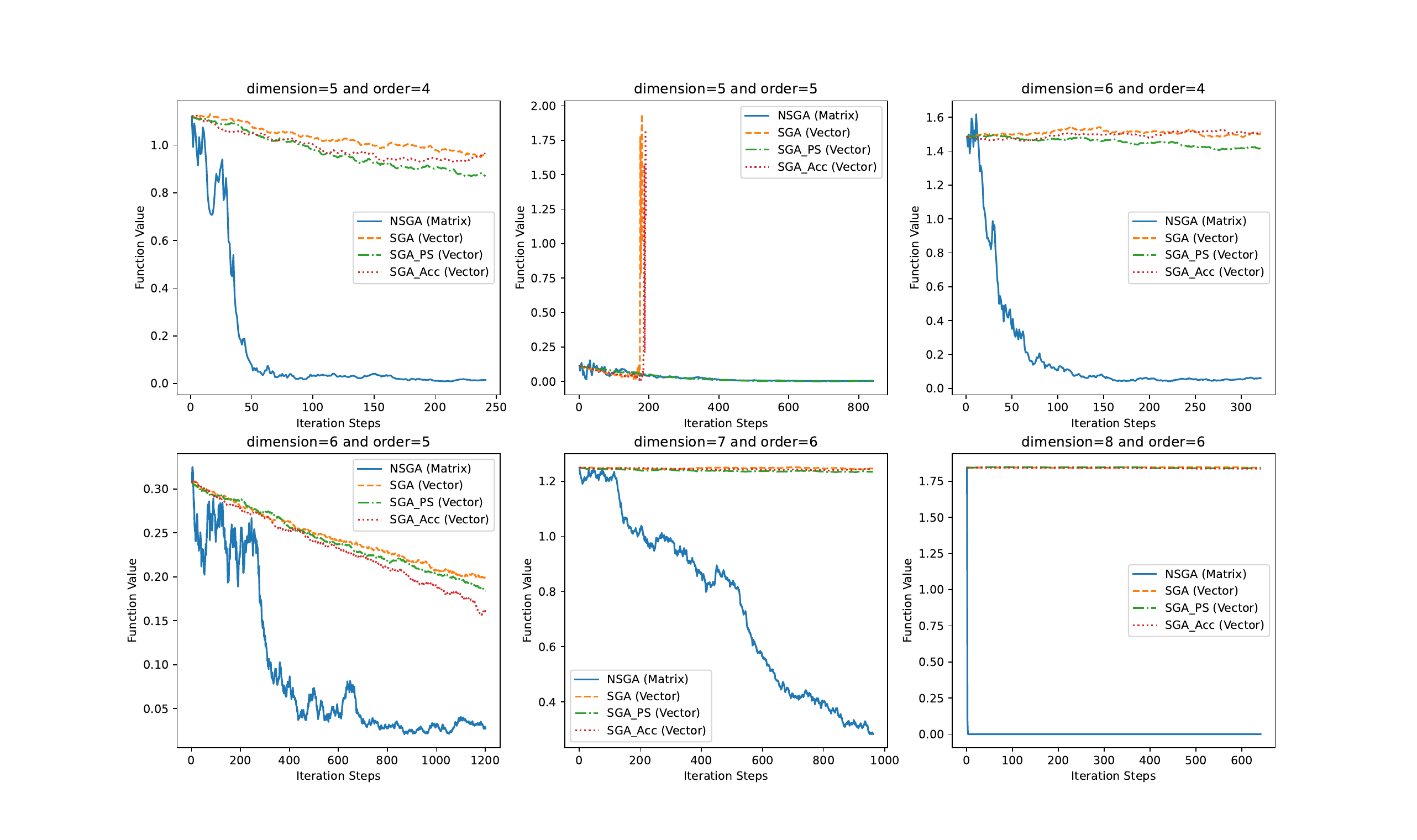} 
	\caption{ The figure shows the convergence behavior among several gradient-based methods, where the horizontal axis represents the number of iterations (or sample size) $N$, and the vertical axis represents the squared estimation error $\left\|\hat{v}^{(N)}-v_*\right\|^2$. The signal vector \(v_*\) is sampled uniformly from the unit sphere. In the comparison, ``NSGA (Matrix)'' denotes our proposed method (corresponding to Algorithm \ref{SGA} when the order \(k\) is even, and Algorithm \ref{k-odd} when \(k\) is odd). ``SGA (Vector)'' represents the standard stochastic gradient ascent under exact parameterization, while ``SGA\_PS (Vector)'' represents the exact parameterization of SGA with projection onto the sphere \citep{arous2020online}. Finally, ``SGA\_Acc (Vector)'' corresponds to the accelerated variant of SGA \citep{lan2020first}.} 
	\label{fig2} 
\end{figure}

\section{Conclusion}
In this work, we introduce NSGA with overparameterization to recover the signal vector $v_*$ for the tensor PCA problem without relying on global or spectral initialization. Our theoretical analysis demonstrates that the proposed method achieves recovery with constant probability when $N\lambda^2\geq\widetilde{\Omega}(d^{\lceil k/2\rceil})$, markedly improving upon the previously conjectured threshold of $\Omega(d^{k-1})$ for gradient-based methods. For even $k$, the $\widetilde{\Omega}(d^{k/2})$ threshold coincides with the optimal threshold under computational constraints, attained by sum-of-squares relaxations and related algorithms. We demonstrate that the overparameterized stochastic gradient method not only establishes a significant initial optimization advantage during the early learning phase but also achieves strong generalization guarantees——a finding that may offer valuable guidance for designing gradient-based methods in other machine learning problems.

\bibliography{bibfile}

\begin{thebibliography}{}

\bibitem [\protect \citeauthoryear {%
Anandkumar%
, Deng%
, Ge%
\BCBL {}\ \BBA {} Mobahi%
}{%
Anandkumar%
\ \protect \BOthers {.}}{%
{\protect \APACyear {2017}}%
}]{%
anandkumar2016homotopy}
\APACinsertmetastar {%
anandkumar2016homotopy}%
\begin{APACrefauthors}%
Anandkumar, A.%
, Deng, Y.%
, Ge, R.%
\BCBL {}\ \BBA {} Mobahi, H.%
\end{APACrefauthors}%
\unskip\
\newblock
\APACrefYearMonthDay{2017}{}{}.
\newblock
{\BBOQ}\APACrefatitle {Homotopy analysis for tensor PCA} {Homotopy analysis for
  tensor pca}.{\BBCQ}
\newblock
\BIn{} \APACrefbtitle {Conference on Learning Theory} {Conference on learning
  theory}\ (\BPGS\ 79--104).
\PrintBackRefs{\CurrentBib}

\bibitem [\protect \citeauthoryear {%
Arous%
, Gerbelot%
\BCBL {}\ \BBA {} Piccolo%
}{%
Arous%
\ \protect \BOthers {.}}{%
{\protect \APACyear {2024}}%
}]{%
arous2024stogra}
\APACinsertmetastar {%
arous2024stogra}%
\begin{APACrefauthors}%
Arous, G\BPBI B.%
, Gerbelot, C.%
\BCBL {}\ \BBA {} Piccolo, V.%
\end{APACrefauthors}%
\unskip\
\newblock
\APACrefYearMonthDay{2024}{}{}.
\newblock
{\BBOQ}\APACrefatitle {Stochastic gradient descent in high dimensions for
  multi-spiked tensor PCA} {Stochastic gradient descent in high dimensions for
  multi-spiked tensor pca}.{\BBCQ}
\newblock
\APACjournalVolNumPages{arXiv preprint arXiv:2410.18162}{}{}{}.
\PrintBackRefs{\CurrentBib}

\bibitem [\protect \citeauthoryear {%
Arous%
, Gheissari%
\BCBL {}\ \BBA {} Jagannath%
}{%
Arous%
\ \protect \BOthers {.}}{%
{\protect \APACyear {2020}}%
}]{%
arous2020algorithmic}
\APACinsertmetastar {%
arous2020algorithmic}%
\begin{APACrefauthors}%
Arous, G\BPBI B.%
, Gheissari, R.%
\BCBL {}\ \BBA {} Jagannath, A.%
\end{APACrefauthors}%
\unskip\
\newblock
\APACrefYearMonthDay{2020}{}{}.
\newblock
{\BBOQ}\APACrefatitle {Algorithmic thresholds for tensor PCA} {Algorithmic
  thresholds for tensor pca}.{\BBCQ}
\newblock
\APACjournalVolNumPages{The Annals of Probability}{48}{4}{2052--2087}.
\PrintBackRefs{\CurrentBib}

\bibitem [\protect \citeauthoryear {%
Arous%
, Gheissari%
\BCBL {}\ \BBA {} Jagannath%
}{%
Arous%
\ \protect \BOthers {.}}{%
{\protect \APACyear {2021}}%
}]{%
arous2020online}
\APACinsertmetastar {%
arous2020online}%
\begin{APACrefauthors}%
Arous, G\BPBI B.%
, Gheissari, R.%
\BCBL {}\ \BBA {} Jagannath, A.%
\end{APACrefauthors}%
\unskip\
\newblock
\APACrefYearMonthDay{2021}{}{}.
\newblock
{\BBOQ}\APACrefatitle {Online stochastic gradient descent on non-convex losses
  from high-dimensional inference} {Online stochastic gradient descent on
  non-convex losses from high-dimensional inference}.{\BBCQ}
\newblock
\APACjournalVolNumPages{Journal of Machine Learning Research}{22}{106}{1--51}.
\PrintBackRefs{\CurrentBib}

\bibitem [\protect \citeauthoryear {%
Arous%
, Mei%
, Montanari%
\BCBL {}\ \BBA {} Nica%
}{%
Arous%
\ \protect \BOthers {.}}{%
{\protect \APACyear {2019}}%
}]{%
arous2017theland}
\APACinsertmetastar {%
arous2017theland}%
\begin{APACrefauthors}%
Arous, G\BPBI B.%
, Mei, S.%
, Montanari, A.%
\BCBL {}\ \BBA {} Nica, M.%
\end{APACrefauthors}%
\unskip\
\newblock
\APACrefYearMonthDay{2019}{}{}.
\newblock
{\BBOQ}\APACrefatitle {The landscape of the spiked tensor model} {The landscape
  of the spiked tensor model}.{\BBCQ}
\newblock
\APACjournalVolNumPages{Communications on Pure and Applied
  Mathematics}{72}{11}{2282--2330}.
\PrintBackRefs{\CurrentBib}

\bibitem [\protect \citeauthoryear {%
Arthur%
, Manthey%
\BCBL {}\ \BBA {} R{\"o}glin%
}{%
Arthur%
\ \protect \BOthers {.}}{%
{\protect \APACyear {2011}}%
}]{%
Arthur2011smoothedkmeans}
\APACinsertmetastar {%
Arthur2011smoothedkmeans}%
\begin{APACrefauthors}%
Arthur, D.%
, Manthey, B.%
\BCBL {}\ \BBA {} R{\"o}glin, H.%
\end{APACrefauthors}%
\unskip\
\newblock
\APACrefYearMonthDay{2011}{}{}.
\newblock
{\BBOQ}\APACrefatitle {Smoothed analysis of the k-means method} {Smoothed
  analysis of the k-means method}.{\BBCQ}
\newblock
\APACjournalVolNumPages{Journal of the ACM (JACM)}{58}{5}{1--31}.
\PrintBackRefs{\CurrentBib}

\bibitem [\protect \citeauthoryear {%
Belkin%
, Hsu%
, Ma%
\BCBL {}\ \BBA {} Mandal%
}{%
Belkin%
\ \protect \BOthers {.}}{%
{\protect \APACyear {2019}}%
}]{%
Belkin2018reconciling}
\APACinsertmetastar {%
Belkin2018reconciling}%
\begin{APACrefauthors}%
Belkin, M.%
, Hsu, D.%
, Ma, S.%
\BCBL {}\ \BBA {} Mandal, S.%
\end{APACrefauthors}%
\unskip\
\newblock
\APACrefYearMonthDay{2019}{}{}.
\newblock
{\BBOQ}\APACrefatitle {Reconciling modern machine-learning practice and the
  classical bias--variance trade-off} {Reconciling modern machine-learning
  practice and the classical bias--variance trade-off}.{\BBCQ}
\newblock
\APACjournalVolNumPages{Proceedings of the National Academy of
  Sciences}{116}{32}{15849--15854}.
\PrintBackRefs{\CurrentBib}

\bibitem [\protect \citeauthoryear {%
Bhojanapalli%
, Boumal%
, Jain%
\BCBL {}\ \BBA {} Netrapalli%
}{%
Bhojanapalli%
\ \protect \BOthers {.}}{%
{\protect \APACyear {2018}}%
}]{%
Bhojanapalli2018smoothed}
\APACinsertmetastar {%
Bhojanapalli2018smoothed}%
\begin{APACrefauthors}%
Bhojanapalli, S.%
, Boumal, N.%
, Jain, P.%
\BCBL {}\ \BBA {} Netrapalli, P.%
\end{APACrefauthors}%
\unskip\
\newblock
\APACrefYearMonthDay{2018}{}{}.
\newblock
{\BBOQ}\APACrefatitle {Smoothed analysis for low-rank solutions to semidefinite
  programs in quadratic penalty form} {Smoothed analysis for low-rank solutions
  to semidefinite programs in quadratic penalty form}.{\BBCQ}
\newblock
\BIn{} \APACrefbtitle {Conference on learning theory} {Conference on learning
  theory}\ (\BPGS\ 3243--3270).
\PrintBackRefs{\CurrentBib}

\bibitem [\protect \citeauthoryear {%
Biroli%
, Cammarota%
\BCBL {}\ \BBA {} Ricci-Tersenghi%
}{%
Biroli%
\ \protect \BOthers {.}}{%
{\protect \APACyear {2020}}%
}]{%
biroli2019how}
\APACinsertmetastar {%
biroli2019how}%
\begin{APACrefauthors}%
Biroli, G.%
, Cammarota, C.%
\BCBL {}\ \BBA {} Ricci-Tersenghi, F.%
\end{APACrefauthors}%
\unskip\
\newblock
\APACrefYearMonthDay{2020}{}{}.
\newblock
{\BBOQ}\APACrefatitle {How to iron out rough landscapes and get optimal
  performances: averaged gradient descent and its application to tensor PCA}
  {How to iron out rough landscapes and get optimal performances: averaged
  gradient descent and its application to tensor pca}.{\BBCQ}
\newblock
\APACjournalVolNumPages{Journal of Physics A: Mathematical and
  Theoretical}{53}{17}{174003}.
\PrintBackRefs{\CurrentBib}

\bibitem [\protect \citeauthoryear {%
M.~Brennan%
\ \BBA {} Bresler%
}{%
M.~Brennan%
\ \BBA {} Bresler%
}{%
{\protect \APACyear {2020}}%
{\protect \APACexlab {{\protect \BCnt {1}}}}}]{%
brennan2020statistical}
\APACinsertmetastar {%
brennan2020statistical}%
\begin{APACrefauthors}%
Brennan, M.%
\BCBT {}\ \BBA {} Bresler, G.%
\end{APACrefauthors}%
\unskip\
\newblock
\APACrefYearMonthDay{2020{\protect \BCnt {1}}}{}{}.
\newblock
{\BBOQ}\APACrefatitle {Reducibility and statistical-computational gaps from
  secret leakage} {Reducibility and statistical-computational gaps from secret
  leakage}.{\BBCQ}
\newblock
\BIn{} \APACrefbtitle {Conference on Learning Theory} {Conference on learning
  theory}\ (\BPGS\ 648--847).
\PrintBackRefs{\CurrentBib}

\bibitem [\protect \citeauthoryear {%
M.~Brennan%
\ \BBA {} Bresler%
}{%
M.~Brennan%
\ \BBA {} Bresler%
}{%
{\protect \APACyear {2020}}%
{\protect \APACexlab {{\protect \BCnt {2}}}}}]{%
brennam2020reduci}
\APACinsertmetastar {%
brennam2020reduci}%
\begin{APACrefauthors}%
Brennan, M.%
\BCBT {}\ \BBA {} Bresler, G.%
\end{APACrefauthors}%
\unskip\
\newblock
\APACrefYearMonthDay{2020{\protect \BCnt {2}}}{}{}.
\newblock
{\BBOQ}\APACrefatitle {Reducibility and statistical-computational gaps from
  secret leakage} {Reducibility and statistical-computational gaps from secret
  leakage}.{\BBCQ}
\newblock
\BIn{} \APACrefbtitle {Conference on Learning Theory} {Conference on learning
  theory}\ (\BPGS\ 648--847).
\PrintBackRefs{\CurrentBib}

\bibitem [\protect \citeauthoryear {%
M\BPBI S.~Brennan%
, Bresler%
, Hopkins%
, Li%
\BCBL {}\ \BBA {} Schramm%
}{%
M\BPBI S.~Brennan%
\ \protect \BOthers {.}}{%
{\protect \APACyear {2021}}%
{\protect \APACexlab {{\protect \BCnt {1}}}}}]{%
brennan2021statistical}
\APACinsertmetastar {%
brennan2021statistical}%
\begin{APACrefauthors}%
Brennan, M\BPBI S.%
, Bresler, G.%
, Hopkins, S.%
, Li, J.%
\BCBL {}\ \BBA {} Schramm, T.%
\end{APACrefauthors}%
\unskip\
\newblock
\APACrefYearMonthDay{2021{\protect \BCnt {1}}}{}{}.
\newblock
{\BBOQ}\APACrefatitle {Statistical Query Algorithms and Low Degree Tests Are
  Almost Equivalent} {Statistical query algorithms and low degree tests are
  almost equivalent}.{\BBCQ}
\newblock
\BIn{} \APACrefbtitle {Conference on Learning Theory} {Conference on learning
  theory}\ (\BPGS\ 774--774).
\PrintBackRefs{\CurrentBib}

\bibitem [\protect \citeauthoryear {%
M\BPBI S.~Brennan%
, Bresler%
, Hopkins%
, Li%
\BCBL {}\ \BBA {} Schramm%
}{%
M\BPBI S.~Brennan%
\ \protect \BOthers {.}}{%
{\protect \APACyear {2021}}%
{\protect \APACexlab {{\protect \BCnt {2}}}}}]{%
brennan2021statis}
\APACinsertmetastar {%
brennan2021statis}%
\begin{APACrefauthors}%
Brennan, M\BPBI S.%
, Bresler, G.%
, Hopkins, S.%
, Li, J.%
\BCBL {}\ \BBA {} Schramm, T.%
\end{APACrefauthors}%
\unskip\
\newblock
\APACrefYearMonthDay{2021{\protect \BCnt {2}}}{}{}.
\newblock
{\BBOQ}\APACrefatitle {Statistical Query Algorithms and Low Degree Tests Are
  Almost Equivalent} {Statistical query algorithms and low degree tests are
  almost equivalent}.{\BBCQ}
\newblock
\BIn{} \APACrefbtitle {Conference on Learning Theory} {Conference on learning
  theory}\ (\BPGS\ 774--774).
\PrintBackRefs{\CurrentBib}

\bibitem [\protect \citeauthoryear {%
Chandrasekaran%
, Klivans%
, Kontonis%
, Meka%
\BCBL {}\ \BBA {} Stavropoulos%
}{%
Chandrasekaran%
\ \protect \BOthers {.}}{%
{\protect \APACyear {2024}}%
}]{%
chandrasek2024smoothedlearning}
\APACinsertmetastar {%
chandrasek2024smoothedlearning}%
\begin{APACrefauthors}%
Chandrasekaran, G.%
, Klivans, A.%
, Kontonis, V.%
, Meka, R.%
\BCBL {}\ \BBA {} Stavropoulos, K.%
\end{APACrefauthors}%
\unskip\
\newblock
\APACrefYearMonthDay{2024}{}{}.
\newblock
{\BBOQ}\APACrefatitle {Smoothed analysis for learning concepts with low
  intrinsic dimension} {Smoothed analysis for learning concepts with low
  intrinsic dimension}.{\BBCQ}
\newblock
\BIn{} \APACrefbtitle {The Thirty Seventh Annual Conference on Learning Theory}
  {The thirty seventh annual conference on learning theory}\ (\BPGS\ 876--922).
\PrintBackRefs{\CurrentBib}

\bibitem [\protect \citeauthoryear {%
Damian%
, Nichani%
, Ge%
\BCBL {}\ \BBA {} Lee%
}{%
Damian%
\ \protect \BOthers {.}}{%
{\protect \APACyear {2023}}%
}]{%
damian2023smoothing}
\APACinsertmetastar {%
damian2023smoothing}%
\begin{APACrefauthors}%
Damian, A.%
, Nichani, E.%
, Ge, R.%
\BCBL {}\ \BBA {} Lee, J\BPBI D.%
\end{APACrefauthors}%
\unskip\
\newblock
\APACrefYearMonthDay{2023}{}{}.
\newblock
{\BBOQ}\APACrefatitle {Smoothing the landscape boosts the signal for sgd:
  Optimal sample complexity for learning single index models} {Smoothing the
  landscape boosts the signal for sgd: Optimal sample complexity for learning
  single index models}.{\BBCQ}
\newblock
\APACjournalVolNumPages{Advances in Neural Information Processing
  Systems}{36}{}{752--784}.
\PrintBackRefs{\CurrentBib}

\bibitem [\protect \citeauthoryear {%
Ding%
, Zhang%
, Zhao%
\BCBL {}\ \BBA {} Fang%
}{%
Ding%
\ \protect \BOthers {.}}{%
{\protect \APACyear {2025}}%
}]{%
ding2025scaling}
\APACinsertmetastar {%
ding2025scaling}%
\begin{APACrefauthors}%
Ding, S.%
, Zhang, H.%
, Zhao, H.%
\BCBL {}\ \BBA {} Fang, C.%
\end{APACrefauthors}%
\unskip\
\newblock
\APACrefYearMonthDay{2025}{}{}.
\newblock
{\BBOQ}\APACrefatitle {Scaling law for stochastic gradient descent in
  quadratically parameterized linear regression} {Scaling law for stochastic
  gradient descent in quadratically parameterized linear regression}.{\BBCQ}
\newblock
\APACjournalVolNumPages{arXiv preprint arXiv:2502.09106}{}{}{}.
\PrintBackRefs{\CurrentBib}

\bibitem [\protect \citeauthoryear {%
Dudeja%
\ \BBA {} Hsu%
}{%
Dudeja%
\ \BBA {} Hsu%
}{%
{\protect \APACyear {2021}}%
}]{%
dudejia2020statistical}
\APACinsertmetastar {%
dudejia2020statistical}%
\begin{APACrefauthors}%
Dudeja, R.%
\BCBT {}\ \BBA {} Hsu, D.%
\end{APACrefauthors}%
\unskip\
\newblock
\APACrefYearMonthDay{2021}{}{}.
\newblock
{\BBOQ}\APACrefatitle {Statistical query lower bounds for tensor pca}
  {Statistical query lower bounds for tensor pca}.{\BBCQ}
\newblock
\APACjournalVolNumPages{Journal of Machine Learning Research}{22}{83}{1--51}.
\PrintBackRefs{\CurrentBib}

\bibitem [\protect \citeauthoryear {%
Dudeja%
\ \BBA {} Hsu%
}{%
Dudeja%
\ \BBA {} Hsu%
}{%
{\protect \APACyear {2024}}%
}]{%
dudeja2024stat}
\APACinsertmetastar {%
dudeja2024stat}%
\begin{APACrefauthors}%
Dudeja, R.%
\BCBT {}\ \BBA {} Hsu, D.%
\end{APACrefauthors}%
\unskip\
\newblock
\APACrefYearMonthDay{2024}{}{}.
\newblock
{\BBOQ}\APACrefatitle {Statistical-computational trade-offs in tensor pca and
  related problems via communication complexity} {Statistical-computational
  trade-offs in tensor pca and related problems via communication
  complexity}.{\BBCQ}
\newblock
\APACjournalVolNumPages{The Annals of Statistics}{52}{1}{131--156}.
\PrintBackRefs{\CurrentBib}

\bibitem [\protect \citeauthoryear {%
Ge%
, Kakade%
, Kidambi%
\BCBL {}\ \BBA {} Netrapalli%
}{%
Ge%
\ \protect \BOthers {.}}{%
{\protect \APACyear {2019}}%
}]{%
ge2019thestep}
\APACinsertmetastar {%
ge2019thestep}%
\begin{APACrefauthors}%
Ge, R.%
, Kakade, S\BPBI M.%
, Kidambi, R.%
\BCBL {}\ \BBA {} Netrapalli, P.%
\end{APACrefauthors}%
\unskip\
\newblock
\APACrefYearMonthDay{2019}{}{}.
\newblock
{\BBOQ}\APACrefatitle {The step decay schedule: A near optimal, geometrically
  decaying learning rate procedure for least squares} {The step decay schedule:
  A near optimal, geometrically decaying learning rate procedure for least
  squares}.{\BBCQ}
\newblock
\APACjournalVolNumPages{Advances in neural information processing
  systems}{32}{}{}.
\PrintBackRefs{\CurrentBib}

\bibitem [\protect \citeauthoryear {%
Goyal%
\ \protect \BOthers {.}}{%
Goyal%
\ \protect \BOthers {.}}{%
{\protect \APACyear {2017}}%
}]{%
goyal2017acc}
\APACinsertmetastar {%
goyal2017acc}%
\begin{APACrefauthors}%
Goyal, P.%
, Doll{\'a}r, P.%
, Girshick, R.%
, Noordhuis, P.%
, Wesolowski, L.%
, Kyrola, A.%
\BDBL {}He, K.%
\end{APACrefauthors}%
\unskip\
\newblock
\APACrefYearMonthDay{2017}{}{}.
\newblock
{\BBOQ}\APACrefatitle {Accurate, large minibatch sgd: Training imagenet in 1
  hour} {Accurate, large minibatch sgd: Training imagenet in 1 hour}.{\BBCQ}
\newblock
\APACjournalVolNumPages{arXiv preprint arXiv:1706.02677}{}{}{}.
\PrintBackRefs{\CurrentBib}

\bibitem [\protect \citeauthoryear {%
HaoChen%
, Wei%
, Lee%
\BCBL {}\ \BBA {} Ma%
}{%
HaoChen%
\ \protect \BOthers {.}}{%
{\protect \APACyear {2021}}%
}]{%
haochen2021shape}
\APACinsertmetastar {%
haochen2021shape}%
\begin{APACrefauthors}%
HaoChen, J\BPBI Z.%
, Wei, C.%
, Lee, J.%
\BCBL {}\ \BBA {} Ma, T.%
\end{APACrefauthors}%
\unskip\
\newblock
\APACrefYearMonthDay{2021}{}{}.
\newblock
{\BBOQ}\APACrefatitle {Shape matters: Understanding the implicit bias of the
  noise covariance} {Shape matters: Understanding the implicit bias of the
  noise covariance}.{\BBCQ}
\newblock
\BIn{} \APACrefbtitle {Conference on Learning Theory} {Conference on learning
  theory}\ (\BPGS\ 2315--2357).
\PrintBackRefs{\CurrentBib}

\bibitem [\protect \citeauthoryear {%
Hardt%
, Recht%
\BCBL {}\ \BBA {} Singer%
}{%
Hardt%
\ \protect \BOthers {.}}{%
{\protect \APACyear {2016}}%
}]{%
hardt2016train}
\APACinsertmetastar {%
hardt2016train}%
\begin{APACrefauthors}%
Hardt, M.%
, Recht, B.%
\BCBL {}\ \BBA {} Singer, Y.%
\end{APACrefauthors}%
\unskip\
\newblock
\APACrefYearMonthDay{2016}{}{}.
\newblock
{\BBOQ}\APACrefatitle {Train faster, generalize better: Stability of stochastic
  gradient descent} {Train faster, generalize better: Stability of stochastic
  gradient descent}.{\BBCQ}
\newblock
\BIn{} \APACrefbtitle {International conference on machine learning}
  {International conference on machine learning}\ (\BPGS\ 1225--1234).
\PrintBackRefs{\CurrentBib}

\bibitem [\protect \citeauthoryear {%
S.~Hopkins%
}{%
S.~Hopkins%
}{%
{\protect \APACyear {2018}}%
}]{%
hopkins2018statisticalonfer}
\APACinsertmetastar {%
hopkins2018statisticalonfer}%
\begin{APACrefauthors}%
Hopkins, S.%
\end{APACrefauthors}%
\unskip\
\newblock
\APACrefYear{2018}.
\newblock
\APACrefbtitle {Statistical inference and the sum of squares method}
  {Statistical inference and the sum of squares method}.
\newblock
\APACaddressPublisher{}{Cornell University}.
\PrintBackRefs{\CurrentBib}

\bibitem [\protect \citeauthoryear {%
S\BPBI B.~Hopkins%
\ \protect \BOthers {.}}{%
S\BPBI B.~Hopkins%
\ \protect \BOthers {.}}{%
{\protect \APACyear {2017}}%
}]{%
hopkins2017thepower}
\APACinsertmetastar {%
hopkins2017thepower}%
\begin{APACrefauthors}%
Hopkins, S\BPBI B.%
, Kothari, P\BPBI K.%
, Potechin, A.%
, Raghavendra, P.%
, Schramm, T.%
\BCBL {}\ \BBA {} Steurer, D.%
\end{APACrefauthors}%
\unskip\
\newblock
\APACrefYearMonthDay{2017}{}{}.
\newblock
{\BBOQ}\APACrefatitle {The power of sum-of-squares for detecting hidden
  structures} {The power of sum-of-squares for detecting hidden
  structures}.{\BBCQ}
\newblock
\BIn{} \APACrefbtitle {2017 IEEE 58th Annual Symposium on Foundations of
  Computer Science (FOCS)} {2017 ieee 58th annual symposium on foundations of
  computer science (focs)}\ (\BPGS\ 720--731).
\PrintBackRefs{\CurrentBib}

\bibitem [\protect \citeauthoryear {%
S\BPBI B.~Hopkins%
, Schramm%
, Shi%
\BCBL {}\ \BBA {} Steurer%
}{%
S\BPBI B.~Hopkins%
\ \protect \BOthers {.}}{%
{\protect \APACyear {2016}}%
}]{%
hopkins2016fast}
\APACinsertmetastar {%
hopkins2016fast}%
\begin{APACrefauthors}%
Hopkins, S\BPBI B.%
, Schramm, T.%
, Shi, J.%
\BCBL {}\ \BBA {} Steurer, D.%
\end{APACrefauthors}%
\unskip\
\newblock
\APACrefYearMonthDay{2016}{}{}.
\newblock
{\BBOQ}\APACrefatitle {Fast spectral algorithms from sum-of-squares proofs:
  tensor decomposition and planted sparse vectors} {Fast spectral algorithms
  from sum-of-squares proofs: tensor decomposition and planted sparse
  vectors}.{\BBCQ}
\newblock
\BIn{} \APACrefbtitle {Proceedings of the forty-eighth annual ACM symposium on
  Theory of Computing} {Proceedings of the forty-eighth annual acm symposium on
  theory of computing}\ (\BPGS\ 178--191).
\PrintBackRefs{\CurrentBib}

\bibitem [\protect \citeauthoryear {%
S\BPBI B.~Hopkins%
, Shi%
\BCBL {}\ \BBA {} Steurer%
}{%
S\BPBI B.~Hopkins%
\ \protect \BOthers {.}}{%
{\protect \APACyear {2015}}%
}]{%
hopkins2015tensor}
\APACinsertmetastar {%
hopkins2015tensor}%
\begin{APACrefauthors}%
Hopkins, S\BPBI B.%
, Shi, J.%
\BCBL {}\ \BBA {} Steurer, D.%
\end{APACrefauthors}%
\unskip\
\newblock
\APACrefYearMonthDay{2015}{}{}.
\newblock
{\BBOQ}\APACrefatitle {Tensor principal component analysis via sum-of-square
  proofs} {Tensor principal component analysis via sum-of-square
  proofs}.{\BBCQ}
\newblock
\BIn{} \APACrefbtitle {Conference on Learning Theory} {Conference on learning
  theory}\ (\BPGS\ 956--1006).
\PrintBackRefs{\CurrentBib}

\bibitem [\protect \citeauthoryear {%
Huang%
, Huang%
, Yang%
\BCBL {}\ \BBA {} Cheng%
}{%
Huang%
\ \protect \BOthers {.}}{%
{\protect \APACyear {2022}}%
}]{%
huang2022power}
\APACinsertmetastar {%
huang2022power}%
\begin{APACrefauthors}%
Huang, J.%
, Huang, D\BPBI Z.%
, Yang, Q.%
\BCBL {}\ \BBA {} Cheng, G.%
\end{APACrefauthors}%
\unskip\
\newblock
\APACrefYearMonthDay{2022}{}{}.
\newblock
{\BBOQ}\APACrefatitle {Power iteration for tensor PCA} {Power iteration for
  tensor pca}.{\BBCQ}
\newblock
\APACjournalVolNumPages{Journal of Machine Learning Research}{23}{128}{1--47}.
\PrintBackRefs{\CurrentBib}

\bibitem [\protect \citeauthoryear {%
Jacot%
, Gabriel%
\BCBL {}\ \BBA {} Hongler%
}{%
Jacot%
\ \protect \BOthers {.}}{%
{\protect \APACyear {2018}}%
}]{%
Jacot2018Neural}
\APACinsertmetastar {%
Jacot2018Neural}%
\begin{APACrefauthors}%
Jacot, A.%
, Gabriel, F.%
\BCBL {}\ \BBA {} Hongler, C.%
\end{APACrefauthors}%
\unskip\
\newblock
\APACrefYearMonthDay{2018}{}{}.
\newblock
{\BBOQ}\APACrefatitle {Neural tangent kernel: Convergence and generalization in
  neural networks} {Neural tangent kernel: Convergence and generalization in
  neural networks}.{\BBCQ}
\newblock
\APACjournalVolNumPages{Advances in neural information processing
  systems}{31}{}{}.
\PrintBackRefs{\CurrentBib}

\bibitem [\protect \citeauthoryear {%
C.~Jin%
, Ge%
, Netrapalli%
, Kakade%
\BCBL {}\ \BBA {} Jordan%
}{%
C.~Jin%
\ \protect \BOthers {.}}{%
{\protect \APACyear {2017}}%
}]{%
jin2017howto}
\APACinsertmetastar {%
jin2017howto}%
\begin{APACrefauthors}%
Jin, C.%
, Ge, R.%
, Netrapalli, P.%
, Kakade, S\BPBI M.%
\BCBL {}\ \BBA {} Jordan, M\BPBI I.%
\end{APACrefauthors}%
\unskip\
\newblock
\APACrefYearMonthDay{2017}{}{}.
\newblock
{\BBOQ}\APACrefatitle {How to escape saddle points efficiently} {How to escape
  saddle points efficiently}.{\BBCQ}
\newblock
\BIn{} \APACrefbtitle {International conference on machine learning}
  {International conference on machine learning}\ (\BPGS\ 1724--1732).
\PrintBackRefs{\CurrentBib}

\bibitem [\protect \citeauthoryear {%
C.~Jin%
, Netrapalli%
\BCBL {}\ \BBA {} Jordan%
}{%
C.~Jin%
\ \protect \BOthers {.}}{%
{\protect \APACyear {2018}}%
}]{%
jin2018accelerated}
\APACinsertmetastar {%
jin2018accelerated}%
\begin{APACrefauthors}%
Jin, C.%
, Netrapalli, P.%
\BCBL {}\ \BBA {} Jordan, M\BPBI I.%
\end{APACrefauthors}%
\unskip\
\newblock
\APACrefYearMonthDay{2018}{}{}.
\newblock
{\BBOQ}\APACrefatitle {Accelerated gradient descent escapes saddle points
  faster than gradient descent} {Accelerated gradient descent escapes saddle
  points faster than gradient descent}.{\BBCQ}
\newblock
\BIn{} \APACrefbtitle {Conference On Learning Theory} {Conference on learning
  theory}\ (\BPGS\ 1042--1085).
\PrintBackRefs{\CurrentBib}

\bibitem [\protect \citeauthoryear {%
J.~Jin%
, Li%
, Lyu%
, Du%
\BCBL {}\ \BBA {} Lee%
}{%
J.~Jin%
\ \protect \BOthers {.}}{%
{\protect \APACyear {2023}}%
}]{%
jin2023understanding}
\APACinsertmetastar {%
jin2023understanding}%
\begin{APACrefauthors}%
Jin, J.%
, Li, Z.%
, Lyu, K.%
, Du, S\BPBI S.%
\BCBL {}\ \BBA {} Lee, J\BPBI D.%
\end{APACrefauthors}%
\unskip\
\newblock
\APACrefYearMonthDay{2023}{}{}.
\newblock
{\BBOQ}\APACrefatitle {Understanding incremental learning of gradient descent:
  A fine-grained analysis of matrix sensing} {Understanding incremental
  learning of gradient descent: A fine-grained analysis of matrix
  sensing}.{\BBCQ}
\newblock
\BIn{} \APACrefbtitle {International Conference on Machine Learning}
  {International conference on machine learning}\ (\BPGS\ 15200--15238).
\PrintBackRefs{\CurrentBib}

\bibitem [\protect \citeauthoryear {%
Kaplan%
\ \protect \BOthers {.}}{%
Kaplan%
\ \protect \BOthers {.}}{%
{\protect \APACyear {2020}}%
}]{%
Kaplan2022scaling}
\APACinsertmetastar {%
Kaplan2022scaling}%
\begin{APACrefauthors}%
Kaplan, J.%
, McCandlish, S.%
, Henighan, T.%
, Brown, T\BPBI B.%
, Chess, B.%
, Child, R.%
\BDBL {}Amodei, D.%
\end{APACrefauthors}%
\unskip\
\newblock
\APACrefYearMonthDay{2020}{}{}.
\newblock
{\BBOQ}\APACrefatitle {Scaling laws for neural language models} {Scaling laws
  for neural language models}.{\BBCQ}
\newblock
\APACjournalVolNumPages{arXiv preprint arXiv:2001.08361}{}{}{}.
\PrintBackRefs{\CurrentBib}

\bibitem [\protect \citeauthoryear {%
Kunisky%
, Wein%
\BCBL {}\ \BBA {} Bandeira%
}{%
Kunisky%
\ \protect \BOthers {.}}{%
{\protect \APACyear {2019}}%
}]{%
kunisky2019notes}
\APACinsertmetastar {%
kunisky2019notes}%
\begin{APACrefauthors}%
Kunisky, D.%
, Wein, A\BPBI S.%
\BCBL {}\ \BBA {} Bandeira, A\BPBI S.%
\end{APACrefauthors}%
\unskip\
\newblock
\APACrefYearMonthDay{2019}{}{}.
\newblock
{\BBOQ}\APACrefatitle {Notes on computational hardness of hypothesis testing:
  Predictions using the low-degree likelihood ratio} {Notes on computational
  hardness of hypothesis testing: Predictions using the low-degree likelihood
  ratio}.{\BBCQ}
\newblock
\BIn{} \APACrefbtitle {ISAAC Congress (International Society for Analysis, its
  Applications and Computation)} {Isaac congress (international society for
  analysis, its applications and computation)}\ (\BPGS\ 1--50).
\PrintBackRefs{\CurrentBib}

\bibitem [\protect \citeauthoryear {%
Lan%
}{%
Lan%
}{%
{\protect \APACyear {2020}}%
}]{%
lan2020first}
\APACinsertmetastar {%
lan2020first}%
\begin{APACrefauthors}%
Lan, G.%
\end{APACrefauthors}%
\unskip\
\newblock
\APACrefYear{2020}.
\newblock
\APACrefbtitle {First-order and stochastic optimization methods for machine
  learning} {First-order and stochastic optimization methods for machine
  learning}\ (\BVOL~1).
\newblock
\APACaddressPublisher{}{Springer}.
\PrintBackRefs{\CurrentBib}

\bibitem [\protect \citeauthoryear {%
Li%
\ \BBA {} Liang%
}{%
Li%
\ \BBA {} Liang%
}{%
{\protect \APACyear {2018}}%
}]{%
Li2018learning}
\APACinsertmetastar {%
Li2018learning}%
\begin{APACrefauthors}%
Li, Y.%
\BCBT {}\ \BBA {} Liang, Y.%
\end{APACrefauthors}%
\unskip\
\newblock
\APACrefYearMonthDay{2018}{}{}.
\newblock
{\BBOQ}\APACrefatitle {Learning overparameterized neural networks via
  stochastic gradient descent on structured data} {Learning overparameterized
  neural networks via stochastic gradient descent on structured data}.{\BBCQ}
\newblock
\APACjournalVolNumPages{Advances in neural information processing
  systems}{31}{}{}.
\PrintBackRefs{\CurrentBib}

\bibitem [\protect \citeauthoryear {%
Li%
, Ma%
\BCBL {}\ \BBA {} Zhang%
}{%
Li%
\ \protect \BOthers {.}}{%
{\protect \APACyear {2018}}%
}]{%
li2018algorithmic}
\APACinsertmetastar {%
li2018algorithmic}%
\begin{APACrefauthors}%
Li, Y.%
, Ma, T.%
\BCBL {}\ \BBA {} Zhang, H.%
\end{APACrefauthors}%
\unskip\
\newblock
\APACrefYearMonthDay{2018}{}{}.
\newblock
{\BBOQ}\APACrefatitle {Algorithmic regularization in over-parameterized matrix
  sensing and neural networks with quadratic activations} {Algorithmic
  regularization in over-parameterized matrix sensing and neural networks with
  quadratic activations}.{\BBCQ}
\newblock
\BIn{} \APACrefbtitle {Conference On Learning Theory} {Conference on learning
  theory}\ (\BPGS\ 2--47).
\PrintBackRefs{\CurrentBib}

\bibitem [\protect \citeauthoryear {%
Montanari%
\ \BBA {} Richard%
}{%
Montanari%
\ \BBA {} Richard%
}{%
{\protect \APACyear {2014}}%
}]{%
montanari2014statistical}
\APACinsertmetastar {%
montanari2014statistical}%
\begin{APACrefauthors}%
Montanari, A.%
\BCBT {}\ \BBA {} Richard, E.%
\end{APACrefauthors}%
\unskip\
\newblock
\APACrefYearMonthDay{2014}{}{}.
\newblock
{\BBOQ}\APACrefatitle {A statistical model for tensor PCA} {A statistical model
  for tensor pca}.{\BBCQ}
\newblock
\APACjournalVolNumPages{Advances in neural information processing
  systems}{27}{}{}.
\PrintBackRefs{\CurrentBib}

\bibitem [\protect \citeauthoryear {%
Ros%
, Ben~Arous%
, Biroli%
\BCBL {}\ \BBA {} Cammarota%
}{%
Ros%
\ \protect \BOthers {.}}{%
{\protect \APACyear {2019}}%
}]{%
ros2019complex}
\APACinsertmetastar {%
ros2019complex}%
\begin{APACrefauthors}%
Ros, V.%
, Ben~Arous, G.%
, Biroli, G.%
\BCBL {}\ \BBA {} Cammarota, C.%
\end{APACrefauthors}%
\unskip\
\newblock
\APACrefYearMonthDay{2019}{}{}.
\newblock
{\BBOQ}\APACrefatitle {Complex energy landscapes in spiked-tensor and simple
  glassy models: Ruggedness, arrangements of local minima, and phase
  transitions} {Complex energy landscapes in spiked-tensor and simple glassy
  models: Ruggedness, arrangements of local minima, and phase
  transitions}.{\BBCQ}
\newblock
\APACjournalVolNumPages{Physical Review X}{9}{1}{011003}.
\PrintBackRefs{\CurrentBib}

\bibitem [\protect \citeauthoryear {%
Spielman%
\ \BBA {} Teng%
}{%
Spielman%
\ \BBA {} Teng%
}{%
{\protect \APACyear {2004}}%
}]{%
spielman2004smoothed}
\APACinsertmetastar {%
spielman2004smoothed}%
\begin{APACrefauthors}%
Spielman, D\BPBI A.%
\BCBT {}\ \BBA {} Teng, S\BHBI H.%
\end{APACrefauthors}%
\unskip\
\newblock
\APACrefYearMonthDay{2004}{}{}.
\newblock
{\BBOQ}\APACrefatitle {Smoothed analysis of algorithms: Why the simplex
  algorithm usually takes polynomial time} {Smoothed analysis of algorithms:
  Why the simplex algorithm usually takes polynomial time}.{\BBCQ}
\newblock
\APACjournalVolNumPages{Journal of the ACM (JACM)}{51}{3}{385--463}.
\PrintBackRefs{\CurrentBib}

\bibitem [\protect \citeauthoryear {%
Vaskevicius%
, Kanade%
\BCBL {}\ \BBA {} Rebeschini%
}{%
Vaskevicius%
\ \protect \BOthers {.}}{%
{\protect \APACyear {2019}}%
}]{%
vaskev2019implicit}
\APACinsertmetastar {%
vaskev2019implicit}%
\begin{APACrefauthors}%
Vaskevicius, T.%
, Kanade, V.%
\BCBL {}\ \BBA {} Rebeschini, P.%
\end{APACrefauthors}%
\unskip\
\newblock
\APACrefYearMonthDay{2019}{}{}.
\newblock
{\BBOQ}\APACrefatitle {Implicit regularization for optimal sparse recovery}
  {Implicit regularization for optimal sparse recovery}.{\BBCQ}
\newblock
\APACjournalVolNumPages{Advances in Neural Information Processing
  Systems}{32}{}{}.
\PrintBackRefs{\CurrentBib}

\bibitem [\protect \citeauthoryear {%
Wainwright%
}{%
Wainwright%
}{%
{\protect \APACyear {2019}}%
}]{%
wainwright2019high}
\APACinsertmetastar {%
wainwright2019high}%
\begin{APACrefauthors}%
Wainwright, M\BPBI J.%
\end{APACrefauthors}%
\unskip\
\newblock
\APACrefYear{2019}.
\newblock
\APACrefbtitle {High-dimensional statistics: A non-asymptotic viewpoint}
  {High-dimensional statistics: A non-asymptotic viewpoint}\ (\BVOL~48).
\newblock
\APACaddressPublisher{}{Cambridge university press}.
\PrintBackRefs{\CurrentBib}

\bibitem [\protect \citeauthoryear {%
Wein%
, El~Alaoui%
\BCBL {}\ \BBA {} Moore%
}{%
Wein%
\ \protect \BOthers {.}}{%
{\protect \APACyear {2019}}%
}]{%
wein2019thekikuchi}
\APACinsertmetastar {%
wein2019thekikuchi}%
\begin{APACrefauthors}%
Wein, A.%
, El~Alaoui, A.%
\BCBL {}\ \BBA {} Moore, C.%
\end{APACrefauthors}%
\unskip\
\newblock
\APACrefYearMonthDay{2019}{}{}.
\newblock
{\BBOQ}\APACrefatitle {The Kikuchi hierarchy and tensor PCA} {The kikuchi
  hierarchy and tensor pca}.{\BBCQ}
\newblock
\APACjournalVolNumPages{Journal of the ACM}{}{}{}.
\PrintBackRefs{\CurrentBib}

\bibitem [\protect \citeauthoryear {%
Woodworth%
\ \protect \BOthers {.}}{%
Woodworth%
\ \protect \BOthers {.}}{%
{\protect \APACyear {2020}}%
}]{%
woodworth2020kernel}
\APACinsertmetastar {%
woodworth2020kernel}%
\begin{APACrefauthors}%
Woodworth, B.%
, Gunasekar, S.%
, Lee, J\BPBI D.%
, Moroshko, E.%
, Savarese, P.%
, Golan, I.%
\BDBL {}Srebro, N.%
\end{APACrefauthors}%
\unskip\
\newblock
\APACrefYearMonthDay{2020}{}{}.
\newblock
{\BBOQ}\APACrefatitle {Kernel and rich regimes in overparametrized models}
  {Kernel and rich regimes in overparametrized models}.{\BBCQ}
\newblock
\BIn{} \APACrefbtitle {Conference on Learning Theory} {Conference on learning
  theory}\ (\BPGS\ 3635--3673).
\PrintBackRefs{\CurrentBib}

\bibitem [\protect \citeauthoryear {%
J.~Wu%
, Zou%
, Braverman%
, Gu%
\BCBL {}\ \BBA {} Kakade%
}{%
J.~Wu%
\ \protect \BOthers {.}}{%
{\protect \APACyear {2022}}%
}]{%
wu2022last}
\APACinsertmetastar {%
wu2022last}%
\begin{APACrefauthors}%
Wu, J.%
, Zou, D.%
, Braverman, V.%
, Gu, Q.%
\BCBL {}\ \BBA {} Kakade, S.%
\end{APACrefauthors}%
\unskip\
\newblock
\APACrefYearMonthDay{2022}{}{}.
\newblock
{\BBOQ}\APACrefatitle {Last iterate risk bounds of sgd with decaying stepsize
  for overparameterized linear regression} {Last iterate risk bounds of sgd
  with decaying stepsize for overparameterized linear regression}.{\BBCQ}
\newblock
\BIn{} \APACrefbtitle {International conference on machine learning}
  {International conference on machine learning}\ (\BPGS\ 24280--24314).
\PrintBackRefs{\CurrentBib}

\bibitem [\protect \citeauthoryear {%
Y.~Wu%
\ \BBA {} Zhou%
}{%
Y.~Wu%
\ \BBA {} Zhou%
}{%
{\protect \APACyear {2024}}%
}]{%
wu2024sharp}
\APACinsertmetastar {%
wu2024sharp}%
\begin{APACrefauthors}%
Wu, Y.%
\BCBT {}\ \BBA {} Zhou, K.%
\end{APACrefauthors}%
\unskip\
\newblock
\APACrefYearMonthDay{2024}{}{}.
\newblock
{\BBOQ}\APACrefatitle {Sharp analysis of power iteration for tensor PCA} {Sharp
  analysis of power iteration for tensor pca}.{\BBCQ}
\newblock
\APACjournalVolNumPages{Journal of Machine Learning Research}{25}{195}{1--42}.
\PrintBackRefs{\CurrentBib}

\bibitem [\protect \citeauthoryear {%
Xiong%
, Ding%
\BCBL {}\ \BBA {} Du%
}{%
Xiong%
\ \protect \BOthers {.}}{%
{\protect \APACyear {{\protect \bibnodate {}}}}%
}]{%
xiong2023how}
\APACinsertmetastar {%
xiong2023how}%
\begin{APACrefauthors}%
Xiong, N.%
, Ding, L.%
\BCBL {}\ \BBA {} Du, S\BPBI S.%
\end{APACrefauthors}%
\unskip\
\newblock
\APACrefYearMonthDay{{\protect \bibnodate {}}}{}{}.
\newblock
{\BBOQ}\APACrefatitle {How Over-Parameterization Slows Down Gradient Descent in
  Matrix Sensing: The Curses of Symmetry and Initialization} {How
  over-parameterization slows down gradient descent in matrix sensing: The
  curses of symmetry and initialization}.{\BBCQ}
\newblock
\BIn{} \APACrefbtitle {The Twelfth International Conference on Learning
  Representations.} {The twelfth international conference on learning
  representations.}
\PrintBackRefs{\CurrentBib}

\bibitem [\protect \citeauthoryear {%
A.~Zhang%
\ \BBA {} Xia%
}{%
A.~Zhang%
\ \BBA {} Xia%
}{%
{\protect \APACyear {2018}}%
{\protect \APACexlab {{\protect \BCnt {1}}}}}]{%
zhang2017Tensor}
\APACinsertmetastar {%
zhang2017Tensor}%
\begin{APACrefauthors}%
Zhang, A.%
\BCBT {}\ \BBA {} Xia, D.%
\end{APACrefauthors}%
\unskip\
\newblock
\APACrefYearMonthDay{2018{\protect \BCnt {1}}}{}{}.
\newblock
{\BBOQ}\APACrefatitle {Tensor SVD: Statistical and computational limits}
  {Tensor svd: Statistical and computational limits}.{\BBCQ}
\newblock
\APACjournalVolNumPages{IEEE Transactions on Information
  Theory}{64}{11}{7311--7338}.
\PrintBackRefs{\CurrentBib}

\bibitem [\protect \citeauthoryear {%
A.~Zhang%
\ \BBA {} Xia%
}{%
A.~Zhang%
\ \BBA {} Xia%
}{%
{\protect \APACyear {2018}}%
{\protect \APACexlab {{\protect \BCnt {2}}}}}]{%
zhang2018tensorSVD}
\APACinsertmetastar {%
zhang2018tensorSVD}%
\begin{APACrefauthors}%
Zhang, A.%
\BCBT {}\ \BBA {} Xia, D.%
\end{APACrefauthors}%
\unskip\
\newblock
\APACrefYearMonthDay{2018{\protect \BCnt {2}}}{}{}.
\newblock
{\BBOQ}\APACrefatitle {Tensor SVD: Statistical and computational limits}
  {Tensor svd: Statistical and computational limits}.{\BBCQ}
\newblock
\APACjournalVolNumPages{IEEE Transactions on Information
  Theory}{64}{11}{7311--7338}.
\PrintBackRefs{\CurrentBib}

\bibitem [\protect \citeauthoryear {%
C.~Zhang%
, Bengio%
, Hardt%
, Recht%
\BCBL {}\ \BBA {} Vinyals%
}{%
C.~Zhang%
\ \protect \BOthers {.}}{%
{\protect \APACyear {2016}}%
}]{%
zhang2017understanding}
\APACinsertmetastar {%
zhang2017understanding}%
\begin{APACrefauthors}%
Zhang, C.%
, Bengio, S.%
, Hardt, M.%
, Recht, B.%
\BCBL {}\ \BBA {} Vinyals, O.%
\end{APACrefauthors}%
\unskip\
\newblock
\APACrefYearMonthDay{2016}{}{}.
\newblock
{\BBOQ}\APACrefatitle {Understanding deep learning requires rethinking
  generalization} {Understanding deep learning requires rethinking
  generalization}.{\BBCQ}
\newblock
\APACjournalVolNumPages{arXiv preprint arXiv:1611.03530}{}{}{}.
\PrintBackRefs{\CurrentBib}

\bibitem [\protect \citeauthoryear {%
Zheng%
\ \BBA {} Tomioka%
}{%
Zheng%
\ \BBA {} Tomioka%
}{%
{\protect \APACyear {2015}}%
}]{%
zheng2015interpolating}
\APACinsertmetastar {%
zheng2015interpolating}%
\begin{APACrefauthors}%
Zheng, Q.%
\BCBT {}\ \BBA {} Tomioka, R.%
\end{APACrefauthors}%
\unskip\
\newblock
\APACrefYearMonthDay{2015}{}{}.
\newblock
{\BBOQ}\APACrefatitle {Interpolating convex and non-convex tensor
  decompositions via the subspace norm} {Interpolating convex and non-convex
  tensor decompositions via the subspace norm}.{\BBCQ}
\newblock
\APACjournalVolNumPages{Advances in Neural Information Processing
  Systems}{28}{}{}.
\PrintBackRefs{\CurrentBib}

\end{thebibliography}
\bibliographystyle{johd}

\appendix

\section{Proof of Theorem \ref{main-theorem}}
\subsection{Preliminaries and Notations}\label{high-probability}

Since Algorithm \ref{SGA} corresponds to SGA on the reward function $\widehat{\mathsf{R}}^{(t)}_{\mathrm{even}}(W)/\|W\|_{\tF}^{k/2-2}$, we begin by analyzing the gradient of this function. Write the gradient at timestep $t$ as 
\begin{equation}\label{eq:gradient}
     \begin{aligned}
     	G^{(t)} :=& \nabla_W \left[\frac{1}{\left\|W\right\|_{\tF}^{\halfk-2}} \left\langle W^{\otimes \halfk}, \lambda v_\star^{\otimes k} + \mathbf{E}^{(t+1)}\right\rangle \right]\Bigg|_{W=\Wt}
     	\\
     	=& \frac{\lambda}{2} \left[  \frac{k\cdot \left(v_\star^\top \Wt v_\star\right)^{\halfk-1} }{\left\|\Wt\right\|_{\tF}^{\halfk-2}}\left(v_\star v_\star^\top\right) - \frac{(k-4)(v_\star^\top \Wt v_\star)^{\halfk}}{\left\|\Wt\right\|_{\tF}^{\halfk}} \Wt\right] +  {E}^{(t+1)},
     \end{aligned}
\end{equation} 
where $E^{(t+1)} \in \mathbb{R}^{d\times d}$ is a matrix dependent on both $\Wt$ and $\mathbf{E}^{(t+1)}$ and satisfies
\begin{equation}\label{def-E}
    \begin{aligned}
    	\left\langle E^{(t+1)}, Q \right\rangle &= \frac{1}{\left\|\Wt\right\|_{\tF}^{\halfk-2}} \left\langle\left.\nabla_W \left\langle W^{\otimes \halfk}, \mathbf{E}^{(t+1)} \right\rangle\right|_{W=\Wt}, Q\right\rangle
    	\\
    	&\qquad \qquad - \frac{(k-4)}{2\left\|\Wt\right\|_{\tF}^{\halfk}} \left\langle \left[\Wt\right]^{\otimes\halfk}, \mathbf{E}^{(t+1)}\right\rangle \left\langle W^{(t)}, Q\right\rangle.
    \end{aligned}
\end{equation}

Observe that for fixed $Q$ and $\Wt$, 
\begin{align*}
     \left\langle \left.\nabla_{W}\left\langle W^{\otimes \halfk}, \mathbf{E}^{(t+1)} \right\rangle\right|_{W=\Wt}, Q\right\rangle &= \frac{d}{d\xi} \left\langle \left[\Wt+\xi Q\right]^{\otimes \halfk}, \mathbf{E}^{(t+1)}\right\rangle\Bigg|_{\xi=0} \\
     &= \left\langle \sum_{l=1}^{\halfk} \left[\Wt\right]^{\otimes (l-1)} \otimes Q \otimes \left[\Wt\right]^{\otimes \halfk-l}, \mathbf{E}^{(t+1)}\right\rangle,
\end{align*} 
hence we can write $\left\langle E^{(t+1)}, Q\right\rangle$ linearly in $\mathbf{E}^{(t+1)}$ as
\begin{align}\label{eq:errort}
\begin{split}
    \left\langle E^{(t+1)}, Q \right\rangle &= \frac{1}{\left\|\Wt\right\|_{\tF}^{\halfk-2}} \sum_{l=1}^{\halfk} \left\langle \left[\Wt\right]^{\otimes (l-1)} \otimes Q \otimes \left[\Wt\right]^{\otimes \halfk-l}, \mathbf{E}^{(t+1)} \right\rangle \\
    &\qquad \qquad - \frac{(k-4)}{2\left\|\Wt\right\|_{\tF}^{\halfk}} \left\langle \left[\Wt\right]^{\otimes\halfk}, \mathbf{E}^{(t+1)}\right\rangle \left\langle W^{(t)}, Q\right\rangle.
\end{split}
\end{align}

We use the following metric to measure how close $W$ is to the unit vector $v$
\begin{align}
\label{eq:alpha}
    \alpha(v, W) := \frac{v ^\top W v}{\|W\|_{\tF}} = \left\langle \frac{W}{\|W\|_{\tF}}, v v^\top \right\rangle,
\end{align} which satisfies $\left|\alpha(v, W)\right|\le 1$. Throughout the proof, we will keep tracking the index \begin{align}
    \alphat := \alpha\left(v_\star, W^{(t)}\right) = \frac{v_\star^\top \Wt v_\star}{\left\|\Wt\right\|_{\tF}}.
\end{align}

We will use the following technical lemma to further represent $\alpha^{(t+1)}$ by $\alphat$ and some negligible error.

\begin{lemma} 
\label{lemma:2nd-expansion}
    Let $W$ and $Q$ be $d\times d$ matrices, $v$ be a $d$-dimensional unit vector, and $\eta>0$. We have
    \begin{align}\label{dynamic-alpha-linear}
        \alpha(v, W + \eta Q) = \alpha(v, W) + \eta \underbrace{\left\{ \frac{v^\top Q v}{\|W\|_{\tF}} - \frac{v^\top W v \times\langle W, Q\rangle}{\|W\|_{\tF}^3}\right\}}_{s(W, Q, v)} + \frac{\eta^2}{2} \Psi_1(W, Q, v, \bar{\eta})
    \end{align} 
    
    where $\bar{\eta} \in [0, \eta]$ and the residual $\Psi_1: \mathbb{R}^{d\times d} \times \mathbb{R}^{d\times d} \times \mathbb{R}^d \times \mathbb{R} \to \mathbb{R}$ is defined as 
    \begin{equation}\label{psi-1}
        \begin{aligned}
        	\Psi_1(W, Q, v, \eta) &= - 2 \frac{(v^\top Q v) \cdot \left(\langle W, Q\rangle + \eta \|Q\|_{\tF}^2\right)}{\|W + \eta Q\|_{\tF}^3} - \frac{(v^\top W v + \eta v^\top Q v) \|Q\|_{\tF}^2}{\|W + \eta Q\|_{\tF}^3} \\
        	&\qquad\qquad + 3 \frac{(v^\top W v + \eta v^\top Q v) \left(\langle W, Q\rangle + \eta \|Q\|_{\tF}^2\right)^2}{\|W + \eta Q\|_{\tF}^5}.
        \end{aligned}
    \end{equation} Moreover, when $k>4$, we further have
    \begin{equation}\label{dynamic-alpha-poly}
    	\begin{aligned}
    		\left[\alpha(v, W + \eta Q)\right]^{-(\halfk-2)} &= [\alpha(v, W)]^{-(\halfk-2)} - \frac{\eta(k-4)}{2} [\alpha(v, W)]^{-(\halfk-1)} s(W, Q, v) \\
    		& \qquad \qquad + \frac{\eta^2(k-4)(k-2)}{8} [\alpha(v, W+\bar{\eta} Q)]^{-\halfk} \Psi_2(W, Q, v, \bar{\eta}) \\
    		& \qquad \qquad - \frac{\eta^2 (k-4)}{4} [\alpha(v, W+\bar{\eta} Q)]^{-(\halfk-1)} \Psi_1(W, Q, v, \bar{\eta}),
    	\end{aligned} 
    \end{equation}
    where $\Psi_2: \mathbb{R}^{d\times d} \times \mathbb{R}^{d\times d} \times \mathbb{R}^d \times \mathbb{R} \to \mathbb{R}$ is defined as 
    \begin{equation}\label{psi-2}
        \Psi_2(W, Q, v, \eta) = \left(\frac{v^\top Q v}{\|W + \eta Q\|_{\tF}} - \frac{\left(v^\top W v + \eta v^\top Q v\right)(\langle W, Q\rangle + \eta \|Q\|_{\tF}^2)}{\|W + \eta Q\|_{\tF}^3}\right)^2.
    \end{equation}
\end{lemma}
Recall $G^{(t)}$ (defined in Eq.~\eqref{eq:gradient}) denotes the stochastic gradient of the risk function $\calR$ with respect to the parameters $W$ at iteration $t$. Leveraging the structural properties of $\Gt$, we analyze the dynamics of the reference variables by bifurcating our analysis into two regimes: $k=4$ and $k>4$.

For $k=4$, one has
\begin{align*}
    s\left(\Wt,G^{(t)},v_*\right)=\frac{\lambda k}{2} \left[\alphat\right]^{\halfk-1} \left\{ 1 - \left[\alphat\right]^2\right\}+ \frac{1}{\left\|\Wt\right\|_{\tF}} \left\langle E^{(t+1)}, v_\star v_\star^\top - \alphat \frac{\Wt}{\left\|\Wt\right\|_{\tF}} \right\rangle
\end{align*} combining the update rule $\Wtn = \Wt + \etat \Gt$ and the 
first part of Lemma \ref{lemma:2nd-expansion} with $v = v_\star$, $W = \Wt$, $Q = \Gt$ and $\eta= \etat$ gives
\begin{align}\label{dynamic-tialpha-original}
\begin{split}
    \alpha^{(t+1)} &= \alphat + \frac{\etat \lambda k}{2} \left[\alphat\right]^{\halfk-1} \left\{ 1 - \left[\alphat\right]^2\right\} \\
    & \qquad \qquad + \frac{\etat}{\left\|\Wt\right\|_{\tF}} \left\langle E^{(t+1)}, v_\star v_\star^\top - \alphat \frac{\Wt}{\left\|\Wt\right\|_{\tF}} \right\rangle + \frac{\left[\etat\right]^2}{2} \Psi_1\left(\Wt, \Gt, v_\star, \bar{\eta}^{(t)}\right).
\end{split}
\end{align} with $\bar{\eta}^{(t)} \in [0, \etat]$ being dependent on $(\Wt, \Gt, \etat)$.

For $k>4$, combining the update rule and the second part of Lemma \ref{lemma:2nd-expansion} yields
\begin{equation}\label{dynamic-tialpha-poly}
	\small
	\begin{aligned}
		\left[\alpha^{(t+1)}\right]^{-(\halfk-2)} &= \left[\alphat\right]^{-(\halfk-2)} - \frac{\etat\lambda k(k-4)}{4}\left\{1-\left[\alphat\right]^2\right\}
		\\
		&\qquad\qquad-\frac{\etat(k-4)}{2}\cdot\left[\alphat\right]^{-\frac{k}{2}+1}\cdot\left[\frac{1}{\left\|\Wt\right\|_{\tF}}\llangle E^{(t+1)},v_*v_*^{\top}\rrangle-\frac{\alphat}{\left\|\Wt\right\|_{\tF}^2}\llangle E^{(t+1)},\Wt\rrangle\right]
		\\
		&\qquad\qquad+\frac{\left[\etat\right]^2(k-4)(k-2)}{8} \left[\alpha\left(v_*, \Wt+\bar{\eta}^{(t)} G^{(t)}\right)\right]^{-\halfk} \Psi_2\left(\Wt, G^{(t)}, v_*, \bar{\eta}^{(t)}\right)
		\\
		&\qquad\qquad-\frac{\left[\etat\right]^2 (k-4)}{4} \left[\alpha\left(v_*, \Wt+\bar{\eta}^{(t)}G^{(t)}\right)\right]^{-(\halfk-1)}\Psi_1\left(\Wt, G^{(t)}, v_*, \bar{\eta}^{(t)}\right)
	\end{aligned}
\end{equation}

We need a high-probability bound of the noise $\mathbf{E}^{(t+1)}$ that is adopted throughout the proof.
\begin{lemma}\label{lemma:high-prob}
    For any $\delta>0$, the following event
    \begin{align*}
        \mathcal{A}^{(t+1)}(\delta) = &\bigcap_{i,j\in [d]}\left\{\left| E^{(t+1)}_{i,j} \right| \le \sqrt{2\cst_4} \left\{\left\|W^{(t)}\right\|_{\tF} + \left|W_{i,j}^{(t)}\right| \right\} \right\}\bigcap\left\{\left|\llangle E^{(t+1)},W^{(t)}\rrangle\right|\leq\sqrt{2\cst_4}\left\|W^{(t)}\right\|_{\tF}^2\right\}
        \\
        & \qquad \qquad \bigcap\left\{\left|\llangle E^{(t+1)},v_*v_*^{\top}\rrangle\right|\leq\sqrt{\cst_4}\left(\left\|W^{(t)}\right\|_{\tF}+\left|\llangle v_*v_*^{\top},W^{(t)}\rrangle\right|\right)\right\}
    \end{align*} satisfies $\mathbb{P}\left[\bigcap_{t=0}^{T-1}\mathcal{A}^{(t+1)}(\delta)\right] \ge 1 - \delta$, where 
    \begin{align}
        \cst_4 = \sigma k\log^{\frac{1}{2}} (kTd^2/\delta).
    \end{align}
\end{lemma}
Furthermore, given the matrix parameter $\Wt$
at iteration $t$, we also need to ensure that, under the truncation of event $\calA^{(t+1)}(\delta)$, the conditional expectation of the inner product $\llangle E^{(t+1)}\cdot\mathds{1}_{\calA^{(t+1)}(\delta)},v_*v_*^{\top}\rrangle$ and $\llangle E^{(t+1)}\cdot\mathds{1}_{\calA^{(t+1)}(\delta)},W^{(t)}\rrangle$ are vanishingly small.
\begin{lemma}\label{control-expectation}
	Given constant $\tau\in\bbR_+$, letting $\delta\lesssim\left(\frac{\tau}{\sigma kd}\right)^4$, we have
	\begin{align}
		&\left|\bbE_t\left[\frac{1}{\left\|W^{(t)}\right\|_{\tF}}\llangle E^{(t+1)}\cdot\mathds{1}_{\calA^{(t+1)}(\delta)},v_*v_*^{\top}\rrangle\right]\right|\leq\tau,\label{tech-lemma-2-1}
		\\
		&\left|\bbE_t\left[\frac{1}{\left\|W^{(t)}\right\|_{\tF}^2}\llangle E^{(t+1)}\cdot\mathds{1}_{\calA^{(t+1)}(\delta)},W^{(t+1)}\rrangle\right]\right|\leq\tau,\label{tech-lemma-2-2}
	\end{align}
\end{lemma}

Our subsequent analysis is conditioned on the event $\mathcal{E}_0 := \bigcap_{t=0}^{T-1}\mathcal{A}^{(t+1)}(\delta)$, which occurs with probability at least $1-\delta/2$ by Lemma \ref{lemma:high-prob}. Within this conditional probability space, the iteration of $\Wt$ proceeds as follows:
\begin{equation}\label{SGD-tensor-PCA-control-sequence}
 	\begin{split}
 		W^{(t+1)}=&\Wt+\frac{\etat\lambda k}{2}\cdot\frac{\langle v_*,\Wt v_*\rangle^{\frac{k}{2}-1}}{\left\|\Wt\right\|_{\text{F}}^{\frac{k}{2}-2}}\cdot v_* v_*^{\top}-\frac{\etat\lambda(k-4)}{2}\cdot\frac{\llangle v_*,\Wt v_*\rrangle^{\frac{k}{2}}}{\left\|\Wt\right\|_{\text{F}}^{\frac{k}{2}}}\cdot\Wt
 		\\
 		&\quad\quad+\etat E^{(t+1)}\cdot\mathds{1}_{\calA^{(t+1)}(\delta)},
 		\\
 		=&\Wt+\etat G^{(t)}.
 	\end{split}
\end{equation}
To streamline notation, we denote $E^{(t+1)}\cdot\mathds{1}_{\calA^{(t+1)}(\delta)}$ simply by $E^{(t+1)}$ throughout the subsequent analysis. We use specific $\cst_r$ to denote constants having polynomial dependency on $k, \sigma, \log(T), \log(1/\delta)$, and $\log(d)$, where each index $r$ refers to a unique constant. 

\subsection{Two Phases and the Proof}

Our analysis involves two phases: during the first phase $t\in [T_1]$ with $T_1 = \lfloor T / \log(T) \rfloor$, the index $\alphat$ maintains above $(1+1/k)^{-1}d^{-1/2}$ and will satisfies $\alpha^{(T_1)} \ge 1 - \epsilon$ for some small $\epsilon>0$ in the end. The final convergence rate will be established in the second phase. 

We first present the result during the first phase.

\begin{theorem}\label{thm-phase-I-tensor-PCA}
    Assume $d\geq\Omega(k)$ and $\lambda\leq\calO\left(d^{k/4}\right)$. Under Assumption \ref{ass-2} with $\sigma\geq\Omega(1)$, consider the dynamic generated via Algorithm \ref{SGA} with initialization $W^{(0)}=I_d$. For any $0<\delta<1$ and $0<\epsilon<1$, if we pick 
    \begin{align*}
        \frac{T_1}{\lceil \log(T_1)\rceil}\geq\frac{2097152\left(2+\log(\sigma kdT/\delta)\right)e^2\cst_4d^{\frac{k}{2}}}{\lambda^2\epsilon^2(1-\epsilon)^{\frac{k}{2}}\max\left\{k(k-4),\log^{-1}(d)\right\}},
    \end{align*} 
    and
    \begin{align*}
        \forall t\in [T_1], ~~~ \eta^{(t)} = \eta_0 = \frac{16d^{\frac{k}{4}-1}}{\lambda\epsilon\max\left\{k(k-4),\log^{-1}(d)\right\}T_1}.
    \end{align*}
    Then $\alpha^{(T_1)} \ge 1 - \epsilon$ with probability at least $1-\delta/2$. 
\end{theorem}
The proof for the second phase comprises two integral parts. In \emph{Part I}, we demonstrate that $\left\{\Wt\right\}_{t=T_1}^T$, which stems from the output of the first phase, can guarantee that $\alpha^{(t)} (T_1\leq t\leq T)$ remain confirmed within the neighborhood of 1 with high probability.
\begin{lemma}\label{phase-II-high-probability}
	Suppose 
	$$
	\eta_0\leq\min\left\{\frac{\lambda\epsilon\left(1-3\epsilon/2\right)^{\frac{k}{2}-1} \left(k+4\log^{-1}(3T_1^2/\delta)\right)}{4096e^2\cst_4d^{\frac{k}{4}+1}},\, \, \frac{\lambda k\left(1-\epsilon\right)^{\frac{k}{2}}\epsilon^2}{128\cst_4\log\left(T^2/{\delta}\right)}\right\}.
	$$
	Under the setting of Theorem \ref{thm-phase-I-tensor-PCA}, we consider SGD iterates starting from step $T_1$ with initialization $\alpha^{(T_1)}>1-\frac{3\epsilon}{2}$. The joint event $\bigcap_{t=T_1}^{T} \ticalE\left(\alpha^{(t)}\right)$ holds with probability at least $1-\delta/2$, where 
	\begin{align}
		\ticalE\left(\alphat\right):=\left\{\alphat\in\left[1-\frac{3\epsilon}{2},1\right]\right\}.\notag
	\end{align}
\end{lemma}
Lemma \ref{phase-II-high-probability} establishes that $\alphat\in\left[1 - \frac{3\epsilon}{2}, 1\right]$ with high probability for any $t \in [T_1:T]$. Leveraging this bounded interval and the recurrence relation of $\alphat$ in the second phase, we derive its convergence rate:
\begin{theorem}\label{phase-II-convergence-thm}
	Under the setting of Lemma \ref{phase-II-high-probability}, $\alpha^{(T)}$ satisfies the following bound
	\begin{align}
		\left(1-\alpha^{(T)}\right)^2\lesssim\left(1-\frac{\eta_0\lambda k\left(1-3\epsilon/2\right)^{\frac{k}{2}}}{2}\right)^{T_1}\frac{\epsilon^2}{\delta}+\frac{\lceil\log(T)\rceil\eta_{0}}{\lambda^2k^2(1-3\epsilon/2)^{k}\delta T^4}+\frac{\lceil\log(T)\rceil(\mathsf{c}_0^2k^2\sigma^4+\lambda^4k^4+\cst_4^2d^4)\eta_{0}}{\lambda^3k^3(1-3\epsilon/2)^{\frac{3k}{2}}\delta T},\notag
	\end{align}
	with probability at least $1-\delta$.
\end{theorem}
The first phase and the second phase results established above enable the proof of Theorem \ref{main-theorem}.
\begin{proof}[Proof of Theorem \ref{main-theorem}]
	As stipulated by the selection rules for the total iteration count $T$ and the initial step size $\eta_{0}$ in Theorem \ref{thm-phase-I-tensor-PCA}, we can require that $T$ and $\eta_0$ satisfy:
	\begin{align}
		\frac{T_1}{\lceil \log(T_1)\rceil}\geq\frac{2097152\left(2+\log(\sigma kdT/\delta)\right)e^2\cst_1^2\cst_4d^{\frac{k}{2}}}{\lambda^2\epsilon^2(1-\epsilon)^{\frac{k}{2}}\max\left\{k(k-4),\log^{-1}(d)\right\}},\quad \eta_0=\frac{16\cst_1d^{\frac{k}{4}-1}}{\lambda\epsilon\max\left\{k(k-4),\log^{-1}(d)\right\}T_1},\notag
	\end{align}
	where 
	$$
	\cst_1=\max\left\{\frac{\epsilon\max\left\{k(k-4),\log^{-1}(d)\right\}\log\left(\text{poly}(T)\right)}{k(1-\epsilon)^{\frac{k}{2}}d^{\frac{k}{4}-1}},1\right\}.
	$$
	Combining Theorem \ref{thm-phase-I-tensor-PCA}, Lemma \ref{phase-II-high-probability}, and Theorem \ref{phase-II-convergence-thm}, one can notice that the last iterate of Algorithm \ref{SGA} satisfies
	\begin{equation}
		\small
		\begin{aligned}
			\left(1-\frac{\llangle v_*,W^{(T)}v^*\rrangle}{\|W^{(T)}\|_{\tF}}\right)^2\lesssim\underbrace{\left(1-\frac{\eta_0\lambda k(1-3\epsilon/2)^{\frac{k}{2}}}{2}\right)^{T_1}\frac{\epsilon^2}{\delta}+\frac{\lceil\log(T)\rceil\eta_{0}}{\lambda^2k^2(1-3\epsilon/2)^{k}\delta T^4}+\frac{\lceil\log(T)\rceil(\mathsf{c}_0^2k^2\sigma^4+\lambda^4k^4+\cst_4^2d^4)\eta_{0}}{\lambda^3k^3(1-3\epsilon/2)^{\frac{3k}{2}}\delta T}}_{\err^2},
		\end{aligned}
	\end{equation}
	with probability at least $1-\delta$.
	
	Consider the symmetric matrix $X= \frac{1}{2\left\|W^{(T)}\right\|_{\tF}}\left(W^{(T)}+ \left[W^{(T)}\right]^{\top}\right)$. Let $\lambda_1 \geq \lambda_2 \geq \cdots \geq \lambda_d$ be its eigenvalues sorted in descending order, with corresponding eigenvectors $\{v_i\}_{i=1}^d$. Then we have
	\begin{align}\label{esti-1}
		\lambda_{1}\geq\llangle v_*,Xv_*\rrangle=\frac{\llangle v_*,W^{(T)}v_*\rrangle}{\|W^{(T)}\|_{\tF}}\geq1-\calO\left(\err\right).
	\end{align}
	Moreover, since vector $v_*$ can be written as the sum of two components: one that is parallel to $v_1$, and one that lies in $(v_1)_{\perp}$, we write $v_*$ as $v_*=\sum_{i=1}^{d}\alpha_iv_i$. Therefore, we can obtain
	\begin{align}\label{esti-2}
		\llangle v_*,Xv_*\rrangle=\sum_{i=1}^{d}\lambda_{i}\alpha_i^2.
	\end{align}
	Noticing that $\|X\|_{\tF}\leq 1$, we derive that $\sum_{i=2}^{d}\lambda_i^2\leq1-\left(1-\calO(\err)\right)^2$. Eqs.~\eqref{esti-1} and \eqref{esti-2} implicate that
	\begin{align}\label{eq-1}
		\alpha_1^2\geq\frac{1-\calO(\err)-\sum_{i=2}^{d}\lambda_i\alpha_i^2}{\lambda_{1}}\overset{\text{(a)}}{\geq}\frac{1-\calO(\err)-\left(\sum_{i=2}^{d}\lambda_{i}^2\right)^{1/2}(1-\alpha_1^2)^{1/2}}{\lambda_{1}},
	\end{align}
	where (a) is derived from the Cauchy-Schwarz inequality and the fact that $\sum_{i=2}^{d}\alpha_i^4\leq\sum_{i=2}^{d}\alpha_i^2=1-\alpha_1^2$. By Eq.~\eqref{eq-1}, we have
	\begin{align}\label{eq-2}
		1-\alpha_1^2\leq(1-\alpha_1^2)^{1/2}\calO(\err^{1/2})+\calO(\err)\Rightarrow 1-\alpha_1^2\leq\calO(\err).
	\end{align} 
	Considering $\|v^*-v_1\|^2$ and $\|v^*+v_1\|^2$, we have
	\begin{equation}\label{eq-3}
		\begin{split}
			\|v^*-v_1\|^2=&(1-\alpha_1)^2+\sum_{i=2}^{d}\alpha_i^2=(1-\alpha_1)^2+(1-\alpha_1^2),
			\\
			\|v^*+v_1\|^2=&(1+\alpha_1)^2+\sum_{i=2}^{d}\alpha_i^2=(1+\alpha_1)^2+(1-\alpha_1^2).
		\end{split}
	\end{equation}
	The combination of Eq.~\eqref{eq-2} and Eq.~\eqref{eq-3} implies that $\min\left\{\|v^*-v_1\|^2,\|v^*+v_1\|^2\right\}\leq\calO(\err)$. Consequently, the power method guarantees that $v_{\text{Alg}}$ converges linearly to $v_1$, achieving high-precision approximation within few iterations. By choosing $\bar{c}=k^{-1}$ (i.e. $\epsilon=2k^{-1}/3
	$) and $\delta\in(0,1)$, we complete the proof.
\end{proof}

\subsection{Proof of the First Phase (Theorem \ref{thm-phase-I-tensor-PCA})}
The following lemma is a deterministic result claiming that the residuals $\Psi_1$ and $\Psi_2$ are of $\log (Tkd/\delta)$ order. 

\begin{lemma}\label{estimation}
	Suppose $d>k$. For any $t\in [T]$, if we further assume that 
	\begin{align*}
		\alpha^{(t)} \ge \frac{1}{(1+1/k)d^{1/2}} \qquad \text{and} \qquad \etat\leq\frac{1}{8\left(\lambda k^2+\sqrt{\mathsf{c}_1}d^2\right)},
	\end{align*} 
	then the following hold
	\begin{align}
		\left|-\Psi_1\left(\Wt,Q^{(t)},v_*,\bar{\eta}^{(t)}\right)\right|\leq&32\left(\lambda\left[\alphat\right]^{\frac{k}{2}}+\sqrt{\mathsf{c}_1}\right)\cdot\left(\lambda k\left[\alphat\right]^{\frac{k}{2}-1}+\sqrt{\mathsf{c}_1}\right)\notag
		\\
		&\qquad\qquad+32\alphat\cdot\left(\lambda^2k^2\left[\alphat\right]^{k-2}+\mathsf{c}_1(d^2+1)\right),\label{esti-Psi1}
		\\
		\Psi_2\left(\Wt,Q^{(t)},v_*,\bar{\eta}^{(t)}\right)\leq&16\left(\lambda k\left[\alphat\right]^{\frac{k}{2}-1}+\sqrt{\mathsf{c}_1}\right)^2.\label{esti-Psi2}
	\end{align}
	Moreover, we can obtain
	\begin{align}
		&\alpha\left(v_*,\Wt+\bar{\eta}^{(t)}Q^{(t)}\right)\geq\left(1-\frac{1}{k}\right)\alphat,\label{esti-perturb-tialpha}
		\\
		&\left|\frac{1}{\left\|\Wt\right\|_{\tF}}\llangle E^{(t+1)}, v_* v_*^{\top}\rrangle-\frac{\alphat}{\left\|\Wt\right\|_{\tF}^2}\llangle E^{(t+1)},\Wt\rrangle\right|\leq\sqrt{\mathsf{c}_1}\left(1+3\alphat\right).\label{uni-bound-1-order-random-term}
	\end{align}
\end{lemma}

\begin{proof}[Proof of Lemma \ref{estimation}]
	Recall the definition of $G^{(t)}$ and $E^{(t+1)}$ in \eqref{eq:gradient} and \eqref{eq:errort}, respectively. 
	Utilizing the construction of $E^{(t+1)}$ provided in section \ref{high-probability} directly, we can obtain
	\begin{align}
		\calI^{(t)}:=&\frac{\left|\llangle G^{(t)}, v_* v_*^{\top}\rrangle\right|}{\left\|\Wt\right\|_{\tF}}\leq\frac{\lambda k}{2}\left[\alphat\right]^{\frac{k}{2}-1}+\frac{\lambda(k-4)}{2}\left[\alphat\right]^{\frac{k}{2}+1}+\left|\frac{\llangle E^{(t+1)}, v_* v_*^{\top}\rrangle}{\left\|\Wt\right\|_{\tF}}\right|\notag
		\\
		&\, \, \, \, \, \, \, \, \, \, \, \, \, \, \, \, \, \, \, \, \, \, \, \, \, \, \, \, \, \, \, \, \, \, \, \, \, \leq\frac{\lambda k}{2}\left[\alphat\right]^{\frac{k}{2}-1}+\frac{\lambda(k-4)}{2}\left[\alphat\right]^{\frac{k}{2}+1}+\sqrt{\mathsf{c}_1}\left(1+\alphat\right),\label{component-1-PCA-dynamic}
		\\
		\calII^{(t)}:=&\frac{\left|\llangle G^{(t)},\Wt\rrangle\right|}{\left\|\Wt\right\|_{\tF}^2}\leq2\lambda\left[\alphat\right]^{\frac{k}{2}}+\left|\frac{\llangle{E}^{(t+1)},\Wt\rrangle}{\left\|\Wt\right\|_{\tF}^2}\right|\leq2\lambda\left[\alphat\right]^{\frac{k}{2}}+\sqrt{2\mathsf{c}_1},\label{component-2-PCA-dynamic}
		\\
		\calIII^{(t)}:=&\frac{\left\|G^{(t)}\right\|_{\tF}^2}{\left\|\Wt\right\|_{\tF}^2}\leq2(\lambda k)^2\left[\alphat\right]^{k-2}+2\lambda^2(k-4)^2\left[\alphat\right]^k+8\frac{\left\|{E}^{(t+1)}\right\|_{\tF}^2}{\left\|\Wt\right\|_{\tF}^2}\notag
		\\
		&\, \, \, \, \, \, \, \, \, \, \, \, \, \, \, \, \, \, \, \, \, \, \, \, \leq2(\lambda k)^2\left[\alphat\right]^{k-2}+2\lambda^2(k-4)^2\left[\alphat\right]^k+8\mathsf{c}_1(d^2+1).\label{component-3-PCA-dynamic}
	\end{align}
	According to the above estimation and the setting of $\eta_t$, we also have
	\begin{equation}\label{component-4-PCA-dynamic}
		\small
		\begin{aligned}
			\frac{\left\|\Wt+\bar{\eta}^{(t)}G^{(t)}\right\|_{\tF}^2}{\left\|\Wt\right\|_{\tF}^2}\leq&1+\etat\left(4\lambda\left[\alphat\right]^{\frac{k}{2}}+2\sqrt{2\mathsf{c}_1}\right)+4\left[\etat\right]^2\left(\lambda^2k^2\left[\alphat\right]^{k-2}+2\mathsf{c}_1(d^2+1)\right)\leq1+\frac{1}{k^2},
			\\	\frac{\left\|\Wt+\bar{\eta}^{(t)}G^{(t)}\right\|_{\tF}^2}{\left\|\Wt\right\|_{\tF}^2}\geq&1-\etat\left(4\lambda\left[\alphat\right]^{\frac{k}{2}}+2\sqrt{2\mathsf{c}_1}\right)\geq1-\frac{1}{k^2}.
		\end{aligned}
	\end{equation} 
	According to the definition of $\Psi_1(W, Q, v, \eta)$ in Eq.~\eqref{psi-1},
	applying Eq.~\eqref{component-1-PCA-dynamic}-\eqref{component-4-PCA-dynamic} to the expression of $\Psi_1\left(\Wt,G^{(t)},v_*,\bar{\eta}^{(t)}\right)$ yields
	\begin{equation}
		\small
		\begin{aligned}
			\left|-\Psi_1\left(\Wt,G^{(t)},v_*,\bar{\eta}^{(t)}\right)\right|\leq&4\sqrt{2}\calI^{(t)}\cdot\left(\calII^{(t)}+\bar{\eta}^{(t)}\calIII^{(t)}\right)+2\sqrt{2}\calIII^{(t)}\cdot\left(\alphat+\bar{\eta}^{(t)}\calI^{(t)}\right)
			\\
			\leq&4\sqrt{2}\left(2\sqrt{2}\left[\alphat\right]^{\frac{k}{2}}+\sqrt{\mathsf{c}_1}\right)\cdot\calI^{(t)}+4\alphat\cdot\calIII^{(t)}.\notag
		\end{aligned}
	\end{equation}
	Similarly, based on the definition of $\Psi_2(W, Q, v, \eta)$ in Eq.~\eqref{psi-2}, we further obtain
	\begin{align*}
		\Psi_2\left(\Wt,G^{(t)},v_*,\bar{\eta}^{(t)}\right)\leq&2\left[\calI^{(t)}+\bar{\eta}^{(t)}\left(\alphat+\bar{\eta}^{(t)}\calI^{(t)}\right)\cdot\left(\calII^{(t)}+\bar{\eta}^{(t)}\calIII^{(t)}\right)\right]^2\notag
		\\
		\leq&2\left[\calI^{(t)}+\sqrt{2}\bar{\eta}^{(t)}\alphat\left(2\sqrt{2}\left[\alphat\right]^{\frac{k}{2}}+\sqrt{\mathsf{c}_1}\right)\right]^2.\notag
	\end{align*}
	It can be derived that
	\begin{align}
		\alpha\left(v_*,\Wt+\bar{\eta}^{(t)}G^{(t)}\right)\geq\frac{k}{\sqrt{k^2+1}}\alphat-\frac{\bar{\eta}^{(t)}k}{\sqrt{k^2-1}}\calI^{(t)}\geq\left(1-\frac{1}{k}\right)\alphat,\notag
	\end{align}
	combining Eq.~\eqref{component-1-PCA-dynamic} and Eq.~\eqref{component-4-PCA-dynamic}. Finally, the union bound \eqref{uni-bound-1-order-random-term} follows from the definition of ${E}^{(t+1)}$.
\end{proof}

For given $\delta$ and $\epsilon$, we define the three ``bad'' events
\begin{align*}
	\mathcal{E}_{1} &= \left\{\exists t\in [T_1], \alphat < \frac{1}{(1+1/k)d^{1/2}} \right\}, 
	\\
	\mathcal{E}_{2} &= \left\{ \max_{t\in [T_1]} \alphat < 1 - \frac{\epsilon}{2} \right\}, 
	\\
	\mathcal{E}_{3} &= \left\{  \alpha^{(T_1)} < 1 - \epsilon \right\}.
\end{align*} 
In the remaining three steps, we will show that $\mathbb{P}\left[\left(\bigcap_{1\le l'< l} \mathcal{E}_{\ell'}^c\right) \bigcap \mathcal{E}_\ell \right] \le \delta/6$ for any $\ell \in \{1,2,3\}$. Thus, applying union bound further yields
\begin{align*}
	\mathbb{P}\left[\mathcal{E}_3\right] &= \mathbb{P} \left[ \mathcal{E}_1 \bigcup \left(\mathcal{E}_1^c \bigcap \mathcal{E}_2 \right) \bigcup \left(\mathcal{E}_1^c \bigcap \mathcal{E}_2^c \bigcap \mathcal{E}_3 \right)\right]
	\le 3\times \frac{\delta}{6} \le \frac{\delta}{2},
\end{align*} 
which further implies $\mathbb{P}\left[\mathcal{E}_3\right] \le \delta/2$. 
Next, we define the constrained coupling processes used in the following lemmas as below. 
\begin{lemma}\label{PCA-phase-I-upper-bound}
	Assume $d\geq\Omega(k)$ and $\lambda\leq\calO\left(d^{k/4}\right)$, and suppose $\eta_0\leq f_1(k,d), T_1\geq\eta_0^{-1}$ for all $k\geq 4$, and $T_1\eta_0^2\leq f_2(k,d)$ when $k>4$, where
	\begin{align}\label{para-lemma-upper-bound}
		f_1(k,d):=\frac{\lambda \left(k+4\log^{-1}(3T_1^2/\delta)\right)}{2048e^2\mathsf{c}_1d^{\frac{k}{4}+1}},\quad f_2(k,d):=\frac{(e^{\frac{2}{3}}-1)\log^{-1}(3T_1^2/\delta)}{16e(k-4)^2\mathsf{c}_1d}.
	\end{align}
	Under the setting of Theorem \ref{thm-phase-I-tensor-PCA}, the event $\calE_1$ holds with probability at most $\frac{\delta}{6}$.
\end{lemma}
\begin{proof}[Proof of Lemma \ref{PCA-phase-I-upper-bound}]
	We commence the proof by defining a constrained coupling process.
	\begin{definition}\label{def-hatW}
		Let $\left\{\Wt\right\}_{t=0}^T$ be a Markov chain in $\bbR^{d\times d}$ adapted to filtration $\left\{\calF^{(t)}\right\}_{t=0}^T$. Define following event for scalar
		\begin{equation}
			\calE(\alpha):=\left\{\alpha\geq \frac{1}{(1+1/k)d^{1/2}}\right\}.\notag
		\end{equation}
		The constrained coupling process $\left\{\hatsW^{(t)}\right\}_{t=0}^T$ with initialization $\hatsW^{(0)}=W^{(0)}$ evolves as
		\begin{enumerate}
			\item \emph{Updating stage: } If $\hatsW^{(t)}$ satisfies $\calE\left(\hatalpha^{(t)}:=\frac{\llangle v_*,\hatsW^{(t)} v_*\rrangle}{\left\|\hatsW^{(t)}\right\|_{\tF}}\right)$,
			let $\hatsW^{(t+1)}=W^{(t+1)}$.
			\item \emph{Absorbing state: } Otherwise, maintain $\hatsW^{(t+1)}=\hatsW^{(t)}$.
		\end{enumerate}
	\end{definition}
	Let $\bar{\tau}$ be the stopping time when $\hatalpha^{(\bar{\tau})} < \left((1+1/k)d^{1/2}\right)^{-1}$, i.e.,
	\begin{equation}\nonumber
		\bar{\tau}=\inf_{t}\left \{ t:\hatalpha^{(\bar{\tau})} < \frac{1}{(1+1/k)d^{1/2}} \right \}.
	\end{equation}
	\noindent\textbf{Case I ($k>4$): }Based on Definition \ref{def-hatW}, when the stopping time $\bar{\tau}$ occurs for some $t_2\in[T_1]$, the coupling process satisfies $\hatalpha^{(t)}=\hatalpha^{(t_2)}$ for all $t>t_2$. That is, the event 
	$\calE\left(\hatalpha^{(t)}\right)$ holds for all $t\in[0:t_2-1]$. According to the dynamic of $\left[\alphat\right]^{-(\frac{k}{2}-2)}$ in Eq.~\eqref{dynamic-tialpha-poly} and the boundedness estimation provided by Lemma \ref{estimation}, one can notice that $\left[\hatalpha^{(t)}\right]^{-(\frac{k}{2}-2)}$ must traverse in and out of the threshold interval $\left[d^{\frac{k-4}{4}},(1+1/k)^{\frac{k-4}{2}}d^{\frac{k-4}{4}}\right]$ before exceeding $(1+1/k)^{\frac{k-4}{2}}d^{\frac{k-4}{4}}$. We aim to estimate the following probability for time pairs $t_1<t_2\in[T_1]$:
	\begin{equation}
		\begin{aligned}
			&\bbP\left(\calE_{t_1}^{\bar{\tau}=t_2}:=\left\{\left[\hatalpha^{(t_1)}\right]^{-(\frac{k}{2}-2)}\leq\frac{1+(1+1/k)^{\frac{k-4}{2}}}{2}d^{\frac{k-4}{4}}\bigcap\left[\hatalpha^{(t_1:t_2-1)}\right]^{-(\frac{k}{2}-2)}\in\left[d^{\frac{k-4}{4}},(1+1/k)^{\frac{k-4}{2}}d^{\frac{k-4}{4}}\right]\right.\right.
			\\
			&\, \, \, \, \, \, \, \, \, \, \, \, \, \, \, \, \, \, \, \, \, \, \, \, \, \, \, \, \, \, \, \, \, \, \, \, \, \, \, \,  \left.\left.\bigcap\left[\hatalpha^{(t_2)}\right]^{-(\frac{k}{2}-2)}\geq (1+1/k)^{\frac{k-4}{2}}d^{\frac{k-4}{4}}\right\}\right).\notag
		\end{aligned}
	\end{equation}
	For any $t\in[t_1:t_2-1]$, we have 
	\begin{equation}\label{dynamic-hatalpha-poly}
		\small
		\begin{aligned}
			\left[\hatalpha^{(t+1)}\right]^{-(\frac{k}{2}-2)}\overset{\text{(a)}}{\leq}&\left[\hatalpha^{(t)}\right]^{-(\frac{k}{2}-2)}-\frac{\eta_0\lambda k(k-4)}{4}\left(1-\left[\hatalpha^{(t)}\right]^2\right)+\frac{\eta_0(k-4)}{2}\cdot\left[\hatalpha^{(t)}\right]^{-(\frac{k}{2}-1)}\cdot\widehat{\xi}^{(t+1)}
			\\
			&\qquad\qquad+\frac{\eta_0(k-4)}{2}\cdot\left[\hatalpha^{(t)}\right]^{-(\frac{k}{2}-1)}\cdot\left|\bbE_t\left[\frac{1}{\left\|\hatsW^{(t)}\right\|_{\tF}}\llangle{E}^{(t+1)}, v_* v_*^{\top}-\hatalpha^{(t)}\frac{\hatsW^{(t)}}{\left\|\hatsW^{(t)}\right\|_{\tF}}\rrangle\right]\right|
			\\
			&\qquad\qquad+\frac{\eta_0^2(k-4)(k-2)\left(1+\frac{2}{k}\right)^{\frac{k}{2}}}{8}\cdot\left[\hatalpha^{(t)}\right]^{-\frac{k}{2}}\cdot\Psi_2\left(\hatsW^{(t)},\widehat{G}^{(t)},v_*,\bar{\eta}^{(t)}\right)
			\\
			&\qquad\qquad+\frac{\eta_0^2(k-4)\left(1+\frac{2}{k}\right)^{\frac{k}{2}-1}}{4}\cdot\left[\hatalpha^{(t)}\right]^{-(\frac{k}{2}-1)}\cdot\left|-\Psi_1\left(\hatsW^{(t)},\widehat{G}^{(t)},v_*,\bar{\eta}^{(t)}\right)\right|
			\\
			\overset{\text{(b)}}{\leq}&\left[\hatalpha^{(t)}\right]^{-(\frac{k}{2}-2)}-\frac{\eta_0\lambda k(k-4)}{8}+\frac{\eta_0(k-4)}{2}\cdot\left[\hatalpha^{(t)}\right]^{-(\frac{k}{2}-1)}\cdot\widehat{\xi}^{(t+1)},
		\end{aligned}
	\end{equation}
	where $\widehat{G}^{(t)}$ denotes the stochastic gradient of the risk function $\calR$ evaluated at the parameter matrix $\hatsW^{(t)}$, and $\widehat{\xi}^{(t+1)}$ is a zero-mean random term which has the following form:
	\begin{equation}\label{def-scalar-hatS}
		\begin{aligned}
			\widehat{\xi}^{(t+1)}:=&\frac{1}{\left\|\hatsW^{(t)}\right\|_{\tF}}\llangle{E}^{(t+1)}, v_* v_*^{\top}\rrangle-\frac{\hatalpha^{(t)}}{\left\|\hatsW^{(t)}\right\|_{\tF}^2}\llangle{E}^{(t+1)},\hatsW^{(t)}\rrangle\\
			&\qquad\qquad-\bbE_t\left[\frac{1}{\left\|\hatsW^{(t)}\right\|_{\tF}}\llangle{E}^{(t+1)}, v_*v_*^{\top}\rrangle-\frac{\hatalpha^{(t)}}{\left\|\hatsW^{(t)}\right\|_{\tF}^2}\llangle{E}^{(t+1)},\hatsW^{(t)}\rrangle\right],
		\end{aligned}
	\end{equation}
	(a) follows from Eq.~\eqref{dynamic-alpha-poly} and Eq.~\eqref{esti-perturb-tialpha}, and (b) is derived from combining the construction of ${E}^{(t+1)}$ which satisfies
	\begin{align}
		\frac{\eta_0\lambda k(k-4)}{32}\geq&\frac{\eta_0^2(k-4)}{2}\cdot\left[\hatalpha^{(t)}\right]^{-(\frac{k}{2}-1)}\cdot\left|\bbE_t\left[\frac{1}{\left\|\hatsW^{(t)}\right\|_{\tF}}\llangle{E}^{(t+1)}, v_* v_*^{\top}\rrangle-\frac{\hatalpha^{(t)}}{\left\|\hatsW^{(t)}\right\|_{\tF}^2}\llangle{E}^{(t+1)},\hatsW^{(t)}\rrangle\right]\right|\notag
		\\
		\geq&\frac{\eta_0^2e(k-4)d^{\frac{k}{4}-\frac{1}{2}}}{2}\cdot\left|\bbE_t\left[\frac{1}{\left\|\hatsW^{(t)}\right\|_{\tF}}\llangle{E}^{(t+1)}, v_* v_*^{\top}\rrangle-\frac{\hatalpha^{(t)}}{\left\|\hatsW^{(t)}\right\|_{\tF}^2}\llangle{E}^{(t+1)},\hatsW^{(t)}\rrangle\right]\right|,\notag
	\end{align}
	and the result of Lemma \ref{estimation} with the setting of $\eta_0$ which implicates that 
	\begin{align}
		\frac{\eta_0\lambda k(k-4)}{64}\geq&2\eta_0^2e(k-4)(k-2)\cdot\left[\hatalpha^{(t)}\right]^{-\frac{k}{2}}\cdot\left(\lambda k\left[\hatalpha^{(t)}\right]^{\frac{k}{2}-1}+\sqrt{\mathsf{c}_1}\right)^2\notag
		\\
		\geq&4\eta_0^2e(k-4)(k-2)\left(\frac{\lambda^2k^2}{d^{\frac{k}{4}-1}}+e\mathsf{c}_1d^{\frac{k}{4}}\right),\notag
		\\
		\frac{\eta_0\lambda k(k-4)}{64}\geq&8\eta_0^2e(k-4)\cdot\left[\hatalpha^{(t)}\right]^{-(\frac{k}{2}-1)}\cdot\left[\left(\lambda\left[\hatalpha^{(t)}\right]^{\frac{k}{2}}+\sqrt{\mathsf{c}_1}\right)\cdot\left(\lambda k\left[\hatalpha^{(t)}\right]^{\frac{k}{2}-1}+\sqrt{\mathsf{c}_1}\right)\right.\notag
		\\
		&\qquad\qquad\left.+\hatalpha^{(t)}\cdot\left(\lambda^2k^2\left[\hatalpha^{(t)}\right]^{k-2}+\mathsf{c}_1(d^2+1)\right)\right]\notag
		\\
		\geq&8\eta_0^2e(k-4)\left(\frac{\lambda^2k(k+1)}{d^{\frac{k}{4}}}+2\lambda k\sqrt{\mathsf{c}_1}+\mathsf{c}_1+2e\mathsf{c}_1d^{\frac{k}{4}+1}\right).\notag
	\end{align}
	Since $\widehat{\xi}^{(t+1)}$
	is bounded, we demonstrate that ${E}^{(t+1)}$ satisfies the sub-Gaussian property for all $t\in[t_1:t_2]$. Thus we have
	\begin{equation}\label{poly-hatalpha-supermartingale}
		\begin{split}
			\bbE_t\left[\exp\left\{\gamma\left(\left[\hatalpha^{(t+1)}\right]^{-(\frac{k}{2}-2)}-\left[\hatalpha^{(t)}\right]^{-(\frac{k}{2}-2)}+\frac{\eta_0\lambda k(k-4)}{8}\right)\right\}\right]\leq \exp\left\{4e\gamma^2(k-4)^2\eta_0^2\mathsf{c}_1d^{\frac{k}{2}-1}\right\},
		\end{split}
	\end{equation}
	for any $\gamma\in\bbR_+$. Applying Eqs.~\eqref{dynamic-hatalpha-poly} and \eqref{poly-hatalpha-supermartingale} to Lemma \ref{aux-martingale-concentration-subtraction}, we can establish the  probability bound for event $\calE_{t_1}^{\bar{\tau}=t_2}$ for any time pair $t_1<t_2\in[T_1]$ as
	\begin{align}\label{prob-calEc-component}
		\bbP\left(\calE_{t_1}^{\bar{\tau}=t_2}\right)\leq\exp\left\{-\frac{e^{\frac{2}{3}}-1}{16e(k-4)^2\mathsf{c}_1T_1\eta_0^2d}\right\}.
	\end{align}
	We observe that the occurrence of event $\calE_1$ is equivalent to the existence of distinct time points $1\leq t_1<t_2\leq T$ such that event $\calE_{t_1}^{\bar{\tau}=t_2}$ occurs. This observation, in conjunction with the probability bound Eq.~\eqref{prob-calEc-component} and the setting of hyper-parameters in Lemma \ref{PCA-phase-I-upper-bound}, we obtain the following probability bound for event $\calE_1$:
	\begin{align}
		\bbP\left(\calE_1\right)\leq\sum_{1\leq t_1<t_2\leq T_1}\bbP\left(\calE_{t_1}^{\bar{\tau}=t_2}\right)\leq\frac{T_1^2}{2}\exp\left\{-\frac{e^{\frac{2}{3}}-1}{16e(k-4)^2\mathsf{c}_1T_1\eta_0^2d}\right\}\notag
		\leq\frac{\delta}{6}.\notag
	\end{align}
	\noindent\textbf{Case II ($k=4$): }Assume the stopping time $\bar{\tau}$ occurs for some $t_2\in[T_1]$. According to the dynamic of $\alphat$ in Eq.~\eqref{dynamic-tialpha-original} and the boundedness estimation provided by Lemma \ref{estimation}, we claim that $\hatalpha^{(t)}$ must traverse in and out of the threshold interval $\left[\frac{4}{5}d^{-1/2},d^{-1/2}\right]$ before subceeding $\frac{4}{5}d^{-1/2}$. We aim to estimate the following probability for time pairs $t_1<t_2\in[T_1]$:
	\begin{align}
		\bbP\left(\ticalE_{t_1}^{\bar{\tau}=t_2}:=\left\{\hatalpha^{(t_1)}\geq\frac{9}{10}d^{-\frac{1}{2}}\bigcap\hatalpha^{(t_1:t_2-1)}\in\left[\frac{4}{5}d^{-\frac{1}{2}},d^{-\frac{1}{2}}\right]\bigcap\hatalpha^{(t_2)}<\frac{4}{5}d^{-\frac{1}{2}}\right\}\right).\notag
	\end{align}
	For any $t\in[t_1:t_2-1]$, we have
	\begin{align}\label{dynamic-hatalpha-linear}
		\hatalpha^{(t+1)}\overset{\text{(c)}}{\geq}&\left[1+2\eta_0\lambda \left(1-\left[\hatalpha^{(t)}\right]^2\right)\right]\hatalpha^{(t)}+\eta_0\cdot\widehat{\xi}^{(t+1)}\notag
		\\
		&\qquad\qquad-\eta_0\left|\bbE_t\left[\frac{1}{\left\|\hatsW^{(t)}\right\|_{\tF}}\llangle{E}^{(t+1)}, v_* v_*^{\top}-\hatalpha^{(t)}\frac{\hatsW^{(t)}}{\left\|\hatsW^{(t)}\right\|_{\tF}}\rrangle\right]\right|\notag
		\\
		&\qquad\qquad-\frac{\eta_0^2}{2}\left|-\Psi_1\left(\hatsW^{(t)},\widehat{G}^{(t)},v_*,\bar{\eta}^{(t)}\right)\right|\notag
		\\
		\overset{\text{(d)}}{\geq}&\left(1+\eta_0\lambda\right)\hatalpha^{(t)}+\eta_0\cdot\widehat{\xi}^{(t+1)},
	\end{align}
	where (c) follows from Eq.~\eqref{dynamic-alpha-linear}, and (d) is also derived from the result of Lemma \ref{estimation} and the setting of $\eta_0$ which implicates that
	\begin{align}
		\frac{\eta_0\lambda}{2}\hatalpha^{(t)}\geq16\eta_0^2\cdot\left[\left(4\lambda \left[\hatalpha^{(t)}\right]+\sqrt{\mathsf{c}_1}\right)^2+\hatalpha^{(t)}\cdot\left(16\lambda^2\left[\hatalpha^{(t)}\right]^{2}+\mathsf{c}_1(d^2+1)\right)\right]
	\end{align}
	Based on the analysis for the sub-Gaussian parameter of $\widehat{\xi}^{(t+1)}$ in \textbf{Case I}, we have
	\begin{align}\label{linear-hatalpha-submartingale}
		\bbE_t\left[\exp\left\{\gamma\left(\hatalpha^{(t+1)}-\left(1+\eta_0\lambda\right)\hatalpha^{(t)}\right)\right\}\right]\leq\exp\left\{8\gamma^2\eta_0^2\mathsf{c}_1\right\},
	\end{align}
	for any $\gamma\in\bbR_-$. Therefore, we can establish the probability bound for event $\ticalE_{t_1}^{\bar{\tau}=t_2}$ for any time pair $T_1<t_2\in[T_1]$ as
	\begin{align}\label{prob-linear-calEc-component}
		\bbP\left(\ticalE_{t_1}^{\bar{\tau}=t_2}\right)\leq\exp\left\{-\frac{d^{-1}}{400\eta_0\mathsf{c}_1}\right\},
	\end{align}
	by applying Eqs.~\eqref{dynamic-hatalpha-linear} and \eqref{linear-hatalpha-submartingale} to Corollary \ref{aux-coro-4}. Finally, combining the probability bound Eq.~\eqref{prob-linear-calEc-component} with the setting of hyper-parameters in Lemma \ref{PCA-phase-I-upper-bound}, we obtain the following probability bound for event $\calE_1$:
	\begin{align}
		\bbP\left(\calE_1\right)\leq\sum_{1\leq t_1<t_2\leq T_1}\bbP\left(\ticalE_{t_1}^{\bar{\tau}=t_2}\right)\leq\frac{T_1^2}{2}\exp\left\{-\frac{d^{-1}}{400\eta_0\mathsf{c}_1}\right\}\notag
		\leq\frac{\delta}{6}.\notag
	\end{align}
\end{proof}
\begin{lemma}\label{PCA-phase-I-lower-bound}
	Assume $d\geq\Omega(k)$ and $\lambda\leq\calO\left(d^{k/4}\right)$, and suppose $\eta_0\leq{\epsilon}f_1(k,d)$ ($f_1(k,d)$ is defined in Eq.~\eqref{para-lemma-upper-bound}), and $T_1\geq f_3(k,{\epsilon},d)$, and $T_1\eta_0\geq f_4(k,{\epsilon},d)$, where
	\begin{align}
		f_3(k,{\epsilon},d)=&\frac{131072\mathsf{c}_1\log(6/\delta)d^{\frac{k}{2}-1}}{{\epsilon}^2\lambda^2k^2},\quad f_4(k,{\epsilon},d)=\frac{16d^{\frac{k}{4}-1}}{{\epsilon}\lambda\max\left\{k(k-4),\log^{-1}(d)\right\}}.\notag
	\end{align}
	Under the setting of Theorem \ref{thm-phase-I-tensor-PCA}, the combined event $\calE_1^c\bigcap\calE_2$ holds with probability at most $\frac{\delta}{6}$.
\end{lemma}
\begin{proof}
	The proof of Lemma \ref{PCA-phase-I-lower-bound} will be established through categorizing the following two cases and analyzing the probability bound respectively.
	
	\noindent\textbf{Case I ($k>4$): } In this part, we begin the proof by introducing a coupling process.
	\begin{definition}\label{def-breveW-1}
		Let $\left\{\Wt\right\}_{t=0}^T$ be a Markov chain in $\bbR^{d\times d}$ adapted to filtration $\left\{\calF^{(t)}\right\}_{t=0}^T$. The coupling process $\left\{\left[\breve{\alpha}^{(t)}\right]^{-(\frac{k}{2}-2)}\right\}_{t=0}^{T_1}$ with initialization $\left[\breve{\alpha}^{(0)}\right]^{-(\frac{k}{2}-2)}=\left[\alpha^{(0)}\right]^{-(\frac{k}{2}-2)}$ evolves as
		\begin{enumerate}
			\item Updating state: If event $\breve{\calE}\left(\breve{\alpha}^{(t)}\right):=\left\{\calE\left(\breve{\alpha}^{(t)}\right)\bigcap\breve{\alpha}^{(t)}<1-\frac{{\epsilon}}{2}\right\}$ holds, let $\left[\breve{\alpha}^{(t+1)}\right]^{-(\frac{k}{2}-2)}=$ $\left[\alpha^{(t+1)}\right]^{-(\frac{k}{2}-2)}$,
			\item Decaying state: Otherwise, let $\left[\breve{\alpha}^{(t+1)}\right]^{-(\frac{k}{2}-2)}=\left[\breve{\alpha}^{(t)}\right]^{-(\frac{k}{2}-2)}-\frac{\eta_0{\epsilon}\lambda k(k-4)}{8}$.
		\end{enumerate}
	\end{definition}
	We aim to demonstrate that $\left[\breve{\alpha}^{(t)}\right]^{-(\frac{k}{2}-2)}+\frac{t\eta_0{\epsilon}\lambda k(k-4)}{8}$ is a supermartingale. 
	If event $\breve{\calE}\left(\breve{\alpha}^{(t)}\right)^c$ holds, we directly obtain $\left[\breve{\alpha}^{(t+1)}\right]^{-(\frac{k}{2}-2)}\leq -\frac{\eta_0{\epsilon}\lambda k(k-4)}{8}+\left[\breve{\alpha}^{(t)}\right]^{-(\frac{k}{2}-2)}$. Otherwise, we have
	\begin{align}
		\left[\breve{\alpha}^{(t+1)}\right]^{-(\frac{k}{2}-2)}=&\left[\alpha^{(t+1)}\right]^{-(\frac{k}{2}-2)}\notag
		\\
		\overset{\text{(a)}}{\leq}&\left[\alpha^{(t)}\right]^{-(\frac{k}{2}-2)}-\frac{\eta_0{\epsilon}\lambda k(k-4)}{8}+\frac{\eta_0(k-4)}{2}\cdot\left[\alpha^{(t)}\right]^{-(\frac{k}{2}-1)}\cdot\xi^{(t+1)}\notag
		\\
		\overset{\text{(b)}}{=}&\left[\breve{\alpha}^{(t)}\right]^{-(\frac{k}{2}-2)}-\frac{\eta_0{\epsilon}\lambda k(k-4)}{8}+\frac{\eta_0(k-4)}{2}\cdot\left[\breve{\alpha}^{(t)}\right]^{-(\frac{k}{2}-1)}\cdot\xi^{(t+1)},\notag
	\end{align}
	where $\xi^{(t+1)}$ is a zero-mean random variable defined analogously to Eq.~\eqref{def-scalar-hatS}, with $\hatsW^{(t)}$ and $\hatalpha^{(t)}$
	substituted for $\Wt$ and $\alphat$, respectively. Here, (a) is derived from Eq.~\eqref{dynamic-hatalpha-poly}, and (b) relies on the temporal exclusivity property that if event $\breve{\calE}\left(\breve{\alpha}^{(t)}\right)^c$ occurs at time $t$, then $\breve{\calE}\left(\breve{\alpha}^{(t')}\right)$ is permanently excluded for all subsequent times $t' > t$. Therefore, based on the supermartingale, we obtain
	\begin{align}\label{lower-case-I}
		\bbP\left(\calE_1^c\bigcap\calE_2\right)\leq&\bbP\left(\breve{\alpha}^{(T_1)}<1-\frac{{\epsilon}}{2}\bigcap\calE\left(\breve{\alpha}^{(T_1)}\right)\right)=\bbP\left(\left[\breve{\alpha}^{(T_1)}\right]^{-(\frac{k}{2}-2)}>\left(1-\frac{{\epsilon}}{2}\right)^{-\frac{k-4}{2}}\bigcap\calE\left(\breve{\alpha}^{(T_1)}\right)\right)\notag
		\\
		\overset{\text{(c)}}{\leq}&\exp\left\{-\frac{\left(\frac{T_1\eta_0{\epsilon}\lambda k(k-4)}{8}+\left(1-\frac{{\epsilon}}{2}\right)^{-\frac{k-4}{2}}-d^{\frac{k}{4}-1}\right)^2}{4e(k-4)^2\mathsf{c}_1T_1\eta_0^2d^{\frac{k}{2}-1}}\right\}\notag
		\\
		\overset{\text{(d)}}{\leq}&\exp\left\{-\frac{T_1{\epsilon}^2\lambda^2k^2}{1024e\mathsf{c}_1d^{\frac{k}{2}-1}}\right\}\overset{\text{(e)}}{\leq}\frac{\delta}{6},
	\end{align}
	where (c) is derived from applying the estimation of $\xi^{(t+1)}$ below:
	\begin{align}
		\left|\xi^{(t+1)}\right|\leq4\sqrt{\mathsf{c}_1},\notag
	\end{align}
	which implicates that $\xi^{(t+1)}$ is sub-Gaussian with parameter $4\sqrt{\mathsf{c}_1}$, to Lemma \ref{aux-martingale-concentration-subtraction}. Moreover, since $T_1\eta_0\geq 16d^{\frac{k-4}{4}}\left({\epsilon}\lambda k(k-4)\right)^{-1}$ and $T_1\geq1024e\mathsf{c}_1(\lambda{\epsilon}k)^{-2}\log(6\delta^{-1})d^{\frac{k}{2}-1}$, we obtain inequalities (d) and (e). 
	
	\noindent\textbf{Case II ($k=4$): }In this part, we also begin the proof by introducing a coupling process.
	\begin{definition}\label{def-breveW-2}
		Let $\left\{\Wt\right\}_{t=0}^T$ be a Markov chain in $\bbR^{d\times d}$ adapted to filtration $\left\{\calF^{(t)}\right\}_{t=0}^T$. The coupling process $\left\{\breve{\alpha}^{(t)}\right\}_{t=0}^{T_1}$ with initialization $\breve{\alpha}^{(0)}=\alpha^{(0)}$ evolves as
		\begin{enumerate}
			\item Updating state: If event $\breve{\calE}\left(\breve{\alpha}^{(t)}\right):=\left\{\calE\left(\breve{\alpha}^{(t)}\right)\bigcap\breve{\alpha}^{(t)}<1-\frac{{\epsilon}}{2}\right\}$ holds, let $\breve{\alpha}^{(t+1)}=\alpha^{(t+1)}$,
			\item Decaying state: Otherwise, let $\breve{\alpha}^{(t+1)}=\left(1+\eta_0\lambda {\epsilon}/2\right)\breve{\alpha}^{(t)}$.
		\end{enumerate}
	\end{definition}
	We aim to demonstrate that $-t\log\left(1+\eta_0\lambda {\epsilon}/2\right)+\log\left(\breve{\alpha}^{(t)}\right)$ is a submartingale. If event $\breve{\calE}\left(\breve{\alpha}^{(t)}\right)^c$ holds, we directly obtain $\bbE_t\left[\log\left(\breve{\alpha}^{(t+1)}\right)\right]\geq\log(1+\eta_0\lambda {\epsilon}/2)+\log\left(\breve{\alpha}^{(t)}\right)$. Otherwise, we obtain
	\begin{align}
		\log\left(\breve{\alpha}^{(t+1)}\right)=&\log\left(\alpha^{(t+1)}\right)\notag
		\\
		\overset{\text{(f)}}{\geq}&\log\left(\left(1+\eta_0\lambda{\epsilon}\right)\alpha^{(t)}+\eta_0\cdot{\xi}^{(t+1)}\right)\notag
		\\
		=&\log\left(\alpha^{(t)}\right)+\log\left(1+\eta_0\lambda{\epsilon}+\frac{\eta_0}{\alpha^{(t)}}{\xi}^{(t+1)}\right)\notag
		\\
		\overset{\text{(g)}}{\geq}&\log\left(\alpha^{(t)}\right)+\log\left(1+\eta_0\lambda{\epsilon}\right)+\frac{1}{1+\eta_0\lambda{\epsilon}}\cdot\frac{\eta_0}{\alpha^{(t)}}{\xi}^{(t+1)}-\frac{\eta_0^2}{\left[\alpha^{(t)}\right]^2}\left[{\xi}^{(t+1)}\right]^2\notag
		\\
		\overset{\text{(h)}}{\geq}&\log\left(\alpha^{(t)}\right)+\log\left(1+\frac{\eta_0\lambda{\epsilon}}{2}\right)+\frac{1}{1+\eta_0\lambda{\epsilon}}\cdot\frac{\eta_0}{\alpha^{(t)}}{\xi}^{(t+1)}\notag
		\\
		=&\log\left(\breve{\alpha}^{(t)}\right)+\log\left(1+\frac{\eta_0\lambda{\epsilon}}{2}\right)+\frac{1}{1+\eta_0\lambda{\epsilon}}\cdot\frac{\eta_0}{\breve{\alpha}^{(t)}}{\xi}^{(t+1)},
	\end{align}
	where (f) is derived from Eq.~\eqref{dynamic-hatalpha-linear}, (g) relies on the Taylor expansion of function $f(x):=\log(1+\eta_0\lambda{\epsilon}+x)$, and (h) is obtained from the following inequality
	\begin{align}
		\log\left(1+\eta_0\lambda{\epsilon}\right)-\log\left(1+\frac{\eta_0\lambda{\epsilon}}{2}\right)\geq\frac{\eta_0\lambda{\epsilon}}{4}\geq\frac{\eta_0^2}{\left[\alpha^{(t)}\right]^2}\left[{\xi}^{(t+1)}\right]^2.\notag
	\end{align}
	Therefore, based on the submartingale, we have
	\begin{align}
		\bbP\left(\calE_1^c\bigcap\calE_2\right)\leq&\bbP\left(\breve{\alpha}^{(T_1)}<1-\frac{{\epsilon}}{2}\bigcap\calE(\breve{\alpha}^{(T_1)})\right)=\bbP\left(\breve{\alpha}^{(T_1)}<\left(1-\frac{{\epsilon}}{2}\right)\bigcap\calE(\breve{\alpha}^{(t)})\right)\notag
		\\
		\overset{\text{(i)}}{\leq}&\exp\left\{-\frac{\left(\frac{T_1\eta_0{\epsilon}\lambda }{4}+\log\left(\breve{\alpha}^{(0)}\right)-\log\left(1-\frac{{\epsilon}}{2}\right)\right)^2}{128\mathsf{c}_1T_1\eta_0^2d}\right\}\notag
		\\
		\overset{\text{(j)}}{\leq}&\exp\left\{-\frac{T_1{\epsilon}^2\lambda^2}{8192\mathsf{c}_1d}\right\}\overset{\text{(k)}}{\leq}\frac{\delta}{6},
	\end{align}
	where (i) is derived from Lemma \ref{aux-martingale-concentration-subtraction} with that $\alpha^{(t)}$ is lower-bounded by $\frac{4}{5}d^{-1/2}$ and $\widehat{\xi}^{(t+1)}$ is sub-Gaussian with parameter $4\sqrt{\mathsf{c}_1}$ under $\calE_1^c\bigcap\calE_2$. Moreover, since $T_1\eta_0\geq8\log(d)({\epsilon}\lambda)^{-1}$ and $T_1\geq8192\mathsf{c}_1$ $(\lambda{\epsilon})^{-2}\log(6\delta^{-1})d$, we obtain (j) and (k).
\end{proof}
\begin{lemma}\label{PCA-phase-I-lower-bound-last-iterate}
	Assume $d\geq\Omega(k)$ and $\lambda\leq\calO\left(d^{k/4}\right)$, and suppose 
	$$
	\eta_0\leq\min\left\{\frac{{\epsilon}\left(1-{\epsilon}\right)^{\frac{k}{2}-1}}{2}\cdot f_1(k,d),\frac{\lambda k\left(1-{\epsilon}\right)^{\frac{k}{2}}{\epsilon}^2}{256\mathsf{c}_1\log\left(\frac{3T_1^2}{\delta}\right)}\right\}.
	$$ 
	Under the setting of Theorem \ref{thm-phase-I-tensor-PCA}, the combined event $\calE_1^c\bigcap\calE_2^c\bigcap\calE_3$ holds with probability at most $\frac{\delta}{6}$.
\end{lemma}
\begin{proof}
	If $\calE_1^c\bigcap\mathcal{E}_{2}^c$ occurs, there exists $t \in [T_1]$ such that $\alpha^{(t)}\geq 1-\frac{{\epsilon}}{2}$. 
	For some $t_0\in[T_1]$, define $\hat{\tau}_{1}(t_0)$ as the stopping time satisfying $\alpha^{(\hat{\tau}_{1}(t_0))}\geq 1-\frac{{\epsilon}}{2}$ as:
	\begin{equation}\nonumber
		\hat{\tau}_{1}(t_0)=\inf_{t\geq t_0}\left \{ t:\alpha^{(t)}\geq 1-\frac{{\epsilon}}{2}\right\}.
	\end{equation}
	We also define $\hat{\tau}_{2}(t_0)$ as the stopping time satisfying $\alpha^{(\hat{\tau}_{2}(t_0))}<1-{\epsilon}$ after $\hat{\tau}_{1}(t_0)$ as:
	\begin{equation}\nonumber
		\hat{\tau}_{2}(t_0) =\inf_{t>\hat{\tau}_{1}(t_0)}\left \{ t:\alpha^{(t)}<1-{\epsilon}\right\}.
	\end{equation}
	According to the dynamic provided in Eq.~\eqref{dynamic-tialpha-original} and the setting of $\eta_0$, there exists $t_0\in[T_1]$ such that $\alpha^{(t)} (t\geq t_0)$ must traverse in and out of the threshold interval $\left[1-{\epsilon},1-\frac{{\epsilon}}{4}\right]$ before subceeding $1-{\epsilon}$. We aim to estimate the following probability for time pairs $t_1<t_2\in[T_1]$:
	\begin{equation}\nonumber
		\bbP\left(\calE _{\hat{\tau}_1(t_0)=t_1}^{\hat{\tau}_2(t_0)=t_2}:=\left\{\alpha^{(t_1)}\geq1-\frac{{\epsilon}}{2}\bigcap\alpha^{(t_1:t_2)}\in\left[1-{\epsilon},1-\frac{{\epsilon}}{4}\right]\bigcap\alpha^{(t_2)}<1-{\epsilon}\right\}\right).
	\end{equation}
	For any $t\in[t_1:t_2-1]$, we have 
	\begin{align}\label{martingale-concen-lower-bound}
		1-\alpha^{(t+1)}=&1-\alpha^{(t)}-\frac{\eta_0\lambda k}{2}\cdot\left[\alpha^{(t)}\right]^{\frac{k}{2}-1}\cdot\left(1+\alpha^{(t)}\right)\cdot\left(1-\alpha^{(t)}\right)\notag
		\\
		&\qquad\qquad-\eta_0\left[\frac{1}{\left\|\Wt\right\|_{\tF}}\llangle{E}^{(t+1)}, v_* v_*^{\top}\rrangle-\frac{\alpha^{(t)}}{\left\|\Wt\right\|_{\tF}^2}\llangle{E}^{(t+1)},\Wt\rrangle\right]\notag
		\\
		&\qquad\qquad-\frac{\eta_0^2}{2}\Psi_1\left(\Wt,G^{(t)},v_*,\bar{\eta}^{(t)}\right)\notag
		\\
		\overset{\text{(a)}}{\leq}&\left(1-\frac{\eta_0\lambda k(1-{\epsilon})^{\frac{k}{2}}}{2}\right)\left(1-\alpha^{(t)}\right)-\eta_0{\xi}^{(t+1)},
	\end{align}
	where (a) is derived from combining Eq.~\eqref{esti-Psi1} with the setting of $\eta_0$ which implicates that 
	\begin{equation}\label{concentrate-dynamic-condition}
		\begin{aligned}
			\frac{\eta_0{\epsilon}\lambda k(1-{\epsilon})^{\frac{k}{2}-1}}{16}\geq&48\eta_0^2\left(\lambda^2k^2+\mathsf{c}_1d^2\right),
			\\
			\frac{\eta_0{\epsilon}\lambda k(1-{\epsilon})^{\frac{k}{2}-1}}{16}\geq&\eta_0\left|\bbE\left[\frac{1}{\left\|\Wt\right\|_{\tF}}\llangle{E}^{(t+1)}, v_* v_*^{\top}\rrangle-\frac{\alpha^{(t)}}{\left\|\Wt\right\|_{\tF}^2}\llangle{E}^{(t+1)},\Wt\rrangle\right]\right|,
		\end{aligned}
	\end{equation}
	Since ${\xi}^{(t+1)}$ is sub-Gaussian with parameter $4\sqrt{\mathsf{c}_1}$, we establish the probability bound for event $\calE_{\hat{\tau}_1(t_0)=t_1}^{\hat{\tau}_2(t_0)=t_2}$ with any time pair $t_1<t_2\in[T_1]$ as
	\begin{align}\label{PCA-lower-bound}
		\bbP\left(\calE_{\hat{\tau}_1(t_0)=t_1}^{\hat{\tau}_2(t_0)=t_2}\right)\leq\exp\left\{-\frac{\lambda k(1-{\epsilon})^{\frac{k}{2}}{\epsilon}^2}{256\eta_0\mathsf{c}_1}\right\},
	\end{align} 
	by combining Lemma \ref{aux-martingale-concentration}. Therefore, we have
	\begin{align}
		\bbP\left(\calE_1^c\bigcap\calE_2^c\bigcap\calE_3\right)\leq\sum_{1\leq t_1<t_2\leq T_1}\bbP\left(\calE_{\hat{\tau}_1(t_0)=t_1}^{\hat{\tau}_2(t_0)=t_2}\right)\leq\frac{T_1^2}{2}\exp\left\{-\frac{\lambda k(1-{\epsilon})^{\frac{k}{2}}{\epsilon}^2}{256\eta_0\mathsf{c}_1}\right\}\leq\frac{\delta}{6},\notag
	\end{align}
	where the second inequality is derived from Eq.~\eqref{PCA-lower-bound} and the last inequality follows from the setting of $\eta_0$.
\end{proof}

\subsection{Proof of the Second Phase (Lemma \ref{phase-II-high-probability} and Theorem \ref{phase-II-convergence-thm})}
\noindent\textbf{Part I: The Union Bound of $\alpha^{(t)} (T_1\leq t\leq T)$.} In this part, we construct a compressed supermartingale sequence by adapting the technique from Lemma \ref{PCA-phase-I-lower-bound-last-iterate}. Leveraging the compression property of this sequence and the sub-Gaussian nature of its increments, we apply a concentration inequality to derive a uniform bound for $\left\{\alphat\right\}_{t=T_1}^T$. 
Recall the scalar event $\ticalE(\cdot)$ is defined as follows:
\begin{align}
	\ticalE(\gamma):=\left\{\gamma\in[1-\bar{\epsilon},1]\right\},\notag
\end{align}
where $\bar{\epsilon}=\frac{3}{2}\epsilon$ with $\epsilon$ has been defined in the first phase.
\begin{proof}[Proof of Lemma \ref{phase-II-high-probability}]
	We define $\hat{\tau}_3$ as the stopping time satisfying $\alpha^{(T_1+\hat{\tau}_3)}<1-\bar{\epsilon}$ as:
	\begin{align}
		\hat{\tau}_3=\inf_{t\geq0}\left\{t:\alpha^{(T_1+t)}<1-\bar{\epsilon}\right\}.\notag
	\end{align}
	Suppose there exists $t_2\in[T-T_1]$ such that $\hat{\tau}_3=t_2$, $\alpha^{(T_1+t)}$ must traverse in and out of the threshold interval $\left[1-\bar{\epsilon},1-\frac{\epsilon}{4}\right]$ before subceeding $1-\bar{\epsilon}$. We aim to estimate the following probability for time pairs $t_1<t_2\in[T-T_1]$ as:
	\begin{align}
		\bbP\left(\calE_{t_1}^{\hat{\tau}_3=t_2}:=\left\{\alpha^{(T_1+t_1)}\geq1-\epsilon\bigwedge\alpha^{(T_1+t_1:T_1+t_2)}\in\left[1-\bar{\epsilon},1-\frac{\epsilon}{4}\right]\bigwedge\alpha^{(T_1+t_2)}<1-\bar{\epsilon}\right\}\right).\notag
	\end{align}
	According to the dynamics of $\alpha^{(t)}$ provided by Eq.~\eqref{dynamic-tialpha-original}, we have
	\begin{align}
		1-\alpha^{(t+1)}=&1-\alpha^{(t)}-\frac{\etat\lambda k\left[\alpha^{(t)}\right]^{\frac{k}{2}-1}}{2}\cdot\left(1+\alpha^{(t)}\right)\cdot\left(1-\alpha^{(t)}\right)\notag
		\\
		&\qquad\qquad-\etat\left[\frac{1}{\left\|\Wt\right\|_{\tF}}\llangle E^{(t+1)}, v_* v_*^{\top}\rrangle-\frac{\alpha^{(t)}}{\left\|\Wt\right\|_{\tF}^2}\llangle E^{(t+1)},\Wt\rrangle\right]\notag
		\\
		&\qquad\qquad-\frac{\left[\etat\right]^2}{2}\Psi_1\left(\Wt,G^{(t)},v_*,\bar{\eta}^{(t)}\right)\notag
		\\
		\overset{\text{(a)}}{\leq}&\left(1-\frac{\etat\lambda k(1-\bar{\epsilon})^{\frac{k}{2}}}{2}\right)\left(1-\alpha^{(t)}\right)-\etat{\xi}^{(t+1)},\notag
	\end{align}
	for any $t\geq T_1$, where ${\xi}^{(t+1)}$ has the following form:
	\begin{equation}
		\begin{aligned}
			{\xi}^{(t+1)}:=&\frac{1}{\left\|\Wt\right\|_{\tF}}\llangle E^{(t+1)}, v_* v_*^{\top}\rrangle-\frac{\alpha^{(t)}}{\left\|\Wt\right\|_{\tF}^2}\llangle E^{(t+1)},\Wt\rrangle\\
			&-\bbE_t\left[\frac{1}{\left\|\Wt\right\|_{\tF}}\llangle E^{(t+1)}, v_*v_*^{\top}\rrangle-\frac{\alpha^{(t)}}{\left\|\Wt\right\|_{\tF}^2}\llangle E^{(t+1)},\Wt\rrangle\right],\notag
		\end{aligned}
	\end{equation}
	and (a) is derived from combining the result of Lemma \ref{estimation} with the setting of $\eta^{(T_1)}$ which implicates that Eq.~\eqref{concentrate-dynamic-condition} holds. Since $\widetilde{\xi}^{(t+1)}$ is sub-Gaussian with parameter $4\sqrt{\mathsf{c}_1}$, we establish the probability bound for event $\calE_{t_1}^{\hat{\tau}_3=t_2}$ with any time pair $t_1<t_2\in[T-T_1]$ as
	\begin{align}\label{PCA-lower-bound-Phase-II}
		\bbP\left(\calE_{t_1}^{\hat{\tau}_3=t_2}\right)\leq\exp\left\{-\frac{\lambda k(1-\bar{\epsilon})^{\frac{k}{2}}\epsilon^2}{128\eta_0\mathsf{c}_1}\right\},
	\end{align} 
	by combining Lemma \ref{aux-martingale-concentration}. Therefore, we have
	\begin{align}
		\sum_{0\leq t_1<t_2\leq T-T_1}\bbP\left(\calE_{t_1}^{\hat{\tau}_3=t_2}\right)\leq&\frac{T^2}{2}\exp\left\{-\frac{\lambda k(1-\bar{\epsilon})^{\frac{k}{2}}\epsilon^2}{128\eta_0\mathsf{c}_1}\right\}\leq\frac{\delta}{2},
	\end{align}
	where the last inequality follows from the setting of $\eta_0$.
\end{proof}

\noindent \textbf{Part II: Linear Approximation of the Dynamic of Objective Parameter Estimator.} Lemma \ref{phase-II-high-probability} illustrates that the output of Algorithm \ref{SGA} after $T_1$ iterations lies in the neighborhood of the ground truth $ v_* v_*^{\top}$, namely, $\alphat\in[1-\bar{\epsilon},1]$ for any $t\in[T_1:T]$ with high probability. Thus, we set the annealing rate to guarantee the output of Algorithm \ref{SGA} fully converge to $ v_*v_*^{\top}$ in the second phase. Before we formally propose Theorem \ref{converge-phase-II}, we preliminarily introduce some of the coupling process, auxiliary function, and notations used for our statement of Theorem \ref{converge-phase-II} and analysis in \textbf{Part II}. Letting $T_2:=T-T_1$, we introduce the truncated coupling $\left\{\olinesW^{(t)}\right\}_{t=0}^{T_2}$ with initialization $\olinesW^{(0)}=W^{(0)}$ as follows:
\begin{align}
	\begin{cases}
		\olinesW^{(t+1)}=W^{(T_1+t+1)}, & \text{ if }\ticalE\left(\alpha^{(T_1+t)}\right)\text{ occurs},
		\\
		\olinesW^{(\tau+1)}= v_{*,\perp} \left(v_{*,\perp}\right)^{\top},\, \, \, \forall \tau\geq t, & \text{ otherwise}.
	\end{cases}\notag
\end{align}
Moreover, we define the auxiliary function $\psi:\bbR^{d\times d}\rightarrow\bbR^{d\times d}$ as:
\begin{align}
	\psi(W)=\begin{cases}
		W, & \text{ if }\ticalE\left(\frac{\llangle v_*,W v_*\rrangle}{\|W\|_{\tF}}\right)\text{ occurs},
		\\
		v_*v_*^{\top}, & \text{ otherwise}.
	\end{cases}\notag
\end{align}
We construct the truncated sequence $\left\{ O^{(t)}=\psi\left(\olinesW^{(t)}\right)\right\}_{t=0}^{T_2}$. In this part, our analysis primarily focuses
on the trajectory of $\beta^{(t)}:=\frac{\llangle v_*, O^{(t)} v_*\rrangle}{\left\| O^{(t)}\right\|_{\tF}}$, which depends on that of $\alpha^{(T_1+t)}$. In \textbf{Part I}, we have proved that $\bigcap_{t=T_1}^T\ticalE\left(\alpha^{(t)}\right)$ occurs with high probability, which implies the
truncated sequence $\left\{ O^{(t)}\right\}_{t=1}^{T_2}$ aligned to $\left\{W^{(T_1+t)}\right\}_{t=1}^{T_2}$ with high probability. Then we approximate the update process of $\left\{\beta^{(t)}\right\}_{t=1}^{T_2}$ to SGD in traditional linear regression, with respective bounds of variance term and bias term.
 
In essence, the sequence $\left\{\beta^{(t)}\right\}_{t=0}^{T_2}$ constitutes a truncated version of $\left\{\alpha^{(t)}\right\}_{t=0}^{T_2}$. Based on the generation mechanism of the sequence $\left\{ O^{(t)}\right\}_{t=0}^{T_2}$, the update from $\beta^{(t)}$ to $\beta^{(t+1)}$ can be categorized into two cases: \emph{Case I}) $\beta^{(t+1)}$ remains updated, which has the following form
\begin{align}
	\beta^{(t+1)}=&\beta^{(t)}+\frac{\etat\lambda k}{2}\left(1-\left[\beta^{(t)}\right]^2\right)\left[\beta^{(t)}\right]^{\frac{k}{2}-1}\notag
	\\
	&\qquad\qquad+\etat\left[\frac{1}{\left\|O^{(t)}\right\|_{\tF}}\llangle{E}^{(t+1)}, v_* v_*^{\top}\rrangle-\frac{\beta^{(t)}}{\left\|O^{(t)}\right\|_{\tF}^2}\llangle{E}^{(t+1)},O^{(t)}\rrangle\right]\notag
	\\
	&\qquad\qquad+\frac{\left[\etat\right]^2}{2}\Psi_1\left(O^{(t)},G^{(t)},v_*,\bar{\eta}^{(t)}\right),\notag
\end{align}
where $G^{(t)}$ denotes the stochastic gradient of the risk function $\calR$ evaluated at the parameter matrix $O^{(t)}$, and function $\Psi_1$ has been defined in Eq.~\eqref{psi-1}, and $\bar{\eta}^{(t)} \in [0, \etat]$ being dependent on $(O^{(t)}, \Gt, \etat)$. \emph{Case II}) For any $\tau\geq t$, $\beta^{(t)}$ does not update and remains constant at 1.

For simplicity, we define $\bbR\ni\widehat{\beta}^{(t)}:=1-\beta^{(t)}$ and $\bbR\ni f_{\beta^{(t)}}:=\frac{\lambda k}{2}\left(1+\beta^{(t)}\right)\left[\beta^{(t)}\right]^{\frac{k}{2}-1}$. The following is the formalized expression of the iteration process for $\widehat{\beta}^{(t)}$. For all $t\in\{0\}\bigcup[T_2-1]$, if $O^{(t+1)}=W^{(T_1+t+1)}$, $\widehat{\beta}^{(t+1)}$ follows the update rule as:
\begin{align}\label{update}
	\widehat{\beta}^{(t+1)}=\left(1-\etat f_{\beta^{(t)}}\right)\widehat{\beta}^{(t)}+\etat\olinesg^{(t)}+\left[\etat\right]^2\tisg^{(t)},
\end{align}
where $\olinesg^{(t)}$ and $\tisg^{(t)}$ have the following definitions for any $t\in[0:T_2]$:
\begin{align}
	\begin{cases}
		\olinesg^{(t)}:=\frac{\beta^{(t)}}{\left\|O^{(t)}\right\|_{\tF}^2}\llangle{E}^{(t+1)},O^{(t)}\rrangle-\frac{1}{\left\|O^{(t)}\right\|_{\tF}}\llangle{E}^{(t+1)}, v_* v_*^{\top}\rrangle,
		\\
		\\
		\tisg^{(t)}:=-\frac{1}{2}\Psi_1\left(O^{(t)},G^{(t)},v_*,\bar{\eta}^{(t)}\right),
	\end{cases}\notag
\end{align}
Otherwise, we have
\begin{align}\label{non-update}
	\widehat{\beta}^{(\tau+1)}=0,\quad \forall\tau\geq t.
\end{align}
By combining Eq.~\eqref{update} with Eq.~\eqref{non-update}, the iterative update of $\left[\widehat{\beta}^{(t)}\right]^2$ can be expressed as:
\begin{align}\label{non-expect-recursion}
	\left[\widehat{\beta}^{(t+1)}\right]^2\leq&\left[\left(1-\etat f_{\beta^{(t)}}\right)\widehat{\beta}^{(t)}+\etat\olinesg^{(t)}+\left[\etat\right]^2\tisg^{(t)}\right]^2\cdot\mathds{1}_{O^{(t)}=W^{(T_1+t)}}\notag
	\\
	\leq&\left\{\left(1-\etat f_{\beta^{(t)}}\right)^2\left[\widehat{\beta}^{(t)}\right]^2+2\left(1-\etat f_{\beta^{(t)}}\right)\widehat{\beta}^{(t)}\left(\etat\olinesg^{(t)}+\left[\etat\right]^2\tisg^{(t)}\right)\right.\notag
	\\
	&\qquad\qquad\left.+2\left(\left[\etat\right]^2\left[\olinesg^{(t)}\right]^2+\left[\etat\right]^4\left[\tisg^{(t)}\right]^2\right)\right\}\cdot\mathds{1}_{O^{(t)}=W^{(T_1+t)}}.
\end{align}
Therefore, we can derive the following iterative relations for $\left\{\bbE\left[\left[\widehat{\beta}^{(t)}\right]^2\right]\right\}_{t=0}^{T_2}$ under Assumption \ref{ass-2}:
\begin{align}\label{main-recursion}
	\bbE\left[\left[\widehat{\beta}^{(t+1)}\right]^2\right]\leq&\left(1-\etat\lambda k\left(1-\bar{\epsilon}\right)^{\frac{k}{2}}\right)\bbE\left[\left[\widehat{\beta}^{(t)}\right]^2\right]+\left[\etat\right]^3g^{(t)}, 
\end{align}
where $g^{(t)}=\frac{1}{\left[\etat\right]^3T^6}+\frac{8192(\mathsf{c}_0^2k^2\sigma^4+\lambda^4k^4+\mathsf{c}_1^2d^4)}{\lambda k(1-\bar{\epsilon})^{\frac{k}{2}}}$.
\begin{proof}[Proof of Eq.~\eqref{main-recursion}]
	According to the recursive dynamic of $\left[\widehat{\beta}^{(t)}\right]^2$ in Eq.~\eqref{non-expect-recursion}, we obtain
	\begin{align}\label{esti-hatbeta}
		\bbE_t\left[\left[\widehat{\beta}^{(t+1)}\right]^2\right]\leq&\left(1-\etat f_{\beta^{(t)}}\right)^2\left[\widehat{\beta}^{(t)}\right]^2+2\etat\left(1-\etat f_{\beta^{(t)}}\right)\widehat{\beta}^{(t)}\bbE_t\left[\olinesg^{(t)}\right]\notag
		\\
		&\qquad\qquad+2\left[\etat\right]^2\left(1-\etat f_{\beta^{(t)}}\right)\widehat{\beta}^{(t)}\bbE_t\left[\tisg^{(t)}\right]\notag
		\\
		&\qquad\qquad+2\left(\left[\etat\right]^2\bbE_t\left[\left[\olinesg^{(t)}\right]^2\right]+\left[\etat\right]^4\bbE_t\left[\left[\tisg^{(t)}\right]^2\right]\right)\notag
		\\
		\overset{\text{(b)}}{\leq}&\left(1-\etat f_{\beta^{(t)}}\right)^2\left[\widehat{\beta}^{(t)}\right]^2+\left[\etat\right]^2\left(1-\etat f_{\beta^{(t)}}\right)^2\left[\widehat{\beta}^{(t)}\right]^2+\left(\bbE_t\left[\olinesg^{(t)}\right]\right)^2\notag
		\\
		&\qquad\qquad+\frac{\etat\lambda k(1-\bar{\epsilon})^{\frac{k}{2}}}{2}\left(1-\etat f_{\beta^{(t)}}\right)^2\left[\widehat{\beta}^{(t)}\right]^2\notag
		\\
		&\qquad\qquad+4\left(\left[\etat\right]^2\bbE_t\left[\left[\olinesg^{(t)}\right]^2\right]+\frac{\left[\etat\right]^3}{\lambda k(1-\bar{\epsilon})^{\frac{k}{2}}}\bbE_t\left[\left[\tisg^{(t)}\right]^2\right]\right)\notag
		\\
		\overset{\text{(c)}}{\leq}&\left(1-\frac{5\etat}{4}\lambda k(1-\bar{\epsilon})^{\frac{k}{2}}\right)\left[\widehat{\beta}^{(t)}\right]^2+\frac{1}{2T^6}\notag
		\\
		&\qquad\qquad+4\left(\left[\etat\right]^2\bbE_t\left[\left[\olinesg^{(t)}\right]^2\right]+\frac{\left[\etat\right]^3}{\lambda k(1-\bar{\epsilon})^{\frac{k}{2}}}\bbE_t\left[\left[\tisg^{(t)}\right]^2\right]\right)\notag
		\\
		\overset{\text{(d)}}{=}&\left(1-\frac{5\etat}{4}\lambda k(1-\bar{\epsilon})^{\frac{k}{2}}\right)\left[\widehat{\beta}^{(t)}\right]^2+\frac{1}{T^6}\notag
		\\
		&\qquad\qquad+4\left(\frac{c_0k\sigma^2\left[\etat\right]^2}{2}\widehat{\beta}^{(t)}\left(1+\beta^{(t)}\right)+\frac{\left[\etat\right]^3}{\lambda k(1-\bar{\epsilon})^{\frac{k}{2}}}\bbE_t\left[\left[\tisg^{(t)}\right]^2\right]\right)\notag
		\\
		\leq&\left(1-\etat\lambda k(1-\bar{\epsilon})^{\frac{k}{2}}\right)\left[\widehat{\beta}^{(t)}\right]^2+\frac{1}{T^6}\notag
		\\
		&\qquad\qquad+\frac{4\left[\etat\right]^3}{\lambda k(1-\bar{\epsilon})^{\frac{k}{2}}}\cdot\left(c_0^2k^2\sigma^4\left(1+\beta^{(t)}\right)^2+\bbE_t\left[\left[\tisg^{(t)}\right]^2\right]\right)\notag
		\\
		\overset{\text{(e)}}{\leq}&\left(1-\etat\lambda k(1-\bar{\epsilon})^{\frac{k}{2}}\right)\left[\widehat{\beta}^{(t)}\right]^2+\frac{1}{T^6}+\frac{8192\left[\etat\right]^3}{\lambda k (1-\bar{\epsilon})^{\frac{k}{2}}}\cdot\left(c_0^2k^2\sigma^4+\lambda^4k^4+\mathsf{c}_1^2d^4\right),
	\end{align}
	where (b) is obtained from Cauchy-Schwarz inequality: $2ab\leq a^2+b^2$ for any scalars $a,b\in\bbR$ and $(\bbE[g])^2\leq\bbE[g^2]$ for any random variable $g\in\bbR$, and (c) follows from the definition of $f_{{\beta^{(t)}}}$ and the fact that $\left|\bbE[\olinesg_t]\right|\leq{T^{-3}/2}$ since Lemma \ref{control-expectation} and the scale of $\mathsf{c}_1$. The inequality (d) relies on the following second-moment bound:
	\begin{align}
		\bbE_t\left[\left[\olinesg^{(t)}\right]^2\right]\overset{\text{(f)}}{\leq}&\bbE_t\left[\left(\frac{\beta^{(t)}}{\left\|O^{(t)}\right\|_{\tF}^2}\llangle E^{(t+1)},O^{(t)}\rrangle-\frac{1}{\left\|{O^{(t)}}\right\|_{\tF}}\llangle E^{(t+1)},v_*v_*^{\top}\rrangle\right)^2\right]+{\frac{1}{2T^6}}\notag
		\\
		=&\underbrace{\frac{1}{\left\|{O^{(t)}}\right\|_{\tF}^2}\bbE_t\left[\llangle E^{(t+1)},v_*v_*^{\top}\rrangle^2\right]}_{\calI^{(t)}}-\underbrace{\frac{2\beta^{(t)}}{\left\|{O^{(t)}}\right\|_{\tF}^3}\bbE_t\left[\llangle E^{(t+1)},v_*v_*^{\top}\rrangle\llangle E^{(t+1)},{O^{(t)}}\rrangle\right]}_{\calII^{(t)}}\notag
		\\
		&\qquad\qquad+\underbrace{\frac{\left[\beta^{(t)}\right]^2}{\left\|{O^{(t)}}\right\|_{\tF}^4}\bbE_t\left[\llangle E^{(t+1)},{O^{(t)}}\rrangle^2\right]}_{\calIII^{(t)}}+{\frac{1}{2T^6}}\notag
		\\
		=&\frac{c_0k\sigma^2}{2}\left(1-\left[{\beta^{(t)}}\right]^2\right)+{\frac{1}{2T^6}},\notag
	\end{align}
	where the random matrix $E^{(t+1)}$
	is defined explicitly in Eq.~\eqref{eq:errort} and excludes the truncation function $\mathds{1}_{\calA^{(t+1)}(\delta)}$, and (f) is derived from Lemma \ref{aux-truncation-subGaussian-second-moment} and the scale of $\mathsf{c}_1$. Moreover, $\calI^{(t)}$, $\calII^{(t)}$ and $\calIII_{3}^{(t)}$ have the following estimation
	\begin{align}
		\calI^{(t)}=&\frac{1}{\left\|{O^{(t)}}\right\|_{\tF}^{k-2}}\bbE_t\left[\left(\sum_{i=1}^{\frac{k}{2}}\llangle\upE^{(t+1)},\left[{O^{(t)}}\right]^{\otimes(i-1)}\otimes v_*v_*^{\top}\otimes \left[{O^{(t)}}\right]^{\otimes(\frac{k}{2}-i)}\rrangle\right)^2\right]\notag
		\\
		&+\frac{(k-4)^2\left[{\beta^{(t)}}\right]^2}{4\left\|{O^{(t)}}\right\|_{\tF}^{k}}\bbE_t\left[\llangle\upE^{(t+1)},\left[{O^{(t)}}\right]^{\otimes\frac{k}{2}}\rrangle^2\right]\notag
		\\
		&-\frac{(k-4){\beta^{(t)}}}{\left\|{O^{(t)}}\right\|_{\tF}^{k-1}}\bbE_t\left[\sum_{i=1}^{\frac{k}{2}}\llangle\upE^{(t+1)},\left[{O^{(t)}}\right]^{\otimes(i-1)}\otimes v_*v_*^{\top}\otimes \left[{O^{(t)}}\right]^{\otimes(\frac{k}{2}-i)}\rrangle\cdot\llangle\upE^{(t+1)},\left[{O^{(t)}}\right]^{\otimes\frac{k}{2}}\rrangle\right]\notag
		\\
		=&\mathsf{c}_0\sigma^2\left\{\frac{k^2}{4}-\frac{k^2-16}{4}\left[{\beta^{(t)}}\right]^2+\frac{k}{2}\left(\frac{k}{2}-1\right)\left(\left[{\beta^{(t)}}\right]^2-1\right)\right\},\notag
		\\
		\calII^{(t)}=&\frac{2{\beta^{(t)}}}{\left\|{O^{(t)}}\right\|_{\tF}^{k-1}}\bbE_t\left[\sum_{i=1}^{\frac{k}{2}}\llangle\upE^{(t+1)},\left[{O^{(t)}}\right]^{\otimes(i-1)}\otimes v_*v_*^{\top}\otimes \left[{O^{(t)}}\right]^{\otimes(\frac{k}{2}-i)}\rrangle\cdot\sum_{i=1}^{\frac{k}{2}}\llangle\upE^{(t+1)},\left[{O^{(t)}}\right]^{\otimes\frac{k}{2}}\rrangle\right]\notag
		\\
		&-\frac{(k-4){\beta^{(t)}}}{\left\|{O^{(t)}}\right\|_{\tF}^{k-1}}\bbE_t\left[\sum_{i=1}^{\frac{k}{2}}\llangle\upE^{(t+1)},\left[{O^{(t)}}\right]^{\otimes(i-1)}\otimes v_*v_*^{\top}\otimes \left[{O^{(t)}}\right]^{\otimes(\frac{k}{2}-i)}\rrangle\cdot\llangle \upE^{(t+1)},\left[{O^{(t)}}\right]^{\otimes\frac{k}{2}}\rrangle\right]\notag
		\\
		&-\frac{k(k-4)\left[{\beta^{(t)}}\right]^2}{2\left\|{O^{(t)}}\right\|_{\tF}^{k}}\bbE_t\left[\llangle\upE^{(t+1)},\left[{O^{(t)}}\right]^{\otimes\frac{k}{2}}\rrangle^2\right]+\frac{(k-4)^2\left[{\beta^{(t)}}\right]^2}{2\left\|{O^{(t)}}\right\|_{\tF}^k}\bbE_t\left[\llangle\upE^{(t+1)},\left[{O^{(t)}}\right]^{\otimes\frac{k}{2}}\rrangle^2\right]\notag
		\\
		=&2\mathsf{c}_0\sigma^2\left[\beta^{(t)}\right]^2\left[\frac{k^2}{4}+\frac{(k-4)^2}{4}-\frac{k(k-4)}{2}\right]=8\mathsf{c}_0\sigma^2\left[{\beta^{(t)}}\right]^2,\notag
		\\
		\calIII^{(t)}=&\frac{\left[{\beta^{(t)}}\right]^2}{\left\|{O^{(t)}}\right\|_{\tF}^k}\left(\frac{k^2}{4}-\frac{k(k-4)}{2}+\frac{(k-4)^2}{4}\right)\bbE_t\left[\llangle\upE^{(t+1)},\left[{O^{(t)}}\right]^{\otimes\frac{k}{2}}\rrangle^2\right]=4\mathsf{c}_0\sigma^2\left[{\beta^{(t)}}\right]^2.\notag
	\end{align}
	Finally, the inequality (e) is derived from the fact that $\beta_{t}\leq 1$ and applying Eq.~\eqref{esti-Psi1} to $\tisg^{(t)}$. Taking the expectation of both sides of Eq.~\eqref{esti-hatbeta} directly yields Eq.~\eqref{main-recursion}.
\end{proof}
\begin{proof}[Proof of Theorem \ref{phase-II-convergence-thm}]
Leveraging the recursive relation for $\bbE\left[\left[\widehat{\beta}^{(t)}\right]^2\right]$ specified in Eq.~\eqref{main-recursion}, we define two auxiliary processes $\left\{V^{(t)}\right\}_{t=0}^{T_2}$ and $\left\{B^{(t)}\right\}_{t=0}^{T_2}$ to complete the proof of Theorem \ref{converge-phase-II}:
\begin{align}
	V^{(t+1)}=&\left(1-\etat\lambda k(1-\bar{\epsilon})^{\frac{k}{2}}\right)V^{(t)}+\left[\etat\right]^3g^{(t)},\label{variance-square-hat-beta}
	\\
	B^{(t+1)}=&\left(1-\etat\lambda k(1-\bar{\epsilon})^{\frac{k}{2}}\right)B^{(t)},\label{bias-square-hat-beta}
\end{align}
with initialization $V^{(0)}=0$ and $B^{(0)}=\left[\widehat{\beta}^{(0)}\right]^2$, where $g^{(t)}$ follows the definition in Eq.~\eqref{main-recursion}. Therefore, we can obtain
\begin{align}\label{main-error}
	\bbE\left[\left[\widehat{\beta}^{(T_2)}\right]^2\right]\leq V^{(T_2)}+B^{(T_2)}.
\end{align}
\begin{theorem}\label{converge-phase-II}
	Under the setting of Lemma \ref{phase-II-high-probability}, we have
	\begin{align}
		\bbE\left[\left[\widehat{\beta}^{(T_2)}\right]^2\right]\leq \left(1-\frac{\eta_0\lambda k(1-\bar{\epsilon})^{\frac{k}{2}}}{2}\right)^{T_1}\bar{\epsilon}^2+\frac{4\lceil\log(T)\rceil\eta_{0}}{\lambda^2k^2(1-\bar{\epsilon})^{k}T^4}+\frac{32768\lceil\log(T)\rceil(\mathsf{c}_0^2k^2\sigma^4+\lambda^4k^4+\mathsf{c}_1^2d^4)\eta_{0}}{\lambda^3k^3(1-\bar{\epsilon})^{\frac{3k}{2}}T}.\notag
	\end{align}
\end{theorem}
\noindent \textbf{Bound of $V^{(T_2)}$: }Lemma \ref{esti-V} provides a uniform upper bound for $V^{(t)}$ over $t\in[0:T_2]$.
\begin{lemma}\label{esti-V}
	Under the setting of Lemma \ref{phase-II-high-probability}, define the step size $\etat$ satisfies $\etat = \eta^{(T_1)} \cdot 2^{-\lfloor t / T_1 \rfloor}$ for $t \in [0, T_2]$. Then we obtain
	\begin{align}
		V^{(T_2)}\leq\frac{8\lceil\log(T)\rceil\eta^{(T_1)}}{\lambda^2k^2(1-\bar{\epsilon})^{k}T^4}+\frac{65536\lceil\log(T)\rceil(\mathsf{c}_0^2k^2\sigma^4+\lambda^4k^4+\mathsf{c}_1^2d^4)\eta^{(T_1)}}{\lambda^3k^3(1-\bar{\epsilon})^{\frac{3k}{2}}T}.
	\end{align}
\end{lemma}
\begin{proof}
	According to the recursion provided by Eq.~\eqref{variance-square-hat-beta}, we can directly derive
	\begin{align}\label{main-variance}
		V^{(T_2)}=\sum_{t=0}^{T_2}\left[\etat\right]^3\prod_{i=t+1}^{T_2}\left(1-\eta^{(i)}\lambda k(1-\bar{\epsilon})^{\frac{k}{2}}\right)g^{(t)}.
	\end{align}
	Based on the update rule for $\eta_t$ defined in Lemma \ref{esti-V}, we have
	\begin{align}\label{eq-variance-I}
		V^{(T_2)}\leq&\sum_{t=0}^{T_2}\left[\etat\right]^3\prod_{i=t+1}^{T_2}\left(1-\eta^{(i)}\lambda k(1-\bar{\epsilon})^{\frac{k}{2}}\right)\cdot\underbrace{\left(\frac{1}{T^3}+\frac{8192(\mathsf{c}_0^2k^2\sigma^4+\lambda^4k^4+\mathsf{c}_1^2d^4)}{\lambda k(1-\bar{\epsilon})^{\frac{k}{2}}}\right)}_{g}\notag
		\\
		\leq&\eta^{(T_1)}\cdot\sum_{l=0}^{L-1}\left[\frac{\eta^{(T_1)}}{2^l}\right]^2\sum_{i=1}^{T_1}\left(1-\frac{\eta^{(T_1)}}{2^l}\lambda k(1-\bar{\epsilon})^{\frac{k}{2}}\right)^{T_1-i}\prod_{j=l+1}^{L-1}\left(1-\frac{\eta^{(T_1)}}{2^j}\lambda k(1-\bar{\epsilon})^{\frac{k}{2}}\right)^{T_1}g\notag
		\\
		\leq&\eta^{(T_1)}\cdot\left[\left[\eta^{(T_1)}\right]^2\sum_{i=1}^{T_1}\left(1-\eta^{(T_1)}\lambda k(1-\bar{\epsilon})^{\frac{k}{2}}\right)^{T_1-i}\prod_{j=1}^{L-1}\left(1-\frac{\eta^{(T_1)}}{2^j}\lambda k(1-\bar{\epsilon})^{\frac{k}{2}}\right)^{T_1}g\right.
		\notag
		\\
		&\qquad\qquad\left.+\sum_{l=1}^{L-1}\left[\frac{\eta^{(T_1)}}{2^l}\right]^2\sum_{i=1}^{T_1}\left(1-\frac{\eta^{(T_1)}}{2^l}\lambda k(1-\bar{\epsilon})^{\frac{k}{2}}\right)^{T_1-i}\prod_{j=l+1}^{L-1}\left(1-\frac{\eta^{(T_1)}}{2^j}\lambda k(1-\bar{\epsilon})^{\frac{k}{2}}\right)^{T_1}g\right]\notag
		\\
		\leq&\frac{\eta^{(T_1)} g}{\lambda k(1-\bar{\epsilon})^{\frac{k}{2}}}\cdot\left[\eta^{(T_1)}\left(1-\left(1-\eta^{(T_1)}\lambda k(1-\bar{\epsilon})^{\frac{k}{2}}\right)^{T_1}\right)\prod_{j=1}^{L-1}\left(1-\frac{\eta^{(T_1)}}{2^j}\lambda k(1-\bar{\epsilon})^{\frac{k}{2}}\right)^{T_1}\right.\notag
		\\
		&\qquad\qquad\left.+\sum_{l=1}^{L-1}\frac{\eta^{(T_1)}}{2^l}\left(1-\left(1-\frac{\eta^{(T_1)}}{2^l}\lambda k(1-\bar{\epsilon})^{\frac{k}{2}}\right)^{T_1}\right)\prod_{j=l+1}^{L-1}\left(1-\frac{\eta^{(T_1)}}{2^j}\lambda k(1-\bar{\epsilon})^{\frac{k}{2}}\right)^{T_1}\right].
	\end{align}
	Then, we define the following scalar function
	\begin{align}
		f(x):=x\left(1-(1-x)^{h+T_1}\right)\prod_{j=1}^{L-1}\left(1-\frac{x}{2^j}\right)^{T_1}+\sum_{l=1}^{L-1}\frac{x}{2^l}\left(1-\left(1-\frac{x}{2^l}\right)^{T_1}\right)\prod_{j=l+1}^{L-1}\left(1-\frac{x}{2^j}\right)^{T_1},\notag
	\end{align}
	as similar as that in [Theorem C.2, \cite{wu2022last}]. Moreover, the following inequality can be directly derived
	\begin{align}\label{aux-scalar-func}
		f\left(\eta^{(T_1)}\lambda k(1-\bar{\epsilon})^{\frac{k}{2}}\right)\leq\frac{8}{T_1},
	\end{align}
	by [Lemma C.3, \cite{wu2022last}]. Applying Eq.~\eqref{aux-scalar-func} to Eq.~\eqref{eq-variance-I} and combining Eq.~\eqref{main-variance}, we obtain
	\begin{align}
		V^{(T_2)}\leq\frac{8\eta^{(T_1)}g}{T_1\lambda^2k^2(1-\bar{\epsilon})^k}.\notag
	\end{align}
\end{proof}
\noindent \textbf{Bound of $B^{(T_2)}$: }By directly applying the recursive expression in Eq.~\eqref{bias-square-hat-beta}, Lemma \ref{esti-B} establishes an estimate for $B^{(T_2)}$.
\begin{lemma}\label{esti-B}
	Under the setting of Lemma \ref{phase-II-high-probability}, define the step size $\etat$ satisfies $\etat = \eta^{(T_1)} \cdot 2^{-\lfloor t / T_1 \rfloor}$ for $t \in [0, T_2]$. Then we obtain
	\begin{align}
		B^{(T_2)}\leq \left(1-\eta^{(T_1)}\lambda k(1-\bar{\epsilon})^{\frac{k}{2}}\right)^{T_1}B^{(0)}.
	\end{align}
\end{lemma}
\begin{proof}
	According to the recursion provided by Eq.~\eqref{bias-square-hat-beta}, we can directly derive
	\begin{align}
		B^{(T_2)}=&\prod_{l=0}^{L-1}\left(1-\frac{\eta^{(T_1)}}{2^l}\lambda k(1-\bar{\epsilon})^{\frac{k}{2}}\right)^{T_1}B^{(0)}\notag
		\\
		\leq&\left(1-\eta^{(T_1)}\lambda k(1-\bar{\epsilon})^{\frac{k}{2}}\right)^{T_1}B^{(0)}.\notag
	\end{align}
\end{proof}
\begin{proof}[Proof of Theorem \ref{converge-phase-II}]
	The proof is completed by applying the conclusions of Lemma \ref{esti-V} and Lemma \ref{esti-B} to Eq. \eqref{main-error}.
\end{proof}
According to the result of Theorem \ref{converge-phase-II}, we have
\begin{align}\label{final-error-bound}
	\left(1-\beta^{(T)}\right)^2\lesssim \left(1-\frac{\eta_0\lambda k(1-\bar{\epsilon})^{\frac{k}{2}}}{2}\right)^{T_1}\frac{\bar{\epsilon}^2}{\delta}+\frac{\lceil\log(T)\rceil(\mathsf{c}_0^2k^2\sigma^4+\lambda^4k^4+\mathsf{c}_1^2d^4)\eta_{0}}{\lambda^3k^3(1-\bar{\epsilon})^{\frac{3k}{2}}\delta T},
\end{align}
with probability at least $1-\delta/2$. The proof is completed by combining the error bound Eq.~\eqref{final-error-bound} with the fact that the event $\left\{\alpha^{(T)}=\beta^{(T_2)}\right\}$ occurs with probability at least $1-\delta/2$, i.e.,
\begin{align*}
\bbP\left(\left\{\alpha^{(T)}=\beta^{(T_2)}\right\}\right)\geq\bbP\left(\bigcap_{t=1}^{T_2}\left\{\alpha^{(T_1+t)}=\beta^{(t)}\right\}\right)\overset{\text{(a)}}{\geq}1-\frac{\delta}{2},
\end{align*}
where (a) is derived from Lemma \ref{phase-II-convergence-thm} and the construction methodology of matrix sequence $\left\{O^{(t)}\right\}_{t=1}^{T_2}$.
\end{proof}

\subsection{Proofs of Complementary Result}
\begin{proof}[Proofs of Corollary \ref{corollary-estimation-second-moment}]
	Let random variable $X=\langle u, \expan(\upE)\rangle$ for given unit vector $u\in\bbR^{d^k}$. According to Eq.~\eqref{eq-subGaussian-prob}, we have $\bbP(|X|\geq r)\leq 2e^{-\frac{r^2}{2\sigma^2}}$ for any $r>0$. Therefore, we obtain
	\begin{align}
		\bbE\left[\llangle u,\expan(\upE)\expan(\upE)^{\top}u\rrangle\right]=\bbE\left[X^2\right]=2\int_{0}^{\infty}r\bbP(|X|>r)\text{d}r\notag\leq4\int_{0}^{\infty}re^{-\frac{r^2}{2\sigma^2}}\text{d}r=4\sigma^2.\notag
	\end{align}
\end{proof}

\begin{proof}[Proof of Corollary \ref{corollary-no-Ass2-evenk}]
	Under the absence of Assumption \ref{assumption-estimation-second-moment}, the estimation of the term 
	$\bbE_t\left[\left[\olinesg^{(t)}\right]^2\right]$ in Eq.~\eqref{esti-hatbeta} is replaced by
	$$
	\bbE_t\left[\left[\olinesg^{(t)}\right]^2\right]\lesssim\sigma^2k^2,
	$$
	thereby leading to the following reformulation of the equation:
	\begin{align}
		\bbE_t\left[\left[\widehat{\beta}^{(t+1)}\right]^2\right]\leq\left(1-\etat\lambda k(1-\bar{\epsilon})^{\frac{k}{2}}\right)\left[\widehat{\beta}^{(t)}\right]^2+\frac{1}{N^6}+\sigma^2k^2\left[\eta^{(t)}\right]^2+\frac{\left[\eta^{(t)}\right]^3}{\lambda k(1-\bar{\epsilon})^{\frac{k}{2}}}\cdot\left(\lambda^4k^4+\mathsf{c}_1^2d^4\right).
	\end{align}
	Therefore, $V^{(T_2)}$ in Eq.~\eqref{main-variance} acquires additional terms as follows:
	$$
	V_{\mathrm{add}}^{(T_2)}:=\sigma^2k^2\sum_{t=0}^{T_2}\left[\etat\right]^2\prod_{i=t+1}^{T_2}\left(1-\eta^{(i)}\lambda k(1-\bar{\epsilon})^{\frac{k}{2}}\right).
	$$
	By employing proof techniques analogous to those used in Lemma \ref{esti-V}, we obtain:
	$$
	V^{(T_2)}\lesssim\frac{\sigma^2\lceil\log(N)\rceil}{\lambda^2(1-\bar{\epsilon})^{k}N}+\frac{\lceil\log(N)\rceil\eta^{(T_1)}}{\lambda^2k^2(1-\bar{\epsilon})^{k}N^4}+\frac{\lceil\log(N)\rceil(\lambda^4k^4+\mathsf{c}_1^2d^4)\eta^{(T_1)}}{\lambda^3k^3(1-\bar{\epsilon})^{\frac{3k}{2}}N}.
	$$
    Applying the above estimate to Theorem \ref{converge-phase-II} completes the proof.
\end{proof}
\begin{proof}[Proof of Corollary \ref{prob-odd}]
Notice that $u$ is sampled from a uniform distribution on the unit sphere in $\bbR^d$ and $v_*$ also lies on this sphere. According to the rotational invariance, without loss of generality, we can assume that $v_*=(1,0,\cdots,0)$. Thus, it suffices to analyze the magnitude of $|u_1|$.

Note that the random unit vector $u$ can be generated as $u=z/\|z\|$, where $z=(z_1,\dots,z_d)\sim\mathcal{N}(\bf{0},\upI_d)$. Consequently, $u_1=z_1/\|z\|$, and $\|z\|^2\sim \chi^2(d)$. For any $\tau>0$, the probability $\bbP\left(|u_1|\geq\tau d^{-1/2}\right)$ is equivalent to $\bbP\left(z_1^2/(z_1^2+s)>\tau^2d^{-1}\right)$, where 
$s=\sum_{i=2}^d z_i^2\sim\chi^2(d-1)$ and is independent of $z_1$.

For simplicity, let $x=z_1^2\sim\chi^2(1)$ and $y=s$. Then
\begin{align}
P\left(|u_1|\geq\tau d^{-1/2}\right)=P\left(\frac{x}{x+y}\geq\tau^2d^{-1}\right)=P\left(\frac{x}{y}\geq\frac{\tau^2d^{-1}}{1-\tau^2d^{-1}}\right)\overset{\text{(a)}}{=}P\left(F_{1,d-1}\geq\frac{\tau^2(d-1)}{d-\tau^2}\right),\notag
\end{align}
where (a) follows from the fact that for any $c>0$, $P\left(x/y\geq c\right)= P\left(F_{1,d-1}\geq c(d-1)\right)$ where $F_{1,d-1}$ denotes the $F$-distribution with $(1,d-1)$ degrees of freedom. 

Furthermore, for any $\delta>0$, the condition $P(F_{1,d-1}<c)\leq\delta$ is satisfied if 
$$
c\leq \left[t_{d-1,(1+\delta)/2} \right]^2,
$$
where $t_{d-1,(1+\delta)/2}$ is the $(1+\delta)/2$-quantile of the $t$-distribution with $d-1$ degrees of freedom. Therefore, in order to achieve the result of Corollary \ref{prob-odd}, we require
$$
\frac{\tau^2(d-1)}{d-\tau^2}\leq\left[t_{d-1,(1+\delta)/2}\right]^2.
$$
\end{proof}

\subsection{Proofs of Technical Lemmas}
\begin{proof}[Proof of Lemma \ref{lemma:high-prob}]
	For any $t\in [T]$, $i, j \in [d]$ , let $V := V^{(i,j)}$ be the matrix such that $V_{i',j'} = \indicator\{i'=i, j'=j\}$, then it follows from the linear in $\mathbf{E}^{(t)}$ expression of $E^{(t)}$ that
	\begin{align}
		\label{eq:eij}
		E_{i,j}^{(t)} = \langle E^{(t)}, V \rangle &= \frac{1}{2\|\Wt\|_F^{\halfk-2}} \sum_{l=1}^{\halfk} Z_l(V) - \frac{(k-4)}{2\|\Wt\|_F^{\halfk}} Z_0(V) W^{(t)}_{i,j}, 
	\end{align} where
	\begin{align*}
		Z_0(V) = \left\langle (\Wt)^{\otimes\halfk}, \mathbf{E}^{(t)}\right\rangle 
		\qquad Z_l(V) = \left\langle (\Wt)^{\otimes (l-1)} \otimes V \otimes (\Wt)^{\otimes \halfk-l}, \mathbf{E}^{(t)} \right\rangle \qquad
	\end{align*}
	Observe that conditional on $\mathcal{F}_{t-1}$, $Z_0(V)$ and $Z_l(V)$ are sum of $d^{k}$ i.i.d. random variables by Assumption \ref{ass-2}. It then follows from standard sub-Gaussian tail that for any $u>0$, the event $\mathcal{C}_{t,i,j,l}(u)$ defined as
	\begin{align*}
		\mathcal{C}_{t,i,j,l}(u) := \left\{|Z_l(V)| > 2\sigma \cdot \|\Wt\|^{\halfk - \indicator\{l>0\}}_F \cdot \sqrt{u}\right\}
	\end{align*} satisfies $\mathbb{P}\left[\mathcal{C}_{i,j,l}(u) \big| \mathcal{F}_{t-1}\right] \le e^{-u}$. Combining this with the tower rule of the conditional expectation further yields,
	\begin{align*}
		\mathbb{P}\left[\mathcal{C}_{t,i,j,l}(u)\right] = \mathbb{E} \left[\indicator\{\mathcal{C}_{t,i,j,l}(u)\}\right] = \mathbb{E} \left[\mathbb{E}\left[\indicator\{\mathcal{C}_{t,i,j,l}(u)\} |\mathcal{F}_{t-1} \right]\right] \le e^{-u}. 
	\end{align*}
	The rest of the proofs conditioned on the event $\mathcal{B}(u):=\left\{\cup_{t\in [T], i\in [d], j\in [d], l\in \{0\}\cup [\halfk]} \mathcal{C}_{t,i,j,l}(u)\right\}^{c}$, which satisfies
	\begin{align*}
		\mathbb{P}\left[\{\mathcal{B}(u)\}^c\right] \le Td^2k \cdot e^{-u}
	\end{align*} by union bound. Plugging the upper bound for $Z_l(V)$ in $\mathcal{B}(u)$ into the identity \eqref{eq:eij} gives
	\begin{align*}
		\left|E_{i,j}^{(t)} \right| &\le \frac{1}{2\|\Wt\|_F^{\halfk-2}} \sum_{l=1}^{\halfk} |Z_l(V)| + \frac{(k-4)}{2\|\Wt\|_F^{\halfk}} |Z_0(V)| |W^{(t)}_{i,j}| \\
		&\le 2 \sigma \left\{\frac{k}{4} \frac{\|\Wt\|_F^{\halfk-1}}{\|\Wt\|_F^{\halfk-2}} + \frac{k-4}{2} \frac{\|\Wt\|_F^{\halfk}}{\|\Wt\|_F^{\halfk}} |W_{i,j}^{(t)}| \right\} \sqrt{\log (1/u)}\\
		&\le \sigma k \left\{\|\Wt\|_F + |W_{i,j}^{(t)}|\right\} \sqrt{u}.
	\end{align*} It then concludes the proof by choosing $u=\log(Td^2k/\delta)$. Similarly, for any fixed matrix $Q\in\bbR^{d\times d}$, the following hold
	\begin{align*}
		\left|\llangle E^{(t)},Q\rrangle\right|&\le \frac{1}{2\|\Wt\|_F^{\halfk-2}} \sum_{l=1}^{\halfk} |Z_l(Q)| + \frac{(k-4)}{2\|\Wt\|_F^{\halfk}} |Z_0(Q)| \left|\llangle W^{(t)},Q\rrangle\right| \\
		&\le 2 \sigma \left\{\frac{k}{4} \frac{\|\Wt\|_F^{\halfk-1}\|Q\|_F}{\|\Wt\|_F^{\halfk-2}} + \frac{k-4}{2} \frac{\|\Wt\|_F^{\halfk}}{\|\Wt\|_F^{\halfk}} \left|\left\langle W^{(t)},Q\right\rangle\right| \right\} \sqrt{u}
	\end{align*}
\end{proof}

\begin{proof}[Proof of Lemma \ref{control-expectation}]
	For any $t\in[1:T]$, recall that event $\mathcal{A}_{1}^{(t)}(\delta)$ has the form of
	\begin{align*}
		\mathcal{A}_{1}^{(t)}(\delta) = &\left\{ \forall i,j\in [d], ~~\left| E^{(t)}_{i,j} \right| \le \sqrt{2\mathsf{c}_1} \left\{\|W^{(t)}\|_F + |W_{i,j}^{(t)}| \right\} \right\},
	\end{align*}
	where $\cst_1 = \sigma k\log^{\frac{1}{2}} (kTd^2/\delta)$. Moreover, event $\calA_{2}^{(t)}(\delta)$ can be decomposed into $\calA_{2}^{(t)}(\delta)=\calA_{2,1}^{(t)}(\delta)\cup\calA_{2,2}^{(t)}(\delta)$, where 
	\begin{align}
		&\calA_{2,1}^{(t)}(\delta):=\left\{\left|\llangle E^{(t)},v_*v_*^{\top}\rrangle\right|\leq\sqrt{\mathsf{c}_1}\left(\|W^{(t)}\|_F+\left|\llangle v_*v_*^{\top},W^{(t)}\rrangle\right|\right)\right\},\notag
		\\
		&\calA_{2,2}^{(t)}(\delta):=\left\{\left|\llangle E^{(t)},W^{(t)}\rrangle\right|\leq\sqrt{2\mathsf{c}_1}\|W^{(t)}\|_F^2\right\}.\notag
	\end{align}
	According to the convexity of $e^x$ and Assumption \ref{ass-2}, one can notice that $\frac{1}{\left\|W^{(t)}\right\|_{\tF}\left\|Q\right\|_{\tF}}\llangle E^{(t)},Q\rrangle$ (defined in Eq.~\eqref{eq:errort}) is sub-Gaussian with parameter $2k\sigma\left(1+\frac{\left|\llangle W^{(t)},Q\rrangle\right|}{\left\|W^{(t)}\right\|_{\tF}\left\|Q\right\|_{\tF}}\right)$ for any fixed matrix $Q\in\bbR^{d\times d}$. Therefore, we have 
	\begin{align}\label{tech-1}
		&\left|\bbE_t\left[\frac{1}{\left\|W^{(t)}\right\|_{\tF}}\llangle E^{(t+1)}\cdot\mathds{1}_{\calA_{1}^{(t+1)}(\delta)\cap\calA_{2}^{(t+1)}(\delta)},v_*v_*^{\top}\rrangle\right]\right|\notag
		\\
		=&\left|\bbE_t\left[\frac{1}{\left\|W^{(t)}\right\|_{\tF}}\llangle E^{(t+1)}\cdot\mathds{1}_{\calA_{1}^{(t+1)}(\delta)\cap\calA_{2}^{(t+1)}(\delta)},v_*v_*^{\top}\rrangle\right]-\bbE_t\left[\frac{1}{\left\|W^{(t)}\right\|_{\tF}}\llangle E^{(t+1)},v_*v_*^{\top}\rrangle\right]\right|\notag
		\\
		\leq&\underbrace{\left|\bbE_t\left[\frac{1}{\left\|W^{(t)}\right\|_{\tF}}\llangle E^{(t+1)}\cdot\mathds{1}_{\calA_{2,1}^{(t+1)}(\delta)},v_*v_*^{\top}\rrangle\right]-\bbE_t\left[\frac{1}{\left\|W^{(t)}\right\|_{\tF}}\llangle E^{(t+1)},v_*v_*^{\top}\rrangle\right]\right|}_{\calI_{2,t}}\notag
		\\
		&+\underbrace{\left|\bbE_t\left[\frac{1}{\left\|W^{(t)}\right\|_{\tF}}\llangle E^{(t+1)}\cdot\mathds{1}_{\left(\calA_{1}^{(t+1)}(\delta)\right)^{c}\cap\calA_{2,1}^{(t+1)}(\delta)},v_*v_*^{\top}\rrangle\right]\right|}_{\calII_{2,t}}\notag
		\\
		&+\underbrace{\left|\bbE_t\left[\frac{1}{\left\|W^{(t)}\right\|_{\tF}}\llangle E^{(t+1)}\cdot\mathds{1}_{\left(\calA_{2,2}^{(t+1)}(\delta)\right)^c\cap\calA_{1}^{(t+1)}(\delta)\cap\calA_{2,1}^{(t+1)}(\delta)},v_*v_*^{\top}\rrangle\right]\right|}_{\calIII_{2,t}}.
	\end{align}
	Based on Lemma \ref{aux-truncation-subGaussian-covariance}, $\mathsf{c}_1\gtrsim k\sigma\log^{\frac{1}{2}}(kd^2T/\delta)$ yields $\calI_{2,t}\leq\frac{\sqrt{\delta}}{3}$ when $T$ is sufficiently large. Utilizing Cauchy-Schwartz inequality and Lemma \ref{aux-truncation-subGaussian-second-moment}, we obtain
	\begin{align}\label{tech-2}
		\calII_{2,t}\leq\frac{1}{\left\|W^{(t)}\right\|_{\tF}}\left(\bbE_t\left[\left\|E^{(t+1)}\right\|_{\tF}^2\cdot\mathds{1}_{\calA_{1}^{(t+1)}(\delta)}\right]\right)^{\frac{1}{2}}\lesssim dk\sigma\delta^{\frac{1}{4}}.
	\end{align}
	Finally, $\calIII_{2,t}$ satisfies
	\begin{align}\label{tech-3}
		\calIII_{2,t}\leq&\sup_{E^{(t+1)}}\frac{1}{\left\|W^{(t)}\right\|_{\tF}}\left|\llangle E^{(t+1)}\mathds{1}_{\calA_{1}^{(t+1)}(\delta)\cap\calA_{2,1}^{(t+1)}(\delta)},v_*v_*^{\top}\rrangle\right|\cdot\left[1-\bbP\left(\calA_{2,2}^{(t+1)}(\delta)\right)\right]\notag
		\\
		\overset{\text{(a)}}{\lesssim}&\sqrt{\mathsf{c}_1}d\cdot\left[1-\bbP\left(\calA_{2,2}^{(t+1)}(\delta)\right)\right]\overset{\text{(b)}}{\lesssim}\frac{\sqrt{\mathsf{c}_1}d}{\text{poly}\left(\frac{kd^2T}{\delta}\right)},
	\end{align}
	where (a) is derived from Cauchy-Schwartz inequality and (b) follows from Proposition \ref{prop-A5}. Therefore, if $\delta\lesssim\frac{\tau^4}{\sigma^4k^4d^4}$, combining Eqs.~\eqref{tech-1}-\eqref{tech-3} can directly derive Eq.~\eqref{tech-lemma-2-1}. By employing a similar proof method, Eq.~\eqref{tech-lemma-2-2} can be further derived.
\end{proof}

\subsection{Auxiliary Lemma}
\begin{definition}[Sub-Gaussian Random Variable]\label{sub-Gaussian}
	A random variable $X$ with mean $\bbE X$ is sub-Gaussian if there is $\sigma\in\bbR_+$ such that
	\begin{align}
		\bbE\left[e^{\lambda(X-\bbE X)}\right]\leq e^{\frac{\lambda^2\sigma^2}{2}},\quad \forall\lambda\in\bbR.\notag
	\end{align}
\end{definition}

\begin{proposition}\label{prop-A5}[\citep{wainwright2019high}]
	For a random variable $X$ which satisfies the sub-Gaussian condition \ref{sub-Gaussian} with parameter $\sigma$, we have
	\begin{align}\label{eq-subGaussian-prob}
		\bbP\left(|X-\bbE X|>c\right)\leq 2e^{-\frac{c^2}{2\sigma^2}},\quad \forall c>0.
	\end{align}
\end{proposition}

\begin{lemma}\label{aux-truncation-subGaussian-second-moment}
	Consider a random variable $X$ which is zero-mean and sub-Gaussian with parameter $\sigma$ for some $\sigma>0$. Then, for any $\tau\in(0,1)$, there exists $R>0$ which depends on $\sigma$ and $\tau$ such that
	\begin{align}
		\bbE\left[X^2\mathds{1}_{|X|\leq R}\right]\geq(1-\tau)\bbE\left[X^2\right],\quad \bbE\left[X^4\mathds{1}_{|X|>R}\right]\leq\tau.\notag
	\end{align}
\end{lemma}
\begin{proof}
	According to Eq.~\eqref{eq-subGaussian-prob}, we have $\bbP(|X|\geq r)\leq 2e^{-\frac{r^2}{2\sigma^2}}$ for any $r>0$. Therefore, we obtain
	\begin{align}\label{estimate-second-moment}
		\bbE\left[X^2\mathds{1}_{|X|>R}\right]\overset{\text{(a)}}{=}&2\int_{0}^{\infty}r\bbP(|X|\mathds{1}_{|X|>R}>r)\text{d}r\notag
		\\
		=&2\int_R^{\infty}r\bbP(|X|>r)\text{d}r+R^2\bbP(|X|>R)\notag
		\\
		\leq&4\int_R^{\infty}re^{-\frac{r^2}{2\sigma^2}}\text{d}r+2R^2e^{-\frac{R^2}{2\sigma^2}}=(4\sigma^2+2R^2)e^{-\frac{R^2}{2\sigma^2}},
	\end{align}
	where (a) is derived from [Lemma 2.2.13, \citep{wainwright2019high}]. Moreover, we have
	\begin{align}\label{original-estimate-forth-moment}
		\bbE\left[X^4\mathds{1}_{|X|>R}\right]=&4\int_{0}^{\infty}r^3\bbP\left(|X|\mathds{1}_{|X|>R}>r\right)\text{d}r\notag
		\\
		=&4\int_{R}^{\infty}r^3\bbP\left(|X|>r\right)\text{d}r+R^4\bbP\left(|X|>R\right)\notag
		\\
		\leq&8\int_{R}^{\infty}r^3e^{-\frac{r^2}{2\sigma ^2}}\text{d}r+2R^4e^{-\frac{R^2}{2\sigma ^2}}\leq(4\sigma^2+2R^2)^2e^{-\frac{R^2}{2\sigma ^2}},
	\end{align}
	Therefore, we only need to set $R$ as:
	\begin{align}
	R= 2\sqrt{2}\sigma\log^{1/2}\left(\frac{4\sigma^2+2K}{\tau\min\{\bbE\left[X^2\right],1\}}\right),\notag
	\end{align}
	where $K$ satisfies $K\geq 8\sigma^2\log\left((4\sigma^2+2K)/(\tau\min\{\bbE[X^2],1\})\right)$.
\end{proof}
\begin{lemma}\label{aux-truncation-subGaussian-covariance}
	Suppose zero-mean random variables $X$ and $Y$ are sub-Gaussian with parameters $\sigma_1$ and $\sigma_2$, respectively.
		Then, for any $\tau\in(0,1)$, there exists $R_1,R_2\geq0$ such that
		$$
		\left|\bbE\left[XY\mathds{1}_{\{|X|\leq R_1\}\bigcap\{|Y|\leq R_2\}}\right]-\bbE[XY]\right|\leq\tau.
		$$ 
\end{lemma}
\begin{proof}
	According to $\bbE\left[XY\mathds{1}_{\{|X|\leq R_1\}\bigcap\{|Y|\leq R_2\}}\right]=\bbE\left[X(1-\mathds{1}_{|X|>R_1})Y(1-\mathds{1}_{|Y|>R_2})\right]$, we have
	\begin{align}
		\bbE\left[XY\mathds{1}_{\{|X|\leq R_1\}\bigcap\{|Y|\leq R_2\}}\right]-\bbE[XY]=&-\bbE\left[XY\mathds{1}_{|X|>R_1}\right]-\bbE\left[XY\mathds{1}_{|Y|>R_2}\right]\notag
		\\
		&+\bbE\left[XY\mathds{1}_{|X|>R_1}\mathds{1}_{|Y|>R_2}\right].\notag
	\end{align}
	For $\bbE\left[XY\mathds{1}_{|X|>R_1}\right]$, we can obtain
	\begin{align}\label{eq:EXY-X}
		\left|\bbE\left[XY\mathds{1}_{|X|>R_1}\right]\right|\overset{\text{(a)}}{\leq}&\left(\bbE\left[X^2\mathds{1}_{|X|>R_1}\right]\right)^{1/2}\left(\bbE\left[Y^2\right]\right)^{1/2}\notag
		\\
		\overset{\text{(b)}}{\leq}&2(\sigma_1^2+R_1^2)^{1/2}e^{-\frac{R_1^2}{4\sigma_1^2}}\left(\bbE\left[Y^2\right]\right)^{1/2}
		\overset{\text{(c)}}{\leq}4\sigma_2(\sigma_1^2+R_1^2)^{1/2}e^{-\frac{R_1^2}{4\sigma_1^2}},
	\end{align}
	where (a) follows from the Cauchy-Schwarz inequality; (b) is derived from the inequality Eq.~\eqref{estimate-second-moment} in the proof of Lemma \ref{aux-truncation-subGaussian-second-moment}; (c) is established through the repeated application of the proof of Lemma \ref{aux-truncation-subGaussian-second-moment}. Similarly, it can be derived that
	\begin{align}
		\left|\bbE\left[XY\mathds{1}_{|Y|>R_2}\right]\right|\leq4\sigma_1(\sigma_2^2+R_2^2)^{1/2}e^{-\frac{R_2^2}{4\sigma_2^2}}.\notag
	\end{align}
	Finally, for $\bbE\left[XY\mathds{1}_{|X|>R_1}\mathds{1}_{|Y|>R_2}\right]$, we have
	\begin{align}
		\left|\bbE\left[XY\mathds{1}_{|X|>R_1}\mathds{1}_{|Y|>R_2}\right]\right|\leq&\left(\bbE\left[X^2\mathds{1}_{|X|>R_1}Y\right]\right)^{1/2}\left(\bbE\left[Y^2\mathds{1}_{|Y|>R_2}Y\right]\right)^{1/2}\notag
		\\
		\leq&4(\sigma_1^2+R_1^2)^{1/2}(\sigma_2^2+R_2^2)^{1/2}e^{-\left(\frac{R_1^2}{4\sigma_1^2}+\frac{R_2^2}{4\sigma_2^2}\right)}.\notag
	\end{align}
	Therefore, we only need to set $R_1$ and $R_2$ as:
	\begin{align}\label{frac-estimator}
		R_1=\sqrt{2}\sigma_1\log^{1/2}\left(\frac{256\max\{\sigma_2^2,1\}(\sigma_1^2+K)}{\tau^2}\right),\quad R_2=\sqrt{2}\sigma_2\log^{1/2}\left(\frac{256\max\{\sigma_1^2,1\}(\sigma_2^2+K)}{\tau^2}\right),
	\end{align}
	where $K$ satisfies $K\geq2(\sigma_1^2+\sigma_2^2)\log\left(256\max\{\sigma_1^2,\sigma_2^2,1\}(\sigma_1^2+\sigma_2^2+K)\tau^{-2}\right)$.
\end{proof}
\begin{lemma}\label{main-estimate-forth-moment}
	Suppose zero-mean random variables $\{X_i\}_{i=1}^4$ are sub-Gaussian with parameters $\{\sigma_i\}_{i=1}^4$, respectively. Then, for any $\tau\in(0,1)$, there exists positive constants $\{R_i\}_{i=1}^4$ such that
	\begin{align}
		\left|\bbE\left[\prod_{i=1}^{4}X_i\mathds{1}_{|X_i|\leq R_i}\right]-\bbE\left[\prod_{i=1}^{4}X_i\right]\right|\leq\tau.\notag
	\end{align}
\end{lemma}
\begin{proof}
	According to
	\begin{align}
		\bbE\left[\prod_{i=1}^{4}X_i\mathds{1}_{|X_i|\leq R_i}\right]=\bbE\left[\prod_{i=1}^4 X_i\left(1-\mathds{1}_{|X_i|>R_i}\right)\right],\notag
	\end{align}
	we have
	\begin{align}\label{decomp-prod}
		&\bbE\left[\prod_{i=1}^{4}X_i\mathds{1}_{|X_i|\leq R_i}\right]-\bbE\left[\prod_{i=1}^{4}X_i\right]\notag
		\\
		=&-\sum_{i=1}^4\bbE\left[X_i\mathds{1}_{|X_i|>R_i}\prod_{j\neq i}X_j\right]+\sum_{1\leq i<j\leq4}\bbE\left[X_iX_j\prod_{k\neq i,j}X_k\mathds{1}_{|X_k|> R_k}\right]\notag
		\\
		&-\sum_{i=1}^4\bbE\left[X_i\prod_{j\neq i}X_j\mathds{1}_{|X_j|> R_j}\right]+\bbE\left[\prod_{i=1}^4X_i\mathds{1}_{|X_i|>R_i}\right].
	\end{align}
	For any random variables $\{Y_i\}_{i=1}^4$, one can notice that
	\begin{align}\label{estimate-forth-moment}
		\left|\bbE\left[\prod_{i=1}^4Y_i\right]\right|\overset{\text{(a)}}{\leq}\left(\bbE\left[Y_1^2Y_2^2\right]\right)^{1/2}\left(\bbE\left[Y_3^2Y_4^2\right]\right)^{1/2}\leq\prod_{i=1}^4\left(\bbE\left[Y_i^4\right]\right)^{1/4}.
	\end{align}
	Applying Eq.~\eqref{estimate-forth-moment} to Eq.~\eqref{decomp-prod}, we obtain
	\begin{align}
		&\left|\bbE\left[\prod_{i=1}^{4}X_i\mathds{1}_{|X_i|\leq R_i}\right]-\bbE\left[\prod_{i=1}^{4}X_i\right]\right|\notag
		\\
		\leq&\sum_{i=1}^4\left(\bbE\left[X_i^4\mathds{1}_{|X_i|>R_i}\right]\right)^{1/4}\prod_{j\neq i}\left(\bbE\left[X_j^4\right]\right)^{1/4}+\sum_{1\leq i<j\leq4}\prod_{l=i,j}\left(\bbE\left[X_l^4\right]\right)^{1/4}\prod_{k\neq i,j}\left(\bbE\left[X_k^4\mathds{1}_{|X_k|>R_k}\right]\right)^{1/4}\notag
		\\
		&+\sum_{i=1}^4\left(\bbE\left[X_i^4\right]\right)^{1/4}\prod_{j\neq i}\left(\bbE\left[X_j^4\mathds{1}_{|X_j|>R_j}\right]\right)^{1/4}+\prod_{i=1}^4\left(\bbE\left[X_i^4\mathds{1}_{|X_i|>R_i}\right]\right)^{1/4}\notag
		\\
		\overset{\text{(a)}}{\leq}&16\left(\prod_{i=1}^4\sigma_i\right)\left[\sum_{i=1}^4\left(1+\frac{R_i^2}{\sigma_i^2}\right)^{1/2}e^{-\frac{R_i^2}{8\sigma_i^2}}+\sum_{1\leq i<j\leq4}\left(\prod_{k\neq i,j}\left(1+\frac{R_k^2}{\sigma_k^2}\right)^{1/2}e^{-\frac{R_k^2}{8\sigma_k^2}}\right)\right.\notag
		\\
		&\, \, \, \, \, \, \, \, \, \, \, \, \, \, \, \, \, \, \, \, \, \, \, \, \, \, \, \, \, \, \, \, \, \, \, \left.+\sum_{i=1}^4\left(\prod_{j\neq i}\left(1+\frac{R_j^2}{\sigma_j^2}\right)^{1/2}e^{-\frac{R_j^2}{8\sigma_j^2}}\right)+\prod_{i=1}^4\left(1+\frac{R_i^2}{\sigma_i^2}\right)^{1/2}e^{-\frac{R_i^2}{8\sigma_i^2}}\right].\notag
	\end{align}
	where (a) is derived from Eq.~\eqref{original-estimate-forth-moment}. Therefore, we only need to set $\{R_i\}_{i=1}^4$ as:
	\begin{align}
		R_i=2\sqrt{2}\sigma_i\log^{1/2}\left(\frac{4196\prod_{j\neq i}\max\{\sigma_j^2,1\}(\sigma_i^2+K_i)}{\tau}\right),\quad \forall i\in[1:4],\notag
	\end{align}
	where each $K_i$ satisfies $K_i\geq8\sigma_i^2\log\left(4196\tau^{-1}(\sigma_i^2+K_i)\prod_{j\neq i}\max\{\sigma_j^2,1\}\right)$.
\end{proof}
\begin{lemma}\label{aux-martingale-concentration}
	Let $c>0$, $\gamma<1$ and $a_t>0$ for any $t\in[0:T-1]$. Consider a sequence of random variables $\{v^t\}_{t=0}^{T-1}\subset[0,2c]$, which satisfies 
	$\bbE[e^{\lambda(v^{t+1}-(1-\eta_t)v^t)}\mid\calF_t]\leq e^{\frac{\lambda^2a_t^2}{2}}$ almost surely for any $\lambda\in\bbR_+$ with stepsize $\eta_t\geq0$. Then, there is 
    \begin{align}
        \bbP\left(v^T>c\bigwedge v^0\leq\gamma c\right)\leq\exp\left\{-\frac{(1-\gamma)^2c^2}{2\sum_{j=0}^{T-1}a_j^2\prod_{i=j+1}^{t-1}(1-\eta_i)^{2}}\right\}.\notag
    \end{align}
\end{lemma}
\begin{proof}
    We define $D_{t+1}:=\prod_{i=0}^{t}(1-\eta_i)^{-1}\tilde{v}^{t+1}-\prod_{i=0}^{t-1}(1-\eta_i)^{-1}\tilde{v}^t$ for any $t\in[0:T-1]$. Therefore, applying iterated expectation yields 
    \begin{align}\label{eq-sub-gaussian}
        \bbE\left[e^{\lambda\left(\sum_{i=1}^tD_i\right)}\right]=&\bbE\left[e^{\lambda\left(\sum_{i=1}^{t-1}D_{i}\right)}\bbE\left[\left.e^{\lambda D_t}\right|\calF_{t-1}\right]\right]\notag
        \\
        =&\bbE\left[e^{\lambda\left(\sum_{i=1}^{t-1}D_{i}\right)}\bbE\left[\left.e^{\frac{\lambda}{\prod_{i=0}^{t-1}(1-\eta_i)} \left(v^{t}-(1-\eta_{t-1})v^{t-1}\right)}\right|\calF_{t-1}\right]\right]
        \notag
        \\
        \overset{\text{(a)}}{\leq}&e^{\frac{\lambda^2a_{t-1}^2}{2\prod_{i=0}^{t-1}(1-\eta_i)^2}}\bbE\left[e^{\lambda\left(\sum_{i=1}^{t-1}D_{i}\right)}\right]
        \notag
        \\
        \leq&e^{\frac{\lambda^2\sum_{j=0}^{t-1}a_j^2\prod_{i=0}^j(1-\eta_i)^{-2}}{2}},
    \end{align}
    for any $\lambda\in\bbR^+$ and $t\in[1:T]$, where 
    (a) follows from the condition that $\bbE[e^{\lambda(v^{t}-(1-\eta_{t-1})v^{t-1})}\mid\calF_{t-1}]\leq e^{\frac{\lambda^2a_{t-1}^2}{2}}$ almost surely for any $\lambda\in\bbR_+$. Then we obtain 
    \begin{align}
        \bbP\left(v^T>c\bigwedge v^0\leq\gamma c\right){\leq}&\bbP\left(\prod_{i=0}^{T-1}(1-\eta_i)^{-1}v^{t}>\prod_{i=0}^{T-1}(1-\eta_i)^{-1}c\bigwedge v^0\leq\gamma c\right)\notag
        \\
        \leq&\min_{\lambda>0}\frac{\bbE\left[e^{\lambda(\sum_{i=1}^TD_i)}\right]}{e^{\lambda\left(\prod_{i=0}^{T-1}(1-\eta_i)^{-1}c-\gamma c\right)}}
        \notag
        \\
        \overset{\text{(b)}}{\leq}&\exp\left\{-\frac{\left(\prod_{i=0}^{T-1}(1-\eta_i)^{-1}c-\gamma c\right)^2}{2\sum_{j=0}^{T-1}a_j^2\prod_{i=0}^{j}(1-\eta_i)^{-2}}\right\}\notag
        \\
        \leq&\exp\left\{-\frac{(1-\gamma)^2\left(\prod_{i=0}^{T-1}(1-\eta_i)^{-1}c\right)^2}{2\sum_{j=0}^{T-1}a_j^2\prod_{i=0}^{j}(1-\eta_i)^{-2}}\right\}\notag
        \\
        =&\exp\left\{-\frac{(1-\gamma)^2c^2}{2\sum_{j=0}^{T-1}a_j^2\prod_{i=j+1}^{T-1}(1-\eta_i)^{2}}\right\},
    \end{align}
    where (b) is derived from Eq.~\eqref{eq-sub-gaussian}.
\end{proof}
\begin{lemma}\label{aux-martingale-concentration-subtraction}
	Let $c>\gamma>0$ and $a_t>0$ for any $t\in[0:T-1]$. Consider a sequence of random variables $\{v^t\}_{t=0}^{T-1}\subset[0,c]$, which satisfies 
	$\bbE[e^{\lambda(v^{t+1}+\eta_t-v^t)}\mid\calF_t]\leq e^{\frac{\lambda^2a_t^2}{2}}$ almost surely for any $\lambda\in\bbR_+$ with stepsize $\eta_t\geq0$. Suppose $\sum_{t=0}^{T-1}\eta_t>v^0$, there is
	\begin{align}
		\bbP\left(v^T>\gamma\right)\leq\exp\left\{-\frac{\left(\gamma+\sum_{t=0}^{T-1}\eta_t-v^0\right)^2}{2\sum_{j=0}^{T-1}a_j^2}\right\}.\notag
	\end{align}
\end{lemma}
\begin{proof}
	We define $D_{t+1}:=v^{t+1}+\sum_{i=0}^t\eta_i-\left(v^t+\sum_{i=0}^{t-1}\eta_i\right)$ for any $t\in[0:T-1]$. Therefore, applying iterated expectation yields
	\begin{align}\label{eq-sub-gaussian-substraction}
		\bbE\left[e^{\lambda\left(\sum_{i=1}^{t}D_i\right)}\right]=&\bbE\left[e^{\lambda\left(\sum_{i=1}^{t-1}D_i\right)}\bbE\left[\left.e^{\lambda D_t}\right|\calF_{t-1}\right]\right]\notag
		\\
		=&\bbE\left[e^{\lambda\left(\sum_{i=1}^{t-1}D_{i}\right)}\bbE\left[\left.e^{\lambda(v^t+\eta_{t-1}-v^{t-1})}\right|\calF_{t-1}\right]\right]
		\notag
		\\
		\overset{\text{(a)}}{\leq}&e^{\frac{\lambda^2a_{t-1}^2}{2}}\bbE\left[e^{\lambda\left(\sum_{i=1}^{t-1}D_{i}\right)}\right]
		\notag
		\\
		\leq&e^{\frac{\lambda^2\sum_{j=0}^{t-1}a_j^2}{2}},
	\end{align}
	for any $\lambda\in\bbR^+$ and $t\in[1:T]$, where 
	(a) follows from the condition that $\bbE[e^{\lambda(v^{t}+\eta_{t-1}-v^{t-1})}\mid\calF_{t-1}]\leq e^{\frac{\lambda^2a_{t-1}^2}{2}}$ almost surely for any $\lambda\in\bbR_+$. Then we obtain 
	\begin{align}
		\bbP\left(v^T>\gamma\right)\leq&\min_{\lambda>0}\frac{\bbE\left[e^{\lambda(\sum_{i=1}^TD_i)}\right]}{e^{\lambda\left(\gamma+\sum_{t=0}^{T-1}\eta_t-v^0\right)}}
		\notag
		\\
		\overset{\text{(b)}}{\leq}&\exp\left\{-\frac{\left(\gamma+\sum_{t=0}^{T-1}\eta_t-v^0\right)^2}{2\sum_{j=0}^{T-1}a_j^2}\right\},
	\end{align}
	where (b) is derived from Eq.~\eqref{eq-sub-gaussian-substraction}.
\end{proof}
\begin{corollary}\label{aux-coro-3}
    Let $c>0$, $\gamma<1$ and $a_t>0$ for any $t\in[0:T-1]$. Consider a sequence of random variables $\{v^i\}_{i=0}^{T-1}\subset[0,c]$, which satisfies $\prod_{i=0}^{T-1}(1+\eta_t)^{-1}c-v^0\geq\gamma c$ with stepsize $\eta_t\geq0$, given $\bbE[e^{\lambda(v^{t+1}-(1+\eta_t)v^t)}\mid\calF_t]\leq e^{\frac{\lambda^2a_t^2}{2}}$ almost surely for any $\lambda\in\bbR_+$. Then, there is 
    \begin{align}
        \bbP\left(v^T>c\right)\leq\exp\left\{-\frac{\gamma^2c^2}{2\sum_{j=0}^{T-1}a_j^2\prod_{i=0}^{j}(1+\eta_i)^{-2}}\right\}\notag.
    \end{align}
\end{corollary}

\begin{corollary}\label{aux-coro-4}
	Let $\gamma<1$ and $a_t>0$ for any $t\in[0:T-1]$. Consider a sequence of random variables $\{v^{(i)}\}_{i=0}^{T-1}\subset[0,v^{(0)}]$, which satisfies $\bbE[e^{\lambda(v^{(t+1)}-(1+\eta_t)v^{(t)})}\mid\calF^{(t)}]\leq e^{\frac{\lambda^2a_t^2}{2}}$  almost surely for any $\lambda\in\bbR_-$ with stepsize $\eta_t\geq0$. Then, there is 
	\begin{align}
		\bbP\left(v^{(T)}<\gamma v^{(0)}\right)\leq\exp\left\{-\frac{(1-\gamma)^2(v^{(0)})^2}{2\sum_{j=0}^{T-1}a_j^2\prod_{i=0}^{j}(1+\eta_i)^{-2}}\right\}\notag.
	\end{align}
\end{corollary}

\begin{lemma}\label{aux-3}
    For $L,K\in\bbN_+$, consider $T\in\bbN^+$ such that $LK\leq T<(L+1)K$. Then we have 
    \begin{align}
        \sum_{t=0}^{T}\left(\prod_{i=t}^{T}(1-c\eta_t)\right)\eta_t^2\leq\frac{2\eta_0}{c},
    \end{align}
    where $\eta_t=\frac{\eta_0}{2^l}$ if $lK\leq t\leq \min\{(l+1)K-1,T\}$ for any $l\in[0:L]$ and $c>0$ is a constant.
\end{lemma}
\begin{proof}
    For any $l\in[0:L]$, we have
    \begin{align}\label{aux-eq-lemma3}
        \sum_{t=lK}^{(l+1)K-1}\left(\prod_{i=t}^{T}\left(1-c\eta_t\right)\right)\eta_t^2=&\eta_{lK}^2\left(\prod_{i=(l+1)K}^{T}\left(1-c\eta_t\right)\right)\sum_{t=lK}^{(l+1)K-1}(1-c\eta_{lK})^{(l+1)K-1-t}\notag
        \\
        \leq&\frac{\eta_{lK}}{c}\left(\prod_{i=(l+1)K}^{T}\left(1-c\eta_t\right)\right).
    \end{align}
    Therefore, we obtain the following estimation
    \begin{align}
        \sum_{t=0}^{T}\left(\prod_{i=t}^{T}(1-c\eta_t)\right)\eta_t^2\leq&\sum_{t=0}^{LK-1}\left(\prod_{i=t}^{T}(1-c\eta_t)\right)\eta_t^2+\sum_{t=LK}^T(1-c\eta_{LK})^{T-t}\eta_{LK}^2\notag
        \\
        \overset{\text{(a)}}{\leq}&\frac{\sum_{l=0}^L\eta_{lK}}{c}\leq\frac{2\eta_0}{c}.
    \end{align}
\end{proof}
\begin{lemma}\label{lemma-aux-frac}
	Consider vector $v\in\bbR^d$ and matrix $W,Q\in\bbR^{d\times d}$. Define the function $f:\bbR_+\rightarrow\bbR$ as $f(\eta):=\frac{\llangle v,Wv\rrangle+\eta\llangle v,Qv\rrangle}{\|W+\eta Q\|_{\tF}}$. The following equality holds:
	\begin{align}\label{Taylor-eq1}
		f(\eta)=\frac{\llangle v,Wv\rrangle}{\|W\|_{\tF}}+\eta\left(\frac{\llangle v,Qv\rrangle}{\|W\|_{\tF}}-\frac{\llangle v,Wv\rrangle\llangle W,Q\rrangle}{\|W\|_{\tF}^3}\right)+\frac{\eta^2}{2}\left(-2\calI(\tau\eta)-\calII(\tau\eta)+3\calIII(\tau\eta)\right),
	\end{align}
	where $\tau\in[0,1]$ depends on $\eta$, $v$, $Q$ and $W$. $\calI(x)$, $\calII(x)$ and $\calIII(x)$ have the following definitions for any $x\in\bbR_+$:
	\begin{align}
		\calI(x):=&\frac{\llangle v,Qv\rrangle\cdot\left(\llangle W,Q\rrangle+x\|Q\|_{\tF}^2\right)}{\|W+x Q\|_{\tF}^3},\notag
		\\
		\calII(x):=&\frac{\left(\llangle v,Wv\rrangle+x\llangle v,Qv\rrangle\right)\cdot\|Q\|_{\tF}^2}{\|W+xQ\|_{\tF}^3},\notag
		\\
		\calIII(x):=&\frac{\left(\llangle v,Wv\rrangle+x\llangle v,Qv\rrangle\right)\cdot\left(\llangle W,Q\rrangle+x\|Q\|_{\tF}^2\right)^2}{\|W+xQ\|_{\tF}^5}.\notag
	\end{align}
	Moreover, define another function $g:\bbR_+\rightarrow\bbR$ as $g(\eta):=\left(f(\eta)\right)^{-(k/2-2)}$ for any $k>4$. Then, we obtain the following equality:
	\begin{align}
		g(\eta)=&\left(\frac{\llangle v,Wv\rrangle}{\|W\|_{\tF}}\right)^{-(k/2-2)}-\frac{\eta(k-4)}{2}\left(\frac{\llangle v,Wv\rrangle}{\|W\|_{\tF}}\right)^{-(k/2-1)}\cdot\left(\frac{\llangle v,Qv\rrangle}{\|W\|_{\tF}}-\frac{\llangle v,Wv\rrangle\llangle W,Q\rrangle}{\|W\|_{\tF}^3}\right)\notag
		\\
		&+\frac{\eta^2(k-4)(k-2)}{8}\left(\frac{\llangle v,(W+\gamma\eta Q)v\rrangle}{\|W+\gamma\eta Q\|_{\tF}}\right)^{-k/2}\cdot\left(f'(\gamma\eta)\right)^2\notag
		\\
		&+\frac{\eta^2(k-4)}{4}\left(\frac{\llangle v,(W+\gamma\eta Q)v\rrangle}{\|W+\gamma\eta Q\|_{\tF}}\right)^{-(k/2-1)}\cdot\left(2\calI(\gamma\eta)+\calII(\gamma\eta)-3\calIII(\gamma\eta)\right),\notag
	\end{align}
	where $\gamma\in[0,1]$ depends on $\eta$, $v$, $Q$ and $W$.
\end{lemma}
\begin{proof}
	Observe that the first derivative of $f(\eta)$ is:
	\begin{align}\label{fd}
		f'(\eta)=\frac{\llangle v,Qv\rrangle}{\|W+\eta Q\|_{\tF}}-\frac{\left(\llangle v,Wv\rrangle+\eta\llangle v,Qv\rrangle\right)\cdot\left(\llangle W,Q\rrangle+\eta\|Q\|_{\tF}^2\right)}{\|W+\eta Q\|_{\tF}^3},
	\end{align}
	In addition, one can notice the second derivative of $f(\eta)$ has the following expression:
	\begin{equation}\label{sd}
		\begin{split}
			f''(\eta)=&-2\underbrace{\frac{\llangle v,Qv\rrangle\cdot\left(\llangle W,Q\rrangle+\eta\|Q\|_{\tF}^2\right)}{\|W+\eta Q\|_{\tF}^3}}_{\calI(\eta)}-\underbrace{\frac{\left(\llangle v,Wv\rrangle+\eta\llangle v,Qv\rrangle\right)\cdot\|Q\|_{\tF}^2}{\|W+\eta Q\|_{\tF}^3}}_{\calII(\eta)}
			\\
			&+3\underbrace{\frac{\left(\llangle v,Wv\rrangle+\eta\llangle v,Qv\rrangle\right)\cdot\left(\llangle W,Q\rrangle+\eta\|Q\|_{\tF}^2\right)^2}{\|W+\eta Q\|_{\tF}^5}}_{\calIII(\eta)},
		\end{split}
	\end{equation}
	Therefore, combining Eq.~\eqref{fd} and Eq.~\eqref{sd} with the Taylor expansion of $f(\eta)$, we complete the proof of Eq.~\eqref{Taylor-eq1}.
	\begin{align}
		f(\eta)=\frac{\llangle v,Wv\rrangle}{\|W\|_{\tF}}+\eta\left[\left.f'(x)\right|_{x=0}\right]+\frac{\eta^2}{2}\left[\left.f''(x)\right|_{x=\tau\eta}\right],\notag
	\end{align}
	where $\tau\in[0,1]$ is a scaling parameter dependent on $\eta$, $v$, $Q$ and $W$. Similarly, for function $g$, we can obtain that
	\begin{align}
		g(\eta)=&\left(\frac{\llangle v,Wv\rrangle}{\|W\|_{\tF}}\right)^{-(k/2-2)}-\frac{\eta(k-4)}{2}\left[\left.\left(f(x)\right)^{-(k/2-1)}f'(x)\right|_{x=0}\right]\notag
		\\
		&+\frac{\eta^2(k-4)(k-2)}{8}\left[\left.\left(f(x)\right)^{-k/2}\left(f'(x)\right)^2\right|_{x=\gamma\eta}\right]-\frac{\eta^2(k-4)}{4}\left[\left.\left(f(x)\right)^{-(k/2-1)}f''(x)\right|_{x=\gamma\eta}\right],\notag
	\end{align}
	where $\gamma\in[0,1]$ is also a scaling parameter dependent on $\eta$, $v$, $Q$ and $W$.
\end{proof}

\end{document}